\documentclass[a4paper]{article}
\usepackage{verbatim}                       

\usepackage{graphicx}      
\usepackage{tikz-cd}

\usepackage{listings}
\usepackage[plain]{fancyref}

\usepackage{amsmath, amsthm, stmaryrd}                
\usepackage{thmtools}
\usepackage{amssymb, mathrsfs,bm, bbm}
\usepackage{enumerate}                      

\usepackage{accents, braket,subfig,titlesec,cleveref,lipsum,xcolor-patch, textcomp,float,booktabs,siunitx}
\usepackage[numbers]{natbib}
\usepackage[section]{placeins}
\usepackage{enumitem}
\usepackage[margin=2.16cm]{geometry}
\usepackage[colorinlistoftodos,prependcaption]{todonotes}
\usepackage{tcolorbox}
\usepackage{lmodern}

\newcommand{\olpi}{\overline{\pi\vphantom{*}}}

\newcommand{\N}{\mathbb{N}}
\newcommand{\Z}{\mathbb{Z}}
\newcommand{\R}{\mathbb{R}}
\newcommand{\T}{\mathbb{T}}
\newcommand{\C}{\mathbb{C}}
\DeclareMathOperator{\conv} {conv}
\DeclareMathOperator*{\supp}{supp}
\DeclareMathOperator*{\dom}{dom}
\DeclareMathOperator{\Lie}{Lie}
\DeclareMathOperator{\Spec}{Spec}
\DeclareMathOperator{\Tr}{Tr}
\DeclareMathOperator{\ev}{ev}
\newcommand{\abs}[1]{\lvert #1 \rvert}
\newcommand{\norm}[1]{\lVert #1 \rVert}
\newcommand{\op}{\textrm{op}}

\newcommand{\squotes}[1]{`#1'}
\newcommand{\dquotes}[1]{``#1''}

\makeatletter
\newcommand{\bcdot}{} 
\DeclareRobustCommand\bcdot{%
	\mathord{\mathpalette\bcdot@{0.5}}%
}
\newcommand{\bcdot@}[2]{%
	\vcenter{\hbox{\scalebox{#2}{$\m@th#1\bullet$}}}%
}
\makeatother

\newcommand{\der}{\mathrm{der}}
\newcommand{\Inn}{\mathrm{Inn}}
\newcommand{\mc}{\mathcal}
\newcommand{\mfr}{\mathfrak}
\newcommand{\su}{\mathfrak{su}}
\newcommand{\g}{\mathfrak{g}}
\newcommand{\gl}{\mathfrak{gl}}
\newcommand{\h}{\mathfrak{h}}
\newcommand{\pu}{\mathfrak{pu}}
\newcommand{\n}{\mathfrak{n}}
\newcommand{\z}{\mathfrak{z}}
\newcommand{\M}{\mathcal{M}}
\newcommand{\mcN}{\mathcal{N}}
\newcommand{\A}{\mathcal{A}}
\newcommand{\mcC}{\mathcal{C}}
\newcommand{\B}{\mathcal{B}}
\newcommand{\mcD}{\mathcal{D}}
\newcommand{\F}{\mathcal{F}}
\newcommand{\mcK}{\mathcal{K}}
\newcommand{\X}{\mathcal{X}}
\newcommand{\U}{\mathrm{U}}
\newcommand{\PU}{\mathrm{PU}}
\newcommand{\p}{\mathfrak{p}}

\newcommand{\mcH}{\mathcal{H}}
\newcommand{\mcZ}{\mathcal{Z}}
\newcommand{\KMS}{\textrm{KMS}}
\newcommand{\ct}{\mathrm{ct}}
\newcommand{\St}{\mathrm{St}}
\newcommand{\ad}{\mathrm{ad}}
\newcommand{\Ad}{\mathrm{Ad}}
\newcommand{\restr}[2]{\left. #1\right|_{#2}}
\DeclareMathOperator*{\st}{\, : \,}
\newcommand{\Span}{\mathrm{Span}}
\newcommand{\Mp}{\mathrm{Mp}}
\newcommand{\Sp}{\mathrm{Sp}}
\newcommand{\SU}{\mathrm{SU}}
\newcommand{\SL}{\mathrm{SL}}
\newcommand{\mrm}[1]{\mathrm{#1}}
\newcommand{\Aut}{\mathrm{Aut}}

\newcommand{\id}{\mathrm{id}}
\newcommand{\circled}[1]{\accentset{\circ}{#1}}

\newcommand{\vl}[1]{\bm{v}_{\mrm{l}}({#1})}
\newcommand{\vls}[1]{\bm{v}_{\mrm{l}}({#1})_{\mrm{s}}}
\newcommand{\vln}[1]{\bm{v}_{\mrm{l}}({#1})_{\mrm{n}}}
\newcommand{\vho}[1]{\bm{v}_{\mrm{ho}}({#1})}
\newcommand{\who}[1]{\bm{w}_{\mrm{ho}}({#1})}

\theoremstyle{plain}
\newtheorem{theorem}{Theorem}[section]

\theoremstyle{definition}
\newtheorem{definition}[theorem]{Definition}

\theoremstyle{plain}
\newtheorem{lemma}[theorem]{Lemma}
\newtheorem{proposition}[theorem]{Proposition}
\newtheorem{corollary}[theorem]{Corollary}

\theoremstyle{remark}
\newtheorem{remark}[theorem]{Remark}

\theoremstyle{definition}
\newtheorem{example}[theorem]{Example}

\theoremstyle{problem}

\newcommand*{\fancyrefthmlabelprefix}{thm}
\fancyrefaddcaptions{english}{%
	\providecommand*{\frefthmname}{theorem}%
	\providecommand*{\Frefthmname}{Theorem}%
}
\frefformat{plain}{\fancyrefthmlabelprefix}{\frefthmname\fancyrefdefaultspacing#1}
\Frefformat{plain}{\fancyrefthmlabelprefix}{\Frefthmname\fancyrefdefaultspacing#1}

\newcommand*{\fancyreflemlabelprefix}{lem}
\fancyrefaddcaptions{english}{%
	\providecommand*{\freflemname}{lemma}%
	\providecommand*{\Freflemname}{Lemma}%
}
\frefformat{plain}{\fancyreflemlabelprefix}{\freflemname\fancyrefdefaultspacing#1}
\Frefformat{plain}{\fancyreflemlabelprefix}{\Freflemname\fancyrefdefaultspacing#1}

\newcommand*{\fancyrefproplabelprefix}{prop}
\fancyrefaddcaptions{english}{%
	\providecommand*{\frefpropname}{proposition}%
	\providecommand*{\Frefpropname}{Proposition}%
}
\frefformat{plain}{\fancyrefproplabelprefix}{\frefpropname\fancyrefdefaultspacing#1}
\Frefformat{plain}{\fancyrefproplabelprefix}{\Frefpropname\fancyrefdefaultspacing#1}

\newcommand*{\fancyrefcorlabelprefix}{cor}
\fancyrefaddcaptions{english}{%
	\providecommand*{\frefcorname}{corollary}%
	\providecommand*{\Frefcorname}{Corollary}%
}
\frefformat{plain}{\fancyrefcorlabelprefix}{\frefcorname\fancyrefdefaultspacing#1}
\Frefformat{plain}{\fancyrefcorlabelprefix}{\Frefcorname\fancyrefdefaultspacing#1}

\newcommand*{\fancyrefdeflabelprefix}{def}
\fancyrefaddcaptions{english}{%
	\providecommand*{\frefdefname}{definition}%
	\providecommand*{\Frefdefname}{Definition}%
}
\frefformat{plain}{\fancyrefdeflabelprefix}{\frefdefname\fancyrefdefaultspacing#1}
\Frefformat{plain}{\fancyrefdeflabelprefix}{\Frefdefname\fancyrefdefaultspacing#1}

\newcommand*{\fancyrefasmlabelprefix}{asm}
\fancyrefaddcaptions{english}{%
	\providecommand*{\frefasmname}{assumption}%
	\providecommand*{\Frefasmname}{Assumption}%
}
\frefformat{plain}{\fancyrefasmlabelprefix}{\frefasmname\fancyrefdefaultspacing#1}
\Frefformat{plain}{\fancyrefasmlabelprefix}{\Frefasmname\fancyrefdefaultspacing#1}

\newcommand*{\fancyrefremlabelprefix}{rem}
\fancyrefaddcaptions{english}{%
	\providecommand*{\frefremname}{remark}%
	\providecommand*{\Frefremname}{Remark}%
}
\frefformat{plain}{\fancyrefremlabelprefix}{\frefremname\fancyrefdefaultspacing#1}
\Frefformat{plain}{\fancyrefremlabelprefix}{\Frefremname\fancyrefdefaultspacing#1}

\newcommand*{\fancyrefexlabelprefix}{ex}
\fancyrefaddcaptions{english}{%
	\providecommand*{\frefexname}{example}%
	\providecommand*{\Frefexname}{Example}%
}
\frefformat{plain}{\fancyrefexlabelprefix}{\frefexname\fancyrefdefaultspacing#1}
\Frefformat{plain}{\fancyrefexlabelprefix}{\Frefexname\fancyrefdefaultspacing#1}

\newcommand*{\fancyrefproblabelprefix}{prob}
\fancyrefaddcaptions{english}{%
	\providecommand*{\frefprobname}{problem}%
	\providecommand*{\Frefprobname}{Problem}%
}
\frefformat{plain}{\fancyrefproblabelprefix}{\frefprobname\fancyrefdefaultspacing#1}
\Frefformat{plain}{\fancyrefproblabelprefix}{\Frefprobname\fancyrefdefaultspacing#1}

\newcommand*{\fancyrefpartlabelprefix}{part}
\fancyrefaddcaptions{english}{%
	\providecommand*{\frefpartname}{part}%
	\providecommand*{\Frefpartname}{Part}%
}
\frefformat{plain}{\fancyrefpartlabelprefix}{\frefpartname\fancyrefdefaultspacing#1}
\Frefformat{plain}{\fancyrefpartlabelprefix}{\Frefpartname\fancyrefdefaultspacing#1}

\title{Generalized Positive Energy Representations of Groups of Jets}
\date{\today}
\author{Niestijl, M.}

\begin{document}
	\pagenumbering{roman}
	\maketitle

	\begin{abstract}
		\noindent
		Let $V$ be a finite-dimensional real vector space and $K$ a compact simple Lie group with Lie algebra $\mfr{k}$. Consider the Fr\'echet-Lie group $G := J_0^\infty(V; K)$ of $\infty$-jets at $0 \in V$ of smooth maps $V \to K$, with Lie algebra $\g = J_0^\infty(V; \mfr{k})$. Let $P$ be a Lie group and write $\p := \Lie(P)$. Let $\alpha$ be a smooth $P$-action on $G$. We study smooth projective unitary representations $\overline{\rho}$ of $G \rtimes_\alpha P$ that satisfy a so-called generalized positive energy condition. In particular, this class captures representations that are in a suitable sense compatible with a KMS state on the von Neumann algebra generated by $\overline{\rho}(G)$. We show that this condition imposes severe restrictions on the derived representation $d\overline{\rho}$ of $\g \rtimes \p$, leading in particular to sufficient conditions for $\restr{\overline{\rho}}{G}$ to factor through $J_0^2(V; K)$, or even through $K$. 
	\end{abstract}

	\section{Introduction}~\label{sec: introduction}
	\pagenumbering{arabic}
	
	\noindent
	This paper is concerned with projective representations of groups and Lie algebras of jets. Let $K$ denote a compact simple Lie group with Lie algebra $\mfr{k}$ and let $V$ be a finite-dimensional real vector space. Then we consider the Fr\'echet-Lie group $G := J_0^\infty(V; K)$ with Lie algebra $\g := J^\infty_0(V; \mfr{k})\cong \R\llbracket V^\ast \rrbracket \otimes \mfr{k}$. These consist of $\infty$-jets at $0 \in V$ of smooth $K$- and $\mfr{k}$-valued functions, respectively. We are interested in smooth projective unitary representations of $G$ which satisfy either a so-called positive energy, or a KMS condition.\\
	
	\noindent
	To describe these, let $P$ be a finite-dimensional Lie group with Lie algebra $\p$. Assume that there is a smooth action $\alpha$ of $P$ on $G$. A continuous projective unitary representation $\overline{\rho} : G \rtimes_\alpha P \to \PU(\mcH_\rho)$ is of \textit{positive energy} at the cone $\mcC\subseteq \p$ if for every $p \in \mcC$, there is a strongly continuous homomorphic lift $t\mapsto e^{itH}$ of $t\mapsto \overline{\rho}(e^{tp})$ whose generator $H$ satisfies $\Spec(H) \geq 0$. We say that $\overline{\rho}$ is KMS at $p \in \p$ if there is a normal state $\phi$ on the von Neumann algebra $\M := \overline{\rho}(G)^{\prime \prime}$ generated by $\overline{\rho}(G)$ such that $\phi$ satisfies the KMS condition for the one-parameter group $\R \to \Aut(\M), \; t \mapsto \Ad(\overline{\rho}(e^{tp}))$. As we shall see, these two seemingly very different classes of representations exhibit similar behavior in certain respects. In particular, they both give rise to so-called \textit{generalized positive energy representations}, a notion which relaxes the positive energy condition and is introduced in \Fref{sec: quas_pe} below. We study these generalized positive energy representations and thereby also those which satisfy either the positive energy or the KMS condition.\\

	\noindent
	The motivation for looking at the positive energy and KMS representations of the group $G\rtimes_\alpha P$ originates in prior work by B.\ Janssens and K.-H.\ Neeb, who studied in \cite{BasNeeb_PE_reps_I} a class of projective unitary representations of the group of compactly supported gauge transformations $\mc{G} := \Gamma_{\mrm{c}}(M; \Ad(\mc{K}))$ of a principal $K$-bundle $\mc{K}$ over $M$, where $\Ad(\mc{K})$ denotes the corresponding adjoint bundle. Suppose that the Lie group $P$ acts smoothly on $\mcK$ by automorphisms of the principal bundle $\mc{K}$. This induces a smooth action of $P$ on the infinite-dimensional Lie group $\mc{G}$. Their main result is:
	\begin{theorem}(\cite[Theorem 7.19]{BasNeeb_PE_reps_I}):\label{thm: basneeb}\\
		Let $(\overline{\rho}, \mcH)$ be a projective unitary representation of $\mc{G} \rtimes P$ which has a dense set of smooth rays and is of positive energy at the cone $\mcC \subseteq \p$. If the cone $\mcC$ has no fixed points in $M$, then there exists a $1$-dimensional $P$-equivariantly embedded submanifold $S\subseteq M$ s.t.\ on the connected component $\mc{G}_0$ of the identity, the projective representation $\overline{\rho}$ factors through the restriction map $r : \mc{G}_0 \to \Gamma_{\mrm{c}}(S; \Ad(\mc{K}))$.
	\end{theorem}
	
	\noindent
	Thus, \textit{if} there are no fixed points in $M$ for $\mcC$, then the problem of classifying the projective unitary positive energy representations of $\mc{G} \rtimes P$ is essentially reduced to the one-dimensional case, which has been extensively studied, see for example \cite{Segal_Loop_Groups, Wassermann_fusion_PE_reps, Tanimoto_ground_state_reps, Kac_book, Kac_bombay, Wallach_loop_gps_diff_circle, Toledano_PE_reps_non_simply_ctd}. Moreover, if there are no one-dimensional $P$-equivariantly embedded submanifolds in $M$, one is essentially reduced to the case where $\overline{\rho}$ factors through the germs at the fixed point set $\Sigma \subseteq M$ of the cone $\mcC \subseteq \p$. In the present paper, we address the setting where fixed points \textit{do} exist and where $\overline{\rho}$ actually factors through the germs at a \textit{single} fixed point.\\
	
	\noindent
	Thus, let $a \in M$ be a fixed point of the $P$-action on $M$ and let $V := T_a(M)$. If a smooth projective unitary representation $\overline{\rho}$ of $\mc{G}$ factors through the germs at $a \in M$, then the continuity of $\overline{\rho}$ implies that it must further factor through the $\infty$-jets $J^\infty_a(\Ad(\mc{K})) \cong J_0^\infty(V; K) = G$ at $a \in M$, as is shown in \Fref{sec: germs_to_jets} of the appendix. This brings us to groups of jets and motivates the study of smooth projective unitary representations of $G \rtimes_\alpha P$. Clearly, any smooth projective unitary representation of $G\rtimes_\alpha P$ defines one of $\mc{G} \rtimes P$ via the jet homomorphism $j^\infty_a : \mc{G} \to J^\infty_a(\Ad(\mc{K}))\cong G$. In this way, the present paper contributes to the understanding of positive energy and KMS-representations of gauge groups. \\

	\noindent
	In \cite{Tobias_typeI_factor_reps}, KMS-representations were very recently studied in the context of finite-dimensional Lie groups, leading to complete characterization of such representations that generate a factor of type $\mrm{I}$. In relation to the unitary representation theory of gauge groups, let us also mention the papers \cite{Gelfand_Graev_energy_repr_1, Albeverio_energy_repr_1, Parthasarathy_Schmidt_en_reps} and \cite{Ismagilov}, in which unitary representations of gauge groups $C^\infty_{\mrm{c}}(M; K)$ are constructed which are non-local in the sense that they do not factor through the restriction map $C^\infty_{\mrm{c}}(M; K) \to C^\infty_{\mrm{c}}(N; K)$ for some proper submanifold $N \subseteq M$. When $\dim(M) \geq 3$, these are irreducible (\cite{Wallach_irred_energy_rep} and \cite{Albeverio_irred_energy_repr}). They are also considered in \cite{Albeverio_Gordina_equivalence_en_brown_reprs} and \cite{Albeverio_book_noncomm_distr}. Unitary representations of groups of jets have also been considered in \cite{Gelfand_Graev_reps_jets} and \cite{Alb_jet_gps_reps}.
	
	\subsubsection*{Structure of the paper}
	
	\noindent
	The paper is divided in two parts. In \Fref{part: reps}, we introduce both the (generalized) positive energy and the KMS condition. We start in \Fref{sec: proj_reprs_ce} by briefly recalling the relation between continuous projective unitary representations, central extensions and the second Lie algebra cohomology $H^2_{\ct}(\g, \R)$. We move on to consider positive energy representations in \Fref{sec: pe_reps} and discuss some of their properties. In \Fref{sec: quas_pe} we relax the positive energy condition and introduce the so-called \textit{generalized positive energy} condition. We show that for a projective unitary representation which of is generalized positive energy, its kernel is related to a particular quadratic form canonically associated to the corresponding class in $H^2_{\mathrm{ct}}(\g, \R)$. This observation will play a key role in \Fref{part: jets}. In \Fref{sec: KMS_reps}, we briefly recall the modular theory of von Neumann algebras and then proceed to define KMS representations. We consider a number of examples and discuss some of their properties, in particular making the important observation that KMS representations give rise to generalized positive energy ones. To some extent, this unifies the positive energy and KMS conditions, allowing their simultaneous study. We remark also that \Fref{part: reps} is formulated in the general context of possibly infinite-dimensional locally convex Lie groups, which is the appropriate context within the larger program that studies the projective unitary representations of gauge groups.\\
	
	\noindent
	In \Fref{part: jets}, we return to the Fr\'echet-Lie group $G := J^\infty_a(V; K)$ of $\infty$-jets at $0 \in V$. After fixing our notation, we discuss in \Fref{sec: normal_form} a normal-form problem for the $\p$-action on $\g = J^\infty_a(V; \mfr{k})$. Using the general observations made in \Fref{part: reps} together with the normal-form results obtained in \Fref{sec: normal_form}, we proceed in \Fref{sec: restr_proj_reps} with the study of (generalized) positive energy representations of the Lie algebra $\g\rtimes_D \p$, where $D : \p \to \der(\g)$ denotes the $\p$-action on $\g$ corresponding to $\alpha$.
	
	\subsubsection*{Overview of main results}
	
	\noindent
	Let us describe the main results of \Fref{sec: restr_proj_reps}. We will first need to introduce some notation. We write $\X_I$ for the Lie algebra of formal vector fields on $V$ vanishing at the origin. The $\p$-action $D$ splits into a horizontal and a vertical part according to $D(p) = -\mc{L}_{\bm{v}(p)} + \ad_{\sigma(p)}$, for some Lie algebra homomorphism $\bm{v} : \p \to \X_I^{\op}$ and a linear map $\sigma : \p \to \g$ satisfying the Maurer-Cartan equation
	\[ - \mc{L}_{\bm{v}(p_1)}\sigma(p_2) + \mc{L}_{\bm{v}(p_2)}\sigma(p_1) - \sigma([p_1, p_2]) + [\sigma(p_1), \sigma(p_2)] = 0, \qquad \forall p_1, p_2 \in \p.  \]
	For any $p \in \p$, the formal vector field $\bm{v}(p)$ splits further as $\bm{v}(p) = \vl{p} + \vho{p}$, into its linearization $\vl{p}$ and its higher order part $\vho{p}$, which is a formal vector field on $V$ vanishing up to first order at the origin. Let $\sigma_0(p)$ be the constant part of the formal power series $\sigma(p)$. Let $\Sigma_p \subseteq \C$ denote the additive subsemigroup of $\C$ generated by $\Spec(\vl{p})$. Let $V_{\mrm{c}}^\C(p)$ denote the span in $V_\C$ of all generalized eigenspaces of $\vl{p}$ corresponding to eigenvalues with zero real part. Set $V_{\mrm{c}}(p) := V_{\mrm{c}}^\C(p) \cap V$. If $\mfr{C} \subseteq \p$ is a subset, define $V_{\mrm{c}}(\mfr{C}) := \bigcap_{p \in \mfr{C}} V_{\mrm{c}}(p)$, which we call the \squotes{center subspace of $V$ associated to $\mfr{C}$}, in analogy with the center manifold of a fixed point of a dynamical system. Let $V_{\mrm{c}}(\mfr{C})^\perp \subseteq V^\ast$ be its annihilator in $V^\ast$. If $\olpi$ is a continuous projective unitary representation of $\g \rtimes_D \p$, let $\mfr{C}(\olpi)$ be the set of all points $p \in \p$ for which $\olpi$ is of generalized positive energy at $p$. Write $R := \R\llbracket V^\ast\rrbracket := \prod_{n=0}^\infty P^n(V)$ for the ring of formal power series on $V$, where $P^n(V)$ denotes the set of degree-$n$ homogeneous polynomials on $V$.\\

	\noindent
	The first main result concerns positive energy representations. It states that unless the spectrum of $\vl{p}$ and $\sigma_0(p)$ happens to intersect non-trivially, any smooth projective unitary representation $\overline{\rho}$ of $G \rtimes_\alpha P$ which is of positive energy at $p \in \p$ factors through the $2$-jets $J_0^2(V; K) \rtimes_\alpha P$:

	\begin{restatable*}{theorem}{drhofactpecase}\label{thm: fact_pe_case}
		Let $\overline{\rho}$ be a smooth projective unitary representation of $G \rtimes_\alpha P$ which is of positive energy at $p \in \p$. Assume that $\Spec(\ad_{\sigma_0(p)}) \cap \Spec(\vl{p}) = \emptyset$. Then $\overline{\rho}$ factors through $J_0^2(V; K) \rtimes_\alpha P$. Moreover the image of $-\mc{L}_{\vl{p}} + \ad_{\sigma_0(p)}$ in $P^2(V) \otimes \mfr{k} \subseteq J_0^2(V; K)$ is contained in $\ker \overline{\rho}$.
	\end{restatable*}

	\noindent
	The second main result determines restrictions imposed by the generalized positive energy condition. If $p \in \mfr{C}(\olpi)$, then unless possibly when the \dquotes{non-resonance condition} $\Spec(\ad_{\sigma_0(p)}) \cap \Sigma_p = \emptyset$ is violated, it suffices to consider the case where all eigenvalues of $\vl{p}$ are purely imaginary. The precise statement is:

	\begin{restatable*}{theorem}{thmfactresultspectralcondition}\label{thm: fact_result_spectral_condition}
		Let $\olpi$ be a continuous projective unitary representation of $\g \rtimes_D \p$. Let $\mfr{C} \subseteq \mfr{C}(\olpi)$. Assume that $\Spec(\ad_{\sigma_0(p)}) \cap \Sigma_p = \emptyset$ for all $p \in \mfr{C}$. Then $RV_{\mrm{c}}(\mfr{C})^\perp\otimes \mfr{k} \subseteq \ker \olpi$, and hence $\restr{\olpi}{\g}$ factors through $\R\llbracket V_{\mrm{c}}(\mfr{C})^\ast \rrbracket \otimes \mfr{k}$. 
	\end{restatable*}

	\noindent
	Since $\R\llbracket V_{\mrm{c}}(\mfr{C})^\ast \rrbracket \otimes \mfr{k} = \mfr{k}$ whenever $V_{\mrm{c}}(\mfr{C}) = \{0\}$, \Fref{thm: fact_result_spectral_condition} in particular gives sufficient conditions for $\olpi$ to factor through $\mfr{k}$, that depend only on the spectrum of $\sigma_0$ and $\vl{p}$. For the third main result, we consider the special case where $\p$ is non-compact and simple:

	\begin{restatable*}{theorem}{thmsemisimplecone}\label{thm: semisimple_cone_must_be_pointed}
		Assume that $\p$ is non-compact and simple. Suppose that $\bm{v}_{\mrm{l}}$ defines a non-trivial irreducible $\p$-representation on $V$. Let $\olpi$ be a continuous projective unitary representation of $\g \rtimes_D \p$. Let $\mfr{C} \subseteq \mfr{C}(\olpi)$ be a $P$-invariant convex cone. Either $\mfr{C}$ is pointed or $\restr{\olpi}{\g}$ factors through $\mfr{k}$.
	\end{restatable*}

	\begin{remark}
		If $\overline{\rho}$ is a smooth projective unitary representation of $G \rtimes_\alpha P$ which is of generalized positive energy at the cone $\mc{C} \subseteq \p$, then its derived representation $d\overline{\rho}$ on the space of smooth vectors $\mcH_\rho^\infty$ is so, too. Moreover, as we shall see in \Fref{lem: exp_restricts_to_diffeo} below, the exponential map of $G = J^\infty_0(V; K)$ restricts to a diffeomorphism from the pro-nilpotent ideal $\ker\big(\ev_0 : J^\infty_0(V; \mfr{k}) \to \mfr{k}\big)$ onto $\ker\big(\ev_0 : J^\infty_0(V; K) \to K\big)$. Thus, the above results all have immediate analogous consequences on the group level.
	\end{remark}

	\subsubsection*{Acknowledgements}
	This research is supported by the NWO grant 639.032.734 \dquotes{Cohomology and representation theory of infinite-dimensional Lie groups}. I am grateful to be given the opportunity to pursue a PhD in mathematics, which was given to me by my supervisor Bas Janssens, whom I would also like to thank for his guidance during the process. I am also grateful for helpful discussions with Karl-Hermann Neeb, Tobias Diez and Lukas Miaskiwskyi. Finally, I am grateful for the anonymous referee, who has provided various helpful suggestions.

	\part{Positive Energy and KMS Representations}\label{part: reps}
	
	\section{Projective Representations and Central Extensions}\label{sec: proj_reprs_ce}
	
	\noindent
	In the following, the category of infinite-dimensional manifolds and smooth maps between them is defined in the Michal-Bastiani sense \cite{bastiani, milnor_inf_lie, neeb_towards_lie}. This also defines the notion of a locally convex Lie group. Throughout the section, let $G$ denote a locally convex Lie group which is \textit{regular} in the sense of \cite[Def.\ II.5.2]{neeb_towards_lie}. Then $G$ in particular admits an exponential map $\exp_G : \g \to G$, see e.g.\ \cite[Rem.\ II.5.3]{neeb_towards_lie}.

	\begin{definition}\label{def: smooth_reps}~
		\begin{itemize}
			\item A (projective) unitary representation of $G$ is said to be continuous if it is so w.r.t. the strong operator topology on $\U(\mcH_\rho)$.	
			\item Let $(\rho, \mcH_\rho)$ be a unitary representation of $G$ on $\mcH_\rho$. A vector $\psi \in \mcH_\rho$ is called \textit{smooth} if the orbit map $G \to \mcH_\rho, \; g \mapsto \rho(g)\psi$ is smooth. Denote by $\mcH_\rho^\infty\subseteq \mcH_\rho$ the subspace of smooth vectors. The representation $\rho$ is called smooth if $\mcH_\rho^\infty$ is dense in $\mcH_\rho$. 
			\item Let $(\overline{\rho}, \mcH_\rho)$ be a projective unitary representation of $G$ on $\mcH_\rho$. A ray $[\psi] \in \mrm{P}(\mcH_\rho)$ is called \textit{smooth} if the orbit map $G \to \mrm{P}(\mcH_\rho), \; g \mapsto \overline{\rho}(g)[\psi]$ is smooth. Denote by $\mrm{P}(\mcH_\rho)^\infty$ the subspace of smooth rays. The projective representation $\overline{\rho}$ is called smooth if $\mrm{P}(\mcH_\rho)^\infty$ is dense in $\mrm{P}(\mcH_\rho)$. 
		\end{itemize}
	\end{definition}

	\begin{definition}
		If $\mcD$ is a complex vector space, denote by $\mc{L}(\mcD)$ the Lie algebra of linear operators on $\mcD$. \\
		If $\mc{D}$ is a pre-Hilbert space with Hilbert space completion $\mcH$, we also write 
		\[\mc{L}^\dagger(\mc{D}) := \set{ X \in \mc{L}(\mc{D}) \st \mcD \subseteq \dom(X^\ast) \text{ and } X^\ast \mc{D} \subseteq \mc{D}}.  \] 
		For $X \in \mc{L}^\dagger(\mcD)$ set $X^\dagger := \restr{X^\ast}{\mc{D}}$. Then $(X^\dagger)^\dagger = X$ and $(\--)^\dagger$ endows $\mc{L}^\dagger(\mc{D})$ with an involution \cite[Ch. 2]{Schmudgen_book}.\\
		If $\mcD$ is a pre-Hilbert space, define the Lie algebra 
		$$\mfr{u}(\mcD) := \set{X \in \mc{L}^\dagger(\mc{D}) \st X^\dagger + X = 0}.$$
	\end{definition}

	\begin{definition}
		Let $\mcD$ be a complex pre-Hilbert space.
		\begin{itemize}
			\item A \textit{unitary representation} of the Lie algebra $\g$ on $\mcD$ is a Lie algebra homomorphism $\pi : \g \to \mfr{u}(\mcD)$. A \textit{projective unitary representation} is a Lie algebra homomorphism $\olpi : \g \to \pu(\mcD) := \mfr{u}(\mcD)/ i\R I$.
			\item A unitary representation $\pi$ of $\g$ is called continuous if $\xi \mapsto \pi(\xi)\psi$ is continuous for any $\psi \in \mcD$. A projective unitary representation $\olpi$ is continuous if $\xi \mapsto \pi(\xi)[\psi]$ is continuous for every $[\psi] \in \mrm{P}(\mcD)$. 
		\end{itemize}
	\end{definition}

	\begin{remark}
		Any unitary $G$-representation on $\mcH_\rho$ defines a unitary $\g$-representation $d\rho : \g \to \mfr{u}(\mcH_\rho^\infty)$ on $\mcH_\rho^\infty$ by $d\rho(\xi)\psi := \restr{d\over dt}{t=0}\rho(e^{t\xi})\psi$. We will always consider elements of $d\rho(\g)$ as unbounded operators defined on the invariant domain $\mcH_\rho^\infty$. Projective unitary $G$-representations similarly define projective unitary $\g$-representations on $\mrm{P}(\mcH_\rho^\infty)$ by differentiation at the identity. If $G$ is finite-dimensional, then $\mcH_\rho^\infty$ is dense in $\mcH_\rho$ for any continuous unitary representation $\rho$ of $G$, by a result of G\r{a}rding \cite[Prop.\ 4.4.1.1]{Warner_book_1}.\\
	\end{remark}

	\noindent
	A continuous projective unitary representation $\overline{\rho} : G \to \PU(\mcH_\rho)$ is equivalently given by a continuous central $\T$-extension $\circled{G}$ together with a unitary representation $\rho : \circled{G} \to \U(\mcH_\rho)$ which satisfies $\rho(z) = zI$ for $z$ in the central $\T$ component. Of course, $\circled{G}$ is the pull-back of the central $\T$-extension $\U(\mcH_\rho) \to \PU(\mcH_\rho)$ along $\overline{\rho}$. We say that $\rho$ \textit{lifts} $\overline{\rho}$. Suppose $\overline{\rho}_1$ and $\overline{\rho}_2$ are two projective unitary representations, inducing by pull-back the lifts $\rho_1 : \circled{G}_1 \to \U(\mcH_{\rho_1})$ and $\rho_2 : \circled{G}_2 \to \U(\mcH_{\rho_1})$ of $\overline{\rho}_1$ and $\overline{\rho}_2$, respectively. Then $\overline{\rho}_1$ and $\overline{\rho}_2$ are unitarily equivalent if and only if there is an isomorphism $\Phi : \circled{G}_1 \to \circled{G}_2$ of central $G$-extensions and a unitary $U : \mcH_{\rho_1} \to \mcH_{\rho_2}$ such that $\rho_2(\Phi(x)) = U\rho_1(x)U^{-1}$ for all $x \in \circled{G}_1$. Analogously, any projective unitary $\g$-representation $\olpi$ with domain $\mcD$ can be lifted to a unitary representation $\pi : \circled{\g} \to \mfr{u}(\mcD)$ of some central $\R$-extension $\circled{\g}$ of $\g$. The continuous central extensions of $\g$ by $\R$ are up to isomorphism classified by $H_{\ct}^2(\g, \R)$, the continuous second Lie algebra cohomology with trivial coefficients \cite[Def.\ 6.2, Prop.\ 6.3]{BasNeeb_ProjReps}. Thus, to study the projective unitary representations of $\g$ up to equivalence, one may first determine $H_{\ct}^2(\g, \R)$, choose for each class $[\omega]\in H_{\ct}^2(\g, \R)$ a representative $\omega$ and then proceed to determine the equivalence classes of unitary representations $\pi$ of the central extension $\R \oplus_\omega \g$ satisfying $\pi(1, 0) = iI$. We will also write $\bm{c} := (1,0) \in \R \oplus_\omega \g$ for the central generator.\\
	
	\begin{remark}
		In the literature, one encounters the notion of the \textit{level} of a unitary representation $\pi$ of $\R \bm{c} \oplus_\omega \g$, which is the number $l \in \R$ such that $\pi(\bm{c}) = ilI$ (see e.g.\ \cite[sec.\ 9.3]{Segal_Loop_Groups}). Let us briefly clarify how such representations are included in the program outlined above, even though $\pi(\bm{c})= iI$ is always assumed. Simply notice that such a representation of level $l$ factors through the map $\R\bm{c} \oplus_\omega \g \to \R\bm{c} \oplus_{l\cdot \omega} \g$ induced by multiplication by $l$ on the central factor. The corresponding representation $\pi_2$ of $\R\bm{c} \oplus_{l\cdot \omega}\g$ satisfies $\pi_2(\bm{c}) = iI$. Notice that $\R\bm{c} \oplus_\omega \g \to \R \oplus_{l\cdot \omega} \g$ is an isomorphism of Lie algebras whenever $l\neq 0$, but \textit{not} as central extensions unless $l=1$, because a morphism of central extensions is required to be the identity on the central component. For $1 \neq l\in \R$, the cocycles $\omega$ and $l\cdot \omega$ are \textit{not} equivalent unless $[\omega] = 0$ in $H_{\ct}^2(\g,\R)$.
	\end{remark}

	\begin{remark}\label{rem: smooth_lift}
		If a projective unitary representation $\overline{\rho}$ of $G$ is smooth, then the corresponding central $\T$-extension $\circled{G}$ is again a locally convex Lie group \cite[Thm.\ 4.3]{BasNeeb_ProjReps}. Moreover, there is a similar correspondence between smooth projective unitary representations $\overline{\rho}$ of $G$ and their lifts $\rho : \circled{G} \to \U(\mcH_\rho)$, which are then again smooth \cite[Cor.\ 4.5, Thm.\ 7.3]{BasNeeb_ProjReps}. We furthermore have $\mrm{P}(\mcH_\rho)^\infty = \mrm{P}(\mcH_\rho^\infty)$ by \cite[Thm.\ 4.3]{BasNeeb_ProjReps}.
	\end{remark}

	\section{Positive Energy Representations}\label{sec: pe_reps}
	\noindent
	Let us introduce the class of positive energy representations. After defining the notion, some immediate consequences are considered that will be relevant in \Fref{part: jets}. Let $G$ be a regular locally convex Lie group with Lie algebra $\g$. If $c\in \R$, $\mc{D}$ is a pre-Hilbert space and $X \in \mc{L}^\dagger(\mc{D})$ satisfies $X^\dagger = X$, we write $X \geq c$ if $\langle \psi, X \psi\rangle \geq c \norm{\psi}^2$ for every $\psi \in \mcD$.
	
	\newpage

	\begin{definition}\label{def: smooth_pos_energy}
		Let $\mcC \subseteq \g$ be a convex cone and $\mcD$ be a pre-Hilbert space.
		\begin{itemize}
			\item A continuous unitary representation $\pi$ of $\g$ on $\mcD$ is said to be of \textit{positive energy} (p.e.) at $\xi\in \g$ if $-i\pi(\xi) \geq 0$. It is of p.e.\ at $\mcC$ if it is of p.e.\ at every $\xi \in \mcC$. Write $\mcC(\pi) = \set{ \xi \in \g \st \pi \text{ is of p.e.\ at }\xi}$.
			\item Let $\olpi$ be a continuous projective unitary representation of $\g$ on $\mcD$ with lift $\pi : \circled{\g} \to \mfr{u}(\mcD)$. Then $\olpi$ is of p.e.\ at $\xi$ if there exists $\circled{\xi} \in \mcC(\pi)\subseteq \circled{\g}$ covering $\xi$. Write $\mcD(\olpi)$ for the set of all such $\xi$. We say that $\olpi$ is of p.e.\ at $\mcC$ if $\mcC \subseteq \mcC(\olpi)$.
			\item A smooth (projective) unitary representation of $G$ on $\mcH_\rho$ is said to be of p.e.\ at $\xi \in \g$ if the corresponding derived (projective) unitary representation of $\g$ on $\mcH_\rho^\infty$ is so. It is said to be of p.e.\ at $\mcC$ if it is so at every $\xi \in \mcC$.
		\end{itemize}
	\end{definition}

	\begin{remark}\label{rem: cone_P-invariant}
		Let $\rho$ be a smooth unitary representation of $G$. Then $\mcC := \mcC(d\rho)$ is always a closed, $G$-invariant convex cone. Consequently, $\mcC \cap -\mcC$ and $\mcC - \mcC$ are ideals in $\g$, called the \textit{edge} and \textit{span} of $\mc{C}$, respectively. If $\xi \in \mcC \cap -\mcC$ then $\xi \in \ker d\rho$, so by passing to the quotient $\g/\ker d\rho$ one may always achieve that $\mcC$ is pointed.\\
	\end{remark}

	\noindent
	Next, we define the notion of a semibounded representation.
	
	\begin{definition}\label{def: mom_semibounded}
		Let $\rho : G \to \U(\mcH)$ be a smooth unitary $G$-representation. Define its \textit{momentum set} by:
		\[I_\rho := \conv\set{\xi \mapsto \langle \psi, -id\rho(\xi)\psi\rangle \st \psi \in \mcH_\rho^\infty, \; \norm{\psi} = 1} \subseteq \g^\ast.\]
		The representation $\rho$ is said to be \textit{semibounded} if $W_\rho$ contains an interior point, where
		\[ W_\rho := \set{ \xi \in \g \st \inf \Spec(-i \overline{d\rho(\xi)}) > - \infty}.\]
	\end{definition}

	\begin{remark}
		For finite-dimensional Lie groups, the class of semibounded representations has been subject to detailed study in \cite{Neeb_book_hol_conv}. In particular, they are highest weight representations \cite[Def.\ X.2.9, Thm.\ X.3.9]{Neeb_book_hol_conv}. For a consideration of semibounded representations in the context of infinite-dimensional Lie groups, we refer to \cite{neeb_semibounded_reps} and \cite{neeb_semibounded_inv_cones}.
	\end{remark}

	\noindent
	In the finite-dimensional context, the semiboundedness condition turns out to be extremely restrictive, which in turn has consequences for arbitrary positive energy representations. The following result, \Fref{thm: pe_semibdd_core}, is based on the results in the monograph \cite{Neeb_book_hol_conv}.
	
	\begin{theorem}\label{thm: pe_semibdd_core}
		Assume that $G$ is connected and locally exponential. Take $\bm{d} \in \g$ and let $\mfr{a} = \langle \bm{d}\rangle \triangleleft \g$ be the closed ideal in $\g$ generated by $\bm{d}$. Assume that $\dim(\mfr{a}) < \infty$ and that $\mfr{a}$ is stable, in the sense that $\Ad_G(\mfr{a}) \subseteq \mfr{a}$. Let $A \triangleleft G$ be a connected normal Lie subgroup integrating $\mfr{a}$. Let $(\rho, \mcH_\rho)$ be a smooth unitary $G$-representation which is of p.e.\ at $\bm{d} \in \g$. Write $\h := \mfr{a} /\ker d\rho$. The following assertions hold:
		\begin{enumerate}
			\item $\mfr{a} = \mcC - \mcC$, where $\mcC \subseteq \g$ is the closed $G$-invariant convex cone in $\g$ generated by $\bm{d}$.
			\item The closure of $\mcC + \ker d\rho$ in $\h$ is a pointed, closed, generating and $G$-invariant convex cone. Thus $\mcC \cap -\mcC \subseteq \ker d\rho$.
			\item $\restr{\rho}{A}$ is semibounded.
			\item Let $\h_n$ denote the maximal nilpotent ideal of $\h$. Then $[\h_n, \h_n] \subseteq \z(\h)$. Moreover, there exists a reductive Lie algebra $\mfr{l}$ such that $\h \cong \h_n \rtimes \mfr{l}$.
			\item Let $\mfr{a}_n$ denote the maximal nilpotent ideal of $\mfr{a}$. Then $[\mfr{a}, [\mfr{a}_n, \mfr{a}_n]] \subseteq \ker d\rho$.
		\end{enumerate}
	\end{theorem}
	\begin{proof}
		 For the first point, let $\mfr{a}^\prime$ be the closure of $\mcC - \mcC$ in $\g$. As $\mfr{a}^\prime$ is a closed ideal in $\g$ containing $\bm{d}$, we have $\mfr{a} \subseteq \mfr{a}^\prime$. On the other hand, we know that $\Ad_G(\bm{d}) \subseteq \mfr{a}$ because $\mfr{a}$ is stable. Thus $\mcC \subseteq \mfr{a}$ and hence $\mfr{a}^\prime \subseteq \mfr{a}$. So $\mfr{a}^\prime = \mfr{a}$. In particular $\dim(\mfr{a}^\prime) < \infty$ and so $\mcC - \mcC = \mfr{a}^\prime = \mfr{a}$. Next we prove the second statement. Take $\xi \in (\overline{\mcC + \ker d\rho}) \cap -(\overline{\mcC + \ker d\rho})$. Then $d\rho(\xi) \geq 0$ and $d\rho(\xi) \leq 0$, in view of \Fref{rem: cone_P-invariant}, and hence $\Spec(\overline{d\rho(\xi)}) = \{0\}$. As $d\rho(\xi)$ is essentially skew-adjoint, it follows that $\xi \in \ker d\rho$. Thus $\overline{\mcC  + \ker d\rho}$ is pointed in $\h$. As $\mcC$ is $G$-invariant and convex, it is clear that the same holds for the cone $\overline{\mcC+\ker d\rho}$ in $\h$. The latter is also generating in $\h$ because $\mfr{a} = \mcC - \mcC$. Next we show that $\restr{\rho}{A}$ is semibounded. As $\mfr{a}$ is spanned by $\mcC$ and $\dim \mfr{a} < \infty$, it follows that $\mcC\subseteq \mfr{a}$ has interior points. As $\mcC \subseteq W_\rho$, this implies that $W_\rho$ has interior points. Hence $\restr{\rho}{A}$ is semibounded. For the remaining points, we use the results in \cite{Neeb_book_hol_conv}. We first show that $\h$ is \textit{admissible}, in the sense of \cite[Def.\ VII.3.2]{Neeb_book_hol_conv}. Using the second point, the convex cone $(\overline{\mcC + \ker d\rho}) \oplus \R_{\geq 0}$ in $\h \oplus \R$ is closed, pointed, generating and $\Inn(\h)$-invariant. By \cite[Lem.\ VII.3.1, Def.\ VII.3.2]{Neeb_book_hol_conv} this implies that $\h$ is admissible. By \cite[Thm.\ VII.3.10]{Neeb_book_hol_conv}, it follows that $[\h_n, \h_n] \subseteq\z(\h)$ and that $\h$ contains a compactly embedded Cartan subalgebra $\mfr{t}$ (where as in \cite[Def.\ VII.1.1]{Neeb_book_hol_conv}, a subalgebra $\mfr{t} \subseteq \h$ is called \textit{compactly embedded} if $\overline{\langle e^{\ad(\mfr{t})}\rangle}$ is compact in $\Aut(\h)$). Using \cite[Lem.\ VII.2.26(iv)]{Neeb_book_hol_conv}, we obtain that there exists some reductive Lie algebra $\mfr{l}$ with $\h \cong \h_n \rtimes \mfr{l}$. Since $[\h, [\h_n, \h_n]] = 0$ and $\h = \mfr{a}/\ker d\rho$, it follows in particular that $[\mfr{a}, [\mfr{a}_n, \mfr{a}_n]] \subseteq \ker d\rho$.
	\end{proof}

	\noindent
	For projective p.e.\ representations, this leads to:

	\begin{corollary}\label{cor: pe_vanishing_ideal}
		Let $G$, $\bm{d}, \mfr{a}$ and $\mfr{a}_n$ be as \Fref{thm: pe_semibdd_core}. Let $(\overline{\rho}, \mcH_\rho)$ be a smooth projective unitary representation of $G$. Suppose that $\overline{\rho}$ is of p.e.\ at $\bm{d} \in \g$. Then $[\mfr{a}, [\mfr{a}_n, \mfr{a}_n]] \subseteq \ker d\overline{\rho}$. 
	\end{corollary}
	\begin{proof}
		Let $\rho : \circled{G} \to \U(\mcH_\rho)$ be the lift of $\overline{\rho}$ to a central $\T$-extension $\circled{G}$ of $G$. Let $\circled{\g} := \Lie(\circled{G})$. There exists some $\circled{\bm{d}}\in \circled{\g}$ s.t.\ $d\rho$ is of p.e.\ at $\circled{\bm{d}} \in \circled{\g}$. Let $\circled{\mfr{a}}$ denote the ideal in $\circled{\g}$ generated by $\circled{\bm{d}}$ and let $\circled{\mfr{a}}_n$ denote the maximal nilpotent ideal in $\circled{\mfr{a}}$. Then $d\rho([\circled{\mfr{a}}, [\circled{\mfr{a}}_n, \circled{\mfr{a}}_n]]) = \{0\}$ by \Fref{thm: pe_semibdd_core}. Thus $d\overline{\rho}([\mfr{a}, [\mfr{a}_n, \mfr{a}_n]]) = \{0\}$, where we used that the quotient map $\circled{\g} \to \g$ projects $\circled{\mfr{a}}$ and $\circled{\mfr{a}}_n$ onto $\mfr{a}$ and $\mfr{a}_n$, respectively.
	\end{proof}

	\noindent
	The following simple lemma will also be useful.
	\begin{lemma}\label{lem: full_cone_then_norm_continuous}
		Assume that $\dim(G) < \infty$. Let $\overline{\rho} : G \to \PU(\mcH_\rho)$ be a continuous projective unitary representation of $G$ which is of p.e.\ at every element of $\g$. Then $\overline{\rho}$ is continuous w.r.t. the norm-topology on $\U(\mcH_\rho)$.
	\end{lemma}
	\begin{proof}
		Let $d\rho : \circled{\g} \to \mfr{u}(\mcH_\rho^\infty)$ be the lift of $d\overline{\rho}$. Identify $\circled{\g} \cong \R \oplus_\omega \g$ for some $2$-cocycle $\omega : \g \times \g \to \R$. The assumptions imply that for every $\xi \in \g$ there exists $E_\xi \in \R$ s.t.\ $-id\rho(\xi) \geq E_\xi$. As this holds in particular for both $\xi$ and $-\xi$, $d\rho(\xi)$ is a bounded operator for any $\xi \in \g$. As $\dim(\g) < \infty$, one finds by choosing a basis $(e_\mu)$ of $\g$ that there exists $C > 0$ s.t.\ $\norm{d\rho(\xi)} \leq C \norm{\xi}$ where $\norm{\xi} := \sup_{\mu}\abs{\xi_\mu}$ if $\xi = \sum_\mu \xi_\mu e_\mu$. Thus $\xi \mapsto d\rho(\xi)$ is norm-continuous. This implies norm-continuity of $\overline{\rho}$ because $\B(\mcH_\rho) \to \B(\mcH_\rho), T \mapsto e^T$ is norm-continuous and $\overline{\rho}(\exp(\xi)) = [e^{d\rho(\xi)}] \in \PU(\mcH_\rho)$ for $\xi \in \g$.
	\end{proof}

	\section{Generalized Positive Energy Representations}\label{sec: quas_pe}
	Let $G$ denote a regular locally convex Lie group with Lie algebra $\g$. The class of p.e.\ representations can be generalized by relaxing the condition $-id\rho(\xi) \geq 0$ in \Fref{def: smooth_pos_energy}. We define a suitable relaxed notion, the \textit{generalized positive energy condition}, and show that it can still be very restrictive. In \Fref{sec: KMS_reps}, we will encounter a class of representations which are not of p.e.\ but are of generalized positive energy. 
	
	\begin{definition}\label{def: qpe}
		Let $\mc{D}$ be a pre-Hilbert space with Hilbert space completion $\mcH$. Let $\h$ be a locally convex topological Lie algebra.
		\begin{itemize}
			\item A continuous unitary representation $\pi : \h \to \mfr{u}(\mcD)$ is of \textit{generalized positive energy} (g.p.e.) at $\xi \in \h$ if there exists a $1$-connected Lie group $H$ with Lie algebra $\h$ and a dense subspace $\mcD_0 \subseteq \mcD$ such that
			\begin{equation}\label{eq: qpe_la}
				\forall \psi \in \mcD_0 \st E_\psi(\pi, \xi) := \inf_{h \in H}\langle \psi, -i\pi(\Ad_h(\xi)) \psi\rangle > - \infty.
			\end{equation}
			We write $\mfr{C}(\pi) := \set{\xi \in \g \st \pi \text{ is of g.p.e.\ at }\xi}$. If $\mfr{C} \subseteq \g$, we say that $\pi$ is of g.p.e.\ at $\mfr{C}$ if $\mfr{C} \subseteq \mfr{C}(\pi)$.
			\item Let $\olpi : \h \to \mfr{pu}(\mcD)$ be a continuous projective unitary representation of $\h$ on $\mcD$ with lift $\pi :\circled{\h} \to \mfr{u}(\mcD)$. Let $\mfr{C}(\olpi)\subseteq \h$ denote the image of $\mfr{C}(\pi)\subseteq \circled{\h}$ under the quotient map $\circled{\h} \to \h$. Then $\olpi$ is said to be of generalized positive energy at $\xi \in \h$ if $\xi \in \mfr{C}(\olpi)$. Similarly, we say it is of g.p.e.\ at $\mfr{C} \subseteq \h$ if $\mfr{C} \subseteq \mfr{C}(\olpi)$.
			\item Let $\rho : G \to \U(\mcH_\rho)$ be a smooth unitary representation of $G$. Then $\rho$ is said to be of g.p.e.\ at $\xi \in \g$ if its derived representation $d\rho$ on $\mcH_\rho^\infty$ is so.
			\item Let $\overline{\rho} : G \to \PU(\mcH_\rho)$ be a smooth projective unitary representation of $G$ with lift $\rho : \circled{G} \to \U(\mcH_\rho)$. Let $\circled{\g}$ be the Lie algebra of $\circled{G}$. Then $\overline{\rho}$ is of g.p.e.\ at $\xi \in \g$ if $\rho$ is of g.p.e.\ at some $\circled{\xi} \in \circled{\g}$ covering $\xi$. 
		\end{itemize}
	\end{definition}

	\begin{remark}\label{rem: qpe_cone_invariant}
		If $\pi$ is a (projective) continuous unitary representation of $\g$, then the set $\mfr{C}(\pi)\subseteq \g$ is always an $\Ad_G$-invariant cone. 
	\end{remark}

	\noindent
	An important observation for the class of g.p.e.\ representations is the following one:
	
	\begin{lemma}\label{lem: CS-qpe}
		Let $\pi : \g \to \mfr{u}(\mc{D})$ be a continuous unitary representation of $\g$ on the pre-Hilbert space $\mc{D}$. Let $\xi \in \mfr{C}(\pi)$. Suppose that $\eta \in \g$ satisfies $[[\xi, \eta], \eta] \in \mfr{Z}(\g)$. Then for every $\psi$ in some dense subspace $\mc{D}_0$ we have:
		\begin{equation}\label{eq: cs-qpe}
			\begin{aligned}
				0 &\leq \langle \psi, -i\pi([[\xi, \eta], \eta])\psi\rangle, \\
				\langle \psi, -i\pi([\xi, \eta]) \psi\rangle^2 &\leq 2 \langle \psi, -i\pi([[\xi, \eta], \eta])\psi\rangle \bigg( \langle \psi, -i\pi(\xi)\psi\rangle - E_\psi(\pi, \xi)) \bigg).
			\end{aligned}
		\end{equation}
		In particular, if $[[\xi, \eta], \eta] = 0$ then $\pi([\xi, \eta]) = 0$.
	\end{lemma}
	\begin{proof}
		Let $\mcD_0 \subseteq \mcD$ be a dense subspace for which \eqref{eq: qpe_la} is valid. Let $\psi \in \mc{D}_0$. Then $\langle \psi, -i\pi(e^{t \ad_\eta}\xi)\psi\rangle \geq E_\psi(\pi, \xi)$ for all $t \in \R$. As $[[\xi, \eta], \eta] \in \mfr{Z}(\g)$, the third derivative $\gamma^{(3)} : \R \to \g$ of the smooth path $\gamma : \R \to \g, \; t \mapsto e^{t\ad_\eta}\xi$ vanishes. From Taylor's formula (which holds for smooth maps between locally convex vector spaces by \cite[Prop.\ I.2.3]{neeb_towards_lie}), it follows that $e^{t \ad_\eta}\xi = \xi + t[\eta, \xi] + {t^2 \over 2}[[\xi, \eta], \eta]$ for all $t \in \R$. Thus
		\[ \big\langle \psi, -i\pi(\xi)\psi \rangle +  t\langle \psi, -i\pi([\eta, \xi]) \psi \rangle +{t^2 \over 2} \langle \psi, -i\pi([[\xi, \eta], \eta]) \psi\rangle \geq E_\psi(\pi, \xi), \qquad \forall t \in \R\]
		The equations \eqref{eq: cs-qpe} follows from the fact that $at^2 + bt + c \geq 0$ for all $t \in \R$ if and only if $a,c \geq 0$ and $b^2 \leq 4ac$. In particular, if $[[\xi, \eta], \eta] = 0$ then $\langle \psi, -i\pi([\xi, \eta]) \psi\rangle = 0$ for all $\psi \in \mc{D}_0$. As $\mcD_0$ is a complex vector space, this implies by the polarization identity that $\pi([\xi, \eta]) = 0$.
	\end{proof}
	
	\noindent
	In the projective context, this sets up a relation between $\ker \olpi$ and the class $[\omega] \in H_{\ct}^2(\g; \R)$ defined by the corresponding central $\R$-extension $\circled{\g}$ of $\g$. This is exploited in \Fref{sec: restr_proj_reps}.
	
	\begin{proposition}\label{prop: qpe_kernel}
		Let $\olpi$ be a continuous projective unitary $\g$-representation on the pre-Hilbert space $\mc{D}$ with lift $\pi : \circled{\g} \to \mfr{u}(\mc{D})$ for some continuous central $\R$-extension $\circled{\g}$ of $\g$. Let $\omega$ represent the corresponding class in $H^2_{\ct}(\g, \R)$. Let $\xi \in \mfr{C}(\olpi)$. Suppose that $\eta \in \g$ satisfies $[[\xi, \eta], \eta] = 0$. Then $\omega([\xi, \eta], \eta)\geq 0$ and
		\begin{align*}
			\omega([\xi, \eta], \eta) = 0 \iff \olpi([\xi, \eta]) = 0.
		\end{align*}
	\end{proposition}	
	\begin{proof}
		Identify $\circled{\g}$ with $\R \oplus_\omega \g$. Let $\circled{\xi} \in \mfr{C}(\pi)$ and $\circled{\eta} \in \circled{\g}$ be lifts of $\xi$ and $\eta$, respectively. We have that $[[\circled{\xi}, \circled{\eta}], \circled{\eta}] = \omega([\xi, \eta], \eta) \in \mfr{Z}(\circled{\g})$, because $[[\xi, \eta], \eta] = 0$. Using \Fref{lem: CS-qpe} it follows that $\omega([\xi, \eta], \eta) \geq 0$. If $\omega([\xi, \eta], \eta) = 0$, then $[[\circled{\xi}, \circled{\eta}], \circled{\eta}] = 0$ and so \Fref{lem: CS-qpe} implies that $\pi([\circled{\xi}, \circled{\eta}]) = 0$. Hence $\olpi([\xi, \eta]) = 0$. Conversely, if $\olpi([\xi, \eta]) = 0$, then $i\omega([\xi, \eta], \eta) = [\pi([\xi, \eta]), \pi(\eta)] - \pi([[\xi, \eta], \eta]) = 0$, because $[[\xi, \eta], \eta] = 0$.
	\end{proof}
	
	\begin{remark}
		Notice in the setting of \Fref{prop: qpe_kernel} that whenever $[[\xi, \eta], \eta] = 0$, the value of $\omega([\xi, \eta], \eta)$ does not depend on the choice of representative $\omega$ of the class $[\omega] \in H^2_{\ct}(\g, \R)$.
	\end{remark}

	\noindent
	In \Fref{part: jets}, a particular special case of \Fref{prop: qpe_kernel} is used extensively:
	
	\begin{corollary}\label{cor: qpe_kernel_special_case}
		Let $\p$ and $\g$ be locally convex Lie algebras. Let $D : \p \to \der(\g)$ be a homomorphism for which the corresponding action $\p \times \g \to \g$ is continuous. Let $\mcD$ be a complex pre-Hilbert space and let $\olpi : \g \rtimes_D \p \to \pu(\mcD)$ be a continuous projective unitary representation of $\g \rtimes_D \p$ on $\mcD$. Let $[\omega] \in H^2_{\ct}(\g \rtimes_D \p; \R)$ be the corresponding class in $H^2_{\ct}(\g \rtimes_D \p; \R)$. Let $\eta \in \g$, $p \in \mfr{C}(\olpi)\cap \p$ and assume that $[D(p)\eta, \eta] = 0$. Then $\omega(D(p)\eta, \eta) \geq 0$ and $\omega(D(p)\eta, \eta) = 0 \iff \olpi(D(p)\eta) = 0$.
	\end{corollary}

	\section{KMS Representations}\label{sec: KMS_reps}
	
	\noindent
	In the following, we introduce the class of KMS representations. We will see in particular that these give rise to generalized positive energy representations. Consequently, they can be studied using the results of \Fref{sec: quas_pe}. Its definition makes use of the modular theory of von Neumann algebras, which we recall first.
	
	\subsection{Modular Theory of von Neumann Algebras}
	Let us recall the modular condition and the notion of a KMS state on a von Neumann algebra $\M$, whilst fixing our conventions and notation. We refer to \cite[Ch.\ VIII]{Takesaki_II},  \cite[Ch.\ 2.5]{bratelli_robinson_1} and \cite[Ch.\ 5.3]{bratelli_robinson_2} for a detailed consideration of the modular theory of von Neumann algebras and of KMS states. \\
	
	\noindent
	If $\M$ is a von Neumann algebra, write $\M_\ast$ for its pre-dual, equipped with the $\sigma(\M_\ast, \M)$-topology. Write $\mc{S}(\M) \subseteq \M_\ast$ for the set of normal states on $\M$. Further, if $\phi \in \mc{S}(\M)$, write $\pi_\phi : \M \to \B(\mcH_\phi)$ for the GNS-representation of $\M$ relative to $\phi$. Write $\M_\phi := \pi_\phi(\M)^{\prime \prime}$. Let $\Omega_\phi \in \mcH_\phi$ denote the canonical cyclic vector satisfying $\phi(x) = \langle \Omega_\phi, \pi_\phi(x)\Omega_\phi\rangle$ for all $x \in \M$. Whenever $\Omega_\phi$ is separating for $\M_\phi$, let $S_\phi$ denote the unique closed conjugate-linear operator satisfying $S_\phi x\Omega_\phi = x^\ast \Omega_\phi$ for all $x \in \M_\phi$. Let $S_\phi = J_\phi\Delta_\phi^{1\over 2}$ be its polar decomposition, where the operators $\Delta_\phi$ and $J_\phi$ are positive and anti-unitary, respectively.
	
	\begin{definition}\label{def: cts_aut_gps_vna}
		A map $\sigma : \R \to \Aut(\M)$ is said to be $\sigma(\M_\ast, \M)$-continuous if for every $x \in \M$, the map $\R \to \M,\; t \mapsto \sigma_t(x)$ is continuous w.r.t.\ the $\sigma(\M_\ast, \M)$-topology on $\M$.
	\end{definition}

	\begin{definition}\label{def: mod_condition_kms}
		Let $\phi \in \mc{S}(\M)$ be a normal state. Let $\sigma : \R \to \Aut(\M)$ be a one-parameter group of automorphisms of $\M$. Define $\St := \set{ z \st z \in \C, \; 0 < \mrm{Im}(z) < 1}$.
		\begin{itemize}
			\item $\phi$ is said to satisfy the \textit{modular condition} for $\sigma$ if the following two conditions are satisfied:
			\begin{enumerate}
				\item $\phi = \phi \circ \sigma_t$ for all $t \in \R$. 
				\item For every $x,y \in \M$, there exists a bounded continuous function $F_{x,y} : \overline{\St} \to \C$ which is holomorphic on $\St$ and s.t.\ for every $t \in \R$:
				\begin{align*}
					F_{x,y}(t) &= \phi(\sigma_t(x)y), \\
					F_{x,y}(t+i) &= \phi(y\sigma_t(x)). 
				\end{align*}
			\end{enumerate}
			\item $\phi$ is said to be \textit{KMS} w.r.t. $\sigma$ at inverse temperature $\beta > 0$ if it satisfies the modular condition for $t \mapsto\sigma_{-\beta t}$. In that case, we also say that $\phi$ is $\sigma$-KMS at inverse-temperature $\beta$. If $\beta = 1$ we simply say that $\phi$ is a $\sigma$-KMS state.
		\end{itemize}
	\end{definition}

	\begin{remark}\label{rem: kms_states_and_mod_groups_uniqueness}~
		\begin{enumerate}
			\item Suppose that $\phi \in \mc{S}(\M)$ is faithful. Then there exists a unique automorphism group $\sigma^\phi : \R \to \Aut(\M)$ for which $\phi$ satisfies the modular condition \cite[Thm.\ VIII.1.2]{Takesaki_II}, \cite[Thm.\ 2.5.14]{bratelli_robinson_1}. The automorphism group $\sigma^\phi$ is $\sigma(\M_\ast, \M)$-continuous. As $\phi$ is faithful, $\pi_\phi : \M \to \M_\phi$ is injective and hence a $\ast$-isomorphism between $\M$ and $\M_\phi$ \cite[Thm.\ 2.4.24]{bratelli_robinson_1}. Thus one may identify $\M$ with $\M_\phi$ via $\pi_\phi : \M \to \M_\phi \subseteq \B(\mcH_\phi)$. Finally, there is a unique conditional expectation $\mc{E} : \M\to \M^\R$ s.t.\ $\phi = \phi_0 \circ \mc{E}$, where $\M^\R := \set{x \in \M \st \sigma_t^\phi(x) = x \quad \forall t \in \R}$ and $\phi_0 := \restr{\phi}{\M^\R}$ \cite[Thm.\ IX.4.2]{Takesaki_II}. 
			\item If $\phi$ is not necessarily faithful, then $\phi$ satisfies the modular condition for some $\sigma(\M_\ast, \M)$-continuous $1$-parameter group $\sigma: \R \to \Aut(\M)$ of $\ast$-automorphisms of $\M$ if and only if $\Omega_\phi \in \mcH_\phi$ is separating for $\M_\phi = \pi_\phi(\M)^{\prime \prime} \subseteq \B(\mcH_\phi)$. In that case, there is a central projection $p \in \mcZ(\M)$ such that $\phi(1-p) = 0$ and $\phi$ is faithful on $\M p$. Moreover, $\sigma_t(p) = p$ for all $t \in \R$ and $\restr{\sigma}{\M p}$ is uniquely determined by the modular condition for $\phi$ \cite[Thm.\ 5.3.10]{bratelli_robinson_2}.
			\item In particular, if $\M$ is a factor and $\phi$ is KMS w.r.t. $\sigma : \R \to \Aut(\M)$, then necessarily $p=I$ and whence $\phi$ must be faithful. Consequently $\sigma = \sigma_t^\phi$ is necessary.
			\item In the converse direction, given a $\sigma(\M_\ast, \M)$-continuous automorphism group $\sigma : \R \to \Aut(\M)$, there may be no, precisely one, or multiple states in $\mc{S}(\M)$ that are KMS w.r.t. $\sigma$. The set of $\sigma$-KMS states in $\mc{S}(\M)$ is considered in \cite[Ch. 5.3.2]{bratelli_robinson_2}. In particular, if $\phi \in \mc{S}(\M)$ is a faithful $\sigma$-KMS state and $\psi \in \mc{S}(\M)$, then $\psi$ is $\sigma$-KMS if and only if there is a (necessarily unique) positive operator $T$ affiliated to $\mc{Z}(\M)$ such that $\psi(x) = \phi(T^{1\over 2}xT^{1\over 2})$ for all $x \in \M$ \cite[Prop.\ 5.3.29]{bratelli_robinson_2}. In \cite{Bratelli_temp_states_I} and \cite{Bratelli_temp_states_II}, the set $K_\beta$ of normal $\sigma$-KMS states at inverse temperature $\beta$ is studied in the setting of $C^\ast$-dynamical systems.
			\item As a consequence of the previous points, if $\M$ is a factor and $\phi, \psi \in \mc{S}(\M)$ are both $\sigma$-KMS, then $\phi = \psi$, so that two distinct normal states can not share the same modular automorphism group.
		\end{enumerate}
	\end{remark}

	\begin{remark}\label{rem: kms_states_and_mod_groups}~
		Suppose $\phi \in \mc{S}(\M)$ is KMS w.r.t. $\sigma : \R \to \Aut(\M)$. Let $\sigma^\phi : \R \to \Aut(\M_\phi)$ denote the modular automorphism group defined by the faithful state $\langle \Omega_\phi, \bcdot\; \Omega_\phi\rangle$ on $\M_\phi = \pi_\phi(\M)^{\prime \prime}$. It then holds true that $\sigma_t^\phi(\pi_\phi(x)) = \pi_\phi(\sigma_{-t}(x))$ for any $x \in \M$ and $t \in \R$. Indeed, by \cite[Cor.\ 5.3.4]{bratelli_robinson_2}, the state $\langle \Omega_\phi, \bcdot\; \Omega_\phi\rangle$ on $\M_\phi$ is KMS w.r.t. the unique automorphism group $\tau : \R \to \Aut(\M_\phi)$ satisfying $\tau_{t}(\pi_\phi(x))\Omega_\phi = \pi_\phi(\sigma_t(x))\Omega_\phi$ for all $t \in \R$. Then $\sigma_t^\phi = \tau_{-t}$ by uniqueness of the modular automorphism group (and the minus sign in the definition of KMS states). As $\Omega_\phi$ is separating for $\M_\phi$, it follows that $\sigma_t^\phi(\pi_\phi(x)) = \pi_\phi(\sigma_{-t}(x))$.
	\end{remark}

	\begin{example}[Gibbs States]\label{ex: type_I_case}
		Let $\M = \B(\mcH)$ and $\sigma_t(x) = e^{itH}xe^{-itH}$ for some self-adjoint operator $H$ satisfying $Z_\beta := \Tr(e^{-\beta H}) < \infty$ for some $\beta > 0$. Consider the normal state $\phi(x) = {1\over Z_\beta}\Tr(e^{-\beta H}x)$ on $\M$. The modular automorphism group corresponding to $\phi$ is given by $\sigma_t^{\phi}(x) = e^{-i\beta tH}xe^{i\beta t H} = \sigma_{-\beta t}(x)$ \cite[Example 2.5.16]{bratelli_robinson_1}. Thus $\phi$ satisfies the modular condition for $\sigma_{-\beta t}$ and is therefore KMS at inverse-temperature $\beta$ w.r.t. $\sigma_t$.
	\end{example}

	\noindent
	Gibbs states $\phi(x) = {1\over Z_\beta}\Tr(e^{-\beta H}x)$ constitute the simplest class of examples of KMS states. We will encounter a variety of different KMS states in \Fref{sec: examples_kms} below.

	\subsection{KMS Representations}\label{sec: kms_rep}

	\noindent
	In the following, let $G$ be a regular locally convex Lie group with Lie algebra $\g$. Let $N \subseteq G$ be an embedded Lie subgroup. 
	
	\begin{definition}\label{def: KMS_reps}
		Let $(\rho, \mcH_\rho)$ be a continuous unitary $G$-representation. Let $\mc{N}:= \rho(N)^{\prime \prime} \subseteq \B(\mcH_\rho)$ be the von Neumann-algebra generated by $\rho(N)$. For $\phi \in \mc{N}_\ast$, define the function $\widehat{\phi} : N \to \C$ by $\widehat{\phi}(n) := \phi(\rho(n))$. Write $\mc{N}_\ast^\infty := \set{\phi \in \mc{N}_\ast \st \widehat{\phi} \in C^\infty(N; \C)}$ and set $\mc{S}(\mc{N})^\infty := \mc{S}(\mc{N}) \cap \mc{N}_\ast^\infty$.
		\begin{itemize}
			\item Let $\xi \in \g$ and $\phi \in \mc{S}(\mc{N})$. We say that $\phi$ is \textit{KMS-compatible} with $(\rho, \xi, N)$ if $e^{t\xi}Ne^{-t\xi} \subseteq N$ for all $t \in \R$ and $\phi$ is KMS w.r.t. the automorphism group $\R \to \Aut(\mc{N})$ defined by $t \mapsto \Ad(\rho(e^{t\xi}))$.
			\item Define $\KMS(\rho, \xi, N) := \set{\phi \in \mc{S}(\mc{N}) \st \phi \text{ is KMS-compatible with } (\rho, \xi, N)}$.\\
			Similarly, let $\KMS(\rho, \xi, N)^\infty := \KMS(\rho, \xi, N) \cap \mc{S}(\mc{N})^\infty$.
			\item $\rho$ is said to be \textit{KMS at} $\xi \in \g$\textit{ relative to} $N$ if $\KMS(\rho, \xi, N) \neq \emptyset$. \\
			It is called \textit{smoothly-KMS at} $\xi$ \textit{relative to} $N$ if $\KMS(\rho, \xi, N)^\infty \neq \emptyset$.
		\end{itemize}
		If the subgroup $N$ is clear from the context, we drop $N$ from the notation and simply write $\KMS(\rho, \xi)$ and $\KMS(\rho, \xi)^\infty$. We then also say that $\rho$ is KMS at $\xi$ if it is so relative to $N$.\\
	\end{definition}

	\begin{remark}\label{rem: KMS-semidirect_prod}
		For any fixed $\xi \in \g$ satisfying $\Ad(e^{t\xi})N \subseteq N$ for all $t \in \R$, one may as well consider the semidirect product $N \rtimes_\alpha \R$, where $\alpha : \R \to \Aut(N)$ is defined by $\alpha_t := \restr{\Ad(e^{t\xi})}{N}$. \Fref{def: KMS_reps} additionally allows for the situation where $\rho$ is KMS at multiple $\xi_I \in \g$, relative to possibly distinct subgroups $N_I \subseteq G$, where $I \in \mc{I}$ for some indexing set $\mc{I}$. We will see an example of this in \Fref{ex: mobius_cov_nets} below.\\
	\end{remark}

	\noindent
	In the following, let $(\rho, \mcH_\rho)$ be a continuous unitary $G$-representation and let $\mc{N}:= \rho(N)^{\prime \prime} \subseteq \B(\mcH_\rho)$ be the von Neumann-algebra generated by $\rho(N)$. If $\phi \in \mc{S}(\mc{N})$, write $\pi_\phi : \mc{N} \to \B(\mcH_\phi)$ for the GNS-representation of $\mc{N}$ relative to $\phi$. Let $\Omega_\phi \in \mcH_\phi$ denote the canonical $\mc{N}$-cyclic vector satisfying $\phi(x) = \langle \Omega_\phi, \pi_\phi(x)\Omega_\phi\rangle$ for all $x \in \mc{N}$. Write $\rho_\phi := \pi_\phi \circ \rho : N \to \U(\mcH_\phi)$ for the unitary $N$-representation on $\mcH_\phi$. Define $\mc{N}_\phi := \rho_\phi(N)^{\prime \prime}\subseteq \B(\mcH_\phi)$.
	
	\begin{lemma}\label{lem: smooth_state_vector}
		Let $\phi \in \mc{S}(\mc{N})$. Then $\widehat{\phi}$ is smooth on $N$ if and only if $\Omega_\phi \in \mcH_{\rho_\phi}^\infty$. In this case $\mcH_{\rho_\phi}^\infty$ is dense, so $\rho_\phi$ is smooth.
	\end{lemma}
	\begin{proof}
		Assume that $\widehat{\phi}$ is smooth on $N$. Then $n \mapsto \langle \Omega_\phi, \rho_\phi(n)\Omega_\phi\rangle$ is smooth. By \cite[Thm.\ 7.2]{Neeb_diffvectors}, it follows $n \mapsto \rho_\phi(n)\Omega_\phi$ is smooth $N \to \mcH_\phi$. The converse direction is trivial. Assume that $\Omega_\phi \in \mcH_{\rho_\phi}^\infty$. As $\mcH_{\rho_\phi}^{\infty}$ is $N$-invariant and $\Omega_\phi$ is cyclic for $N$, it follows that $\mcH_{\rho_\phi}^\infty$ is dense in $\mcH_{\phi}$. 
	\end{proof}

	\noindent
	Consider the left action of $G$ on $\mc{S}(\mc{N})$ defined by $(g.\phi)(x) := \phi(\rho(g)^{-1}x\rho(g))$ for $x \in \mc{N}$. Notice that this action leaves $\mc{S}(\mc{N})^\infty$ invariant.
	
	\begin{lemma}\label{lem: covariance_kms_states}
		Let $g \in G$ and $\xi \in \g$. Then $\phi \in \KMS(\rho, \xi, N) \iff g.\phi \in \KMS(\rho, \Ad_g(\xi), gNg^{-1})$.
	\end{lemma}
	\begin{proof}
		Write $\mc{N}_g := \rho(g)\mc{N}\rho(g)^{-1}$. Let $\phi \in \KMS(\rho, \xi, N)$. As $e^{t\xi}Ne^{-t\xi} \subseteq N$ it follows that $e^{t\Ad_g(\xi)}$ normalizes $gNg^{-1}$ for every $t\in \R$. Define the following automorphism groups:
		\begin{alignat*}{2}
			\sigma^\xi : \R &\to \Aut(\mc{N}),& \qquad \sigma^\xi &:= \Ad(\rho(e^{t\xi})),\\
			\eta^\xi : \R &\to \Aut(\mc{N}_g),& \qquad \eta^\xi &:= \Ad(\rho(e^{t\Ad_g\xi})).
		\end{alignat*}
		In order to show $g.\phi \in \KMS(\rho, \Ad_g(\xi), gNg^{-1})$, we must verify that $g.\phi$ satisfies the modular condition for the automorphism group $\eta_{-t}^\xi$ of $\mc{N}_g$. Notice that as isomorphisms $\mc{N}_g \to \mc{N}$ we have
		\begin{equation}\label{eq: intermediate_equivariance_of_action}
			\sigma_t^\xi \circ \Ad(\rho(g)^{-1}) = \Ad(\rho(g)^{-1}) \circ \eta_t^\xi, \qquad \forall t \in \R.
		\end{equation}
		As $\phi \in \KMS(\rho, \xi, N)$, we know that $\phi \circ \sigma_t^\xi = \phi$ for all $t \in \R$. It then follows immediately from \eqref{eq: intermediate_equivariance_of_action} that 
		\[(g.\phi) \circ \eta_t^\xi = \phi \circ \Ad(\rho(g)^{-1}) \circ \eta_t^\xi  = \phi \circ \sigma_t^\xi \circ \Ad(\rho(g)^{-1}) = \phi \circ  \Ad(\rho(g)^{-1})  = g.\phi, \qquad \forall t \in \R. \]
		Next, take $x,y \in \mc{N}_g$. Then $x = \rho(g)x^\prime \rho(g)^{-1}$ and $y = \rho(g)y^\prime \rho(g)^{-1}$ for some $x^\prime,y^\prime \in \mc{N}$. Let the function $F_{x^\prime,y^\prime} : \overline{\St} \to \C$ be continuous and bounded, holomorphic on $\St$ and satisfy $F_{x^\prime, y^\prime}(t) = \phi(\sigma_{-t}^\xi(x^\prime)y^\prime)$ and $F_{x^\prime, y^\prime}(t + i) = \phi(y^\prime\sigma_{-t}^\xi(x^\prime))$ for all $t \in \R$. Define $\widetilde{F}_{x,y} : \overline{\St} \to \C$ by $\widetilde{F}_{x,y}(z) := F_{x^\prime, y^\prime}(z)$. Then $\widetilde{F}_{x,y}$ satisfies the conditions of \Fref{def: mod_condition_kms} for $\eta_{-t}^\xi$. Indeed, notice using \Fref{eq: intermediate_equivariance_of_action} that $\sigma_{t}^\xi(x^\prime) = \rho(g)^{-1}\eta_t^\xi(x)\rho(g)$. Thus
		\begin{align*}
			\widetilde{F}_{x,y}(t) = F_{x^\prime, y^\prime}(t) = \phi(\sigma_{-t}^\xi(x^\prime)y^\prime) = \phi\bigg(\rho(g)^{-1}\eta_{-t}^\xi(x)y\rho(g)\bigg) = (g.\phi)(\eta_{-t}^{\xi}(x)y),\\
			\widetilde{F}_{x,y}(t + i) = F_{x^\prime, y^\prime}(t + i) = \phi(y^\prime\sigma_{-t}^\xi(x^\prime)) = \phi\bigg(\rho(g)^{-1}y\eta_{-t}^{\xi}(x)\rho(g)\bigg) = (g.\phi)(y\eta_{-t}^{\xi}(x)).
		\end{align*}
		Thus $g.\phi \in \KMS(\rho, \Ad_g(\xi), gNg^{-1})$.\qedhere \\
	\end{proof}

	\noindent
	Let $\phi \in \KMS(\rho, \xi, N)$. Let $\alpha$ denote the smooth $\R$-action on $N$ defined by $\alpha_t(n) := e^{t\xi}ne^{-t\xi}$. We extend $\rho_\phi$ to $N \rtimes_{\alpha} \R$ by setting $\rho_\phi(n,t) = \rho_\phi(n)\Delta_\phi^{-it}$. Define
	
	\begin{equation}\label{eq: kms_domains}
		\begin{aligned}
			\mc{N}^{\infty, \phi} &:= \set{ x \in \mc{N} \st (n,t) \mapsto \rho_\phi(n,t)\pi_\phi(x)\Omega_\phi \text{ is smooth }N \rtimes_{\alpha} \R \to \mcH_\phi},\\
			\mcD_\phi &:= \pi_\phi(\mc{N}^{\infty, \phi})\Omega_\phi \subseteq \mcH_{\rho_\phi}^\infty.
		\end{aligned}
	\end{equation}
	Notice that $\mc{N}^{\infty, \phi}$ and $\mcD_\phi$ are invariant under the left $N$- and $N \rtimes_{\alpha} \R$-actions, respectively.\\
	
	\begin{lemma}\label{lem: smooth_kms_domain_dense}
		If $\phi \in \KMS(\rho, \xi, N)^\infty$, then $\mc{N}^{\infty, \phi}$ is SOT-dense in $\mc{N}$ and $\mcD_\phi$ is dense in $\mcH_\phi$.\\
		In particular, $\rho_\phi$ is smooth when considered as representation of $N \rtimes_{{\alpha}} \R$.
	\end{lemma}
	\begin{proof}
		Since $\phi \in \mc{S}(\mc{N})^\infty$, the vector $\Omega_\phi$ is smooth for the $N$-action $\rho_\phi$ by \Fref{lem: smooth_state_vector}. Let $m \in N$. Then
		\[ \rho_\phi(n,t)\rho_\phi(m)\Omega_\phi = \rho_\phi(n)\Delta_\phi^{-it} \rho_\phi(m)\Delta_{\phi}^{it}\Omega_\phi = \rho_\phi(n)\sigma_{-t}^\phi(\rho_\phi(m))\Omega_\phi = \rho_\phi(ne^{t\xi}me^{-t\xi})\Omega_\phi, \qquad \forall n \in N, \; t \in \R,\]
		where the last equality follows by \Fref{rem: kms_states_and_mod_groups}. Thus $(n,t)\mapsto \rho_\phi(n,t)\rho_\phi(m)\Omega_\phi$ is smooth $N \rtimes_{\alpha} \R \to \mcH_\phi$ and so $\rho(m) \in \mc{N}^{\infty, \phi}$. Thus $\rho(N) \subseteq \mc{N}^{\infty, \phi}$ and $\rho_\phi(N)\Omega_\phi \subseteq \mcD_\phi$. Since $\rho(N)^{\prime \prime} = \mc{N}$ and $\rho_\phi(N)\Omega_\phi$ is total for $\mcH_\phi$, it follows that $\mc{N}^{\infty, \phi}$ is SOT-dense in $\mc{N}$ and that $\mcD_\phi$ is dense in $\mcH_\phi$. As $\mcD_\phi$ is contained in the set of $N\rtimes_{\alpha}\R$-smooth vectors by definition, the final observation follows.
	\end{proof}

		\subsubsection{Restrictions Imposed by the KMS Condition}
	
	\noindent
	Let us next determine some consequences of the KMS condition. Most notably, we will show that representations $\rho$ which are smoothly-KMS give rise to generalized positive energy representations $\rho_\phi$ on the GNS-Hilbert space $\mcH_\phi$ of the corresponding state $\phi$.\\
	
	\noindent
	We continue in the notation of \Fref{sec: kms_rep}. Fixing a Lie subgroup $N \subseteq G$ and some element $\xi \in \g$ satisfying $\Ad(e^{t\xi})N \subseteq N$ for all $t \in \R$, we may as well suppose that $G = N \rtimes_\alpha \R$ for some smooth $\R$-action $\alpha$ on $N$ by automorphisms. Let $\g := \Lie(G)$, $\n := \Lie(N)$ and write $D \in \der(\n)$ for the derivation on $\n$ corresponding to $\alpha$. Thus $\g = \n \rtimes_D \R \bm{d}$, where $\bm{d} := 1 \in \R$ denotes the standard basis element. Assume that $\rho$ is KMS at $\bm{d}$ relative to $N$, and let $\phi \in \KMS(\rho, \bm{d}, N)$. We extend the $N$-representation $\rho_\phi = \pi_\phi \circ \rho$ on the GNS-Hilbert space $\mcH_\phi$ to $G = N \rtimes_{\alpha} \R$ by setting $\rho_\phi(n,t) = \rho_\phi(n)\Delta_\phi^{-it}$. Define further $H_\phi := -\log \Delta_\phi = -i \overline{d\rho_\phi(\bm{d})}$. \\
	
	\noindent
	A first observation is the following:
	
	\begin{proposition}\label{prop: abelian_invariant_implies_trivial_kms}
		Let $A$ be an Abelian Lie subgroup of $N$ such that $\alpha_t(A) \subseteq A$ for all $t \in \R$. \\
		Then $\rho_\phi(\alpha_t(a)) = \rho_\phi(a)$ for every $t \in \R$ and $a \in A$. In particular, if $\mc{N}$ is a factor then $\rho(\alpha_t(a)) = \rho(a)$ for every $t \in \R$ and $a \in A$.
	\end{proposition}
	\begin{proof}
		Let $\A_\phi := \rho_\phi(A)^{\prime \prime}$. Write again $\phi$ for the vector state $\langle \Omega_\phi, \bcdot\; \Omega_\phi\rangle$ on $\mc{N}_\phi$. Let $\psi := \restr{\phi}{\A_\phi}$ denote its restriction to $\A_\phi$. As $A$ is $\R$-invariant, so is $\A_\phi \subseteq \mc{N}_\phi$. Thus, the modular automorphism group $\sigma^\phi$ of $\mc{N}_\phi$ leaves $\A_\phi$ invariant. As $\phi$ satisfies the modular condition for $\sigma^\phi$, so does $\psi$ for the automorphism group $t \mapsto\restr{\sigma_t^\phi}{\A_\phi}$. Recall from \Fref{rem: kms_states_and_mod_groups_uniqueness}(2) that $\Omega_\phi$ is separating for $\mc{N}_\phi$. Hence it is so for $\A_\phi$. In view of \Fref{rem: kms_states_and_mod_groups_uniqueness}(1) this implies that the modular automorphism group $\sigma^{\psi}$ on $\A_\phi$ is uniquely determined by the modular condition. Thus $\sigma_t^{\psi} = \restr{\sigma_t^\phi}{\A_\phi}$ for all $t \in \R$. As $\A_\phi$ is Abelian, we know by \cite[Prop.\ 5.3.28]{bratelli_robinson_2} that $\sigma_t^{\psi} = \id_{\A_\phi}$. Thus $\restr{\sigma_t^\phi}{\A_\phi} = \id_{\A_\phi}$. We know from \Fref{rem: kms_states_and_mod_groups} that $\rho_\phi \circ \alpha_{-t} = \sigma_t^\phi \circ \rho_\phi$. Thus $\rho_\phi(\alpha_t(a)) = \rho_\phi(a)$ for all $a \in \A$ and $t \in \R$. If $\mc{N}$ is a factor, then $\phi$ is faithful and $\pi_\phi$ is injective by \Fref{rem: kms_states_and_mod_groups_uniqueness}(1,3). Thus $\rho(\alpha_t(a)) = \rho(a)$ follows from $\rho_\phi(\alpha_t(a)) = \rho_\phi(a)$.
	\end{proof}
	
	\noindent
	Let us illustrate \Fref{prop: abelian_invariant_implies_trivial_kms} with the following noteworthy consequence for loop groups:
	
	\begin{corollary}\label{cor: loop_group_no_kms}
		Let $K$ be a compact $1$-connected simple Lie group with Lie algebra $\mfr{k}$. Define $LK := C^\infty(S^1; K)$ and $L\mfr{k} := C^\infty(S^1; \mfr{k})$. Let $\alpha$ denote the $\T$-action on $LK$ by rotations, with corresponding derivation $D := {d\over d\theta}$ on $L\mfr{k}$. Consider the Lie group $G := LK \rtimes_\alpha \T$ with Lie algebra $\g := L\mfr{k} \rtimes_D \R \bm{d}$, where $\bm{d} := 1 \in \R$. Suppose that the smooth unitary $G$-representation $\rho$ is KMS at $\bm{d} \in \g$ relative to $LK$. Assume that $\rho(LK)^{\prime \prime}$ is a factor. Then $LK \subseteq \ker \rho$.
	\end{corollary}
	\begin{proof}
		Suppose $T \subseteq K$ is a maximal torus with Lie algebra $\mfr{t}$. Then $LT \subseteq LK$ is an Abelian $\alpha$-invariant subgroup. By \Fref{prop: abelian_invariant_implies_trivial_kms} it follows that $d\rho(DL\mfr{t}) = \{0\}$. As any $X \in \mfr{k}$ is contained in a maximal torus, it follows that $d\rho({df\over d\theta}\otimes X) = 0$ for any $f \in C^\infty(S^1)$ and $X \in \mfr{k}$. Consequently $d\rho(D\g) = \{0\}$ and hence $D \g_\C \subseteq \ker d\rho$, where we have extended $d\rho : \g \to \mc{L}^\dagger(\mcH_\rho^\infty)$ $\C$-linearly to the complexification $\g_\C$. As $\ker d\rho$ is an ideal in $\g_\C$ and $L\mfr{k}_\C = DL\mfr{k}_\C + [DL\mfr{k}_\C, DL\mfr{k}_\C]$, it follows that $L\mfr{k}_\C \subseteq \ker d\rho$. Notice that $LK$ is connected because $K$ is $1$-connected. It is also locally exponential by \cite[Thm.\ II.1]{Neeb_borel_weil_loop_groups}. It follows that $LK \subseteq \ker \rho$.
	\end{proof}

	\noindent
	Thus, one necessarily has to pass to a non-trivial central $\T$-extension $\circled{LK}$ of $LK \rtimes_\alpha \T$ to allow for interesting KMS-representations of $\circled{LK}$ that are smoothly-KMS at some $\circled{\bm{d}} \in \circled{L\mfr{k}}$ covering $\bm{d} \in L\mfr{k} \rtimes_D \R \bm{d}$, as one may have expected from the positive energy analogue (which follows from \cite[Thm 9.3.5]{Segal_Loop_Groups}).\\

	\noindent
	We now proceed with the observation that KMS representations give rise to generalized positive energy representations on the GNS-Hilbert space corresponding to the KMS state:
	
	\begin{theorem}\label{thm: kms_of_qpe}
		Let $\phi \in \KMS(\rho, \bm{d}, N)^\infty$. Let $x \in \mcN^{\phi, \infty}$ and assume $\psi := \pi_\phi(x)\Omega_\phi \in \mcD_\phi$ has unit norm. Then
		\begin{equation}\label{eq: unif_estimate_kms}
			\langle\pi_\phi(x)\Omega_\phi,  -id\rho_\phi(\Ad_n(\bm{d}))\pi_\phi(x)\Omega_\phi\rangle \geq -\log\big(\norm{\pi_\phi(x)}^2\big) \qquad \forall n \in N.
		\end{equation}
		In particular the representation $\rho_\phi$ of $N \rtimes_{\alpha} \R$ on $\mcH_\phi$ is of generalized positive energy at $\bm{d} \in \n \rtimes_D \R \bm{d}$.
	\end{theorem}
	
	\begin{lemma}
		Let $x \in \mc{N}$ be such that $0\neq \psi := \pi_\phi(x) \Omega_\phi \in \dom(H_\phi)$. Then 
		\begin{equation}\label{eq: KMS_entropy_inequality}
			{\langle \psi, H_\phi\psi\rangle \over \norm{\psi}^2 } \geq -\log\bigg({ \norm{S_\phi \psi}^2 \over \norm{\psi}^2}\bigg).	
		\end{equation}
	\end{lemma}
	\begin{proof}
		In view of the correlation lower bounds satisfied by KMS states, see e.g.\ \cite[Thm.\ 5.3.15 $\textrm{(1)} \implies \textrm{(2)}$]{bratelli_robinson_2} or \cite[Thm.\ II.4, $\textrm{(i)} \implies \textrm{(iii)}$]{FannesVerbeure_KMS_correlation_inequalities}, we have
		\[ \langle \pi_\phi(x)\Omega_\phi, [H_\phi, \pi_\phi(x)]\Omega_\phi\rangle \geq -\norm{\pi_\phi(x)\Omega_\phi}^2\log\bigg({ \norm{\pi_\phi(x)^\ast\Omega_\phi}^2 \over \norm{\pi_\phi(x)\Omega_\phi}^2}\bigg). \]
		Since $H_\phi\Omega_\phi = 0$, it follows that $\langle \pi_\phi(x)\Omega_\phi, [H_\phi, \pi_\phi(x)]\Omega_\Phi\rangle = \langle \pi_\phi(x)\Omega_\phi, H_\phi \pi_\phi(x)\Omega_\phi\rangle$. The assertion follows.\qedhere 
	\end{proof}

	\begin{proof}[Proof of \Fref{thm: kms_of_qpe}:]~\\
		Recall that $\mc{D}_\phi \subseteq \mcH_{\rho_\phi}^\infty$ and that $\mc{D}_\phi \subseteq \dom(S_\phi)$, because the stronger condition $\mc{N}_\phi \Omega_\phi \subseteq \dom(S_\phi)$ is satisfied. Let $n \in N$. Notice that $\norm{S_\phi \rho_\phi(n)\psi} = \norm{\pi_\phi(x^\ast)\rho_\phi(n)^{-1}\Omega_\phi} \leq \norm{\pi_\phi(x)}$. Recalling that $\mcD_\phi$ is $N$-invariant, we can apply \fref{eq: KMS_entropy_inequality} to the vector $\rho_\phi(n)\psi$. Using $-i \overline{d\rho_\phi(\bm{d})} = -\log \Delta_\phi = H_\phi$ it follows that
		\begin{equation*}
			\langle\psi,  -id\rho_\phi(\Ad_{n^{-1}}(\bm{d}))\psi\rangle
			= \langle\rho_\phi(n)\psi,  -id\rho_\phi(\bm{d})\rho_\phi(n)\psi\rangle
			\geq -\log\big(\norm{S_\phi \rho_\phi(n)\psi}^2\big)
			\geq -\log\big( \norm{\pi_\phi(x)}^2\big).\qedhere
		\end{equation*}
	\end{proof}

	\noindent
	As a consequence of \Fref{thm: kms_of_qpe}, we find that the observations of \Fref{sec: quas_pe} impose restrictions on KMS representations. Let us illustrate this with the following immediate consequence:

	\begin{corollary}
		Let $\overline{\rho}$ be a smooth projective unitary representation of $G$ on $\mcH_\rho$. Assume that $\mc{N} := \overline{\rho}(N)^{\prime \prime}$ is a factor. Let $\rho : \circled{G} \to \U(\mcH_\rho)$ be the lift of $\overline{\rho}$, for some central $\T$-extension $\circled{G}$ of $G$ with Lie algebra $\circled{\g}$. Let $\circled{N}\subseteq \circled{G}$ cover $N$. Let $\omega$ represent the class in $H^2_{\ct}(\g, \R)$ corresponding to $\circled{\g}$. Let $\xi \in \g$ and suppose $\circled{\xi} \in \circled{\g}$ covers $\xi$. Let $\phi \in \KMS(\rho, \circled{\xi},\circled{N})^\infty$. Assume that $\eta \in \n$ satisfies $[[\xi, \eta], \eta] = 0$. Then $\omega([\xi, \eta], \eta)\geq 0$ and
		\begin{align*}
			\omega([\xi, \eta], \eta) = 0 \iff d\overline{\rho}([\xi, \eta]) = 0.
		\end{align*}
	\end{corollary}
	\begin{proof}
		Consider the representation $\rho_\phi$ of $\circled{N} \rtimes \R$ on the GNS Hilbert space $\mcH_\phi$, where $\R$ acts on $\circled{N}$ by $\restr{\Ad(e^{t \circled{\xi}})}{\circled{N}}$ and where $\rho_\phi(1,t) = \Delta_\phi^{-it}$ for $t \in \R$. Let $\overline{\rho}_\phi$ be the corresponding projective unitary representation of $N \rtimes \R$ on $\mcH_\phi$, where $\R$ acts on $N$ by $\restr{\Ad(e^{t\xi})}{N}$. By \Fref{thm: kms_of_qpe}, $\rho_\phi$ is of g.p.e.\ at $\bm{d} \in \circled{\n} \rtimes \R \bm{d}$ and so $\overline{\rho}_\phi$ is of g.p.e.\ at $\bm{d}$. It follows from \Fref{prop: qpe_kernel} that $\omega([\xi, \eta], \eta)\geq 0$ and $\omega([\xi, \eta], \eta) = 0 \iff d\overline{\rho}_\phi([\xi, \eta]) = 0$. As $\mc{N}$ is a factor, the KMS state $\phi \in \mc{S}(\mcN)$ is faithful and the corresponding GNS-representation $\pi_\phi : \mc{N} \to \B(\mcH_\phi)$ is injective, by \Fref{rem: kms_states_and_mod_groups_uniqueness}(1,3). This implies that $\ker d\rho_\phi = \ker d\rho$. Thus $\omega([\xi, \eta], \eta) = 0 \iff d\overline{\rho}([\xi, \eta]) = 0$.
	\end{proof}

	\begin{remark}
		A related notation is that of a \textit{passive state}, which is usually considered in the context of a $C^\ast$-dynamical system $(\mc{A}, \sigma)$, where $\mc{A}$ is a $C^\ast$-algebra and $\sigma : \R \to \Aut(\mc{A})$ is a strongly continuous homomorphism. If $\delta$ is the generator of $\sigma$ with domain $\mcD(\delta) \subseteq \mc{A}$, a state $\phi$ on $\mc{A}$ is said to be \textit{passive} if 
		\begin{equation}\label{eq: passive_state}
			-i\phi(u^\ast \delta(u)) \geq 0, \qquad \forall u \in \U_0(\mc{A}) \cap \mcD(\delta),
		\end{equation}
		where $\U_0(\mc{A})$ denotes the identity component of the group $\U(\mc{A})$ of unitary elements in $\mc{A}$. In this case, $\phi$ is necessarily $\sigma$-invariant \cite[Thm.\ 1.1]{Pusz_passive_states}, so that $\sigma$ is canonically implemented by a strongly-continuous unitary $1$-parameter group $t \mapsto e^{it H_\phi}$ on the GNS-Hilbert space $\mcH_\phi$. Let $\pi_\phi : \A \to \B(\mcH_\phi)$ be the GNS-representation of $\mc{A}$ associated to $\phi$ and let $\Omega_\phi \in \mcH_\phi$ be the corresponding cyclic vector. Then \eqref{eq: passive_state} becomes
		\[ -i \langle \Omega_\phi, \pi_\phi(u)^{-1} H_\phi \pi_\phi(u) \Omega_\phi\rangle \geq 0, \qquad \forall u \in \U_0(\mc{A}) \cap \mcD(\delta), \]
		which is similar to \fref{eq: qpe_la}. It was moreover shown in \cite{Pusz_passive_states} that any ground- or $\sigma$-KMS state is necessarily passive (cf.\ \cite[Thm.\ 5.3.22]{bratelli_robinson_2}), which is analogous to the observation that both positive energy and KMS representations provide examples of generalized positive energy ones, in view of \Fref{thm: kms_of_qpe}. We refer to \cite{Pusz_passive_states} and \cite{bratelli_robinson_2} for more information on (completely) passive states.
	\end{remark}

	\subsubsection{Some Examples of KMS Representations}\label{sec: examples_kms}

	\noindent
	Let us consider a variety of examples of KMS representations, thereby showing in various situations that a well-known $\sigma$-KMS state $\phi$ on a von Neumann algebra $\mc{N}$ admits some underlying smooth structure. More precisely, we construct a continuous unitary representation $\rho$ of a (typically infinite-dimensional) Lie group $G$ such that $\mc{N} = \rho(N)^{\prime\prime}$, $\phi \in \KMS(\rho, \xi, N)^\infty$ and $\sigma_t = \Ad(\rho(e^{t\xi}))$ for some $\xi \in \g$ and Lie subgroup $N$ of $G$. In particular, in this case the $1$-parameter group $\sigma$ on $\mc{N}$ implements the $\R$-action $t\mapsto \restr{\Ad(e^{t\xi})}{N}$ on the Lie subgroup $N$ of $G$. \\
	
	\noindent
	Let us begin with the simplest class of examples, which correspond to Gibbs states, as in \Fref{ex: type_I_case}:

	\begin{example}\label{ex: Gibbs_KMS_reprs}
		Take for $N$ simply $N = G$. Let $(\rho, \mcH)$ be a continuous irreducible unitary $G$-representation. Then $\mc{N} = \B(\mcH)$. Let $\xi \in \g$ and define the self-adjoint operator $H := -i \restr{d\over dt}{t=0}\rho(e^{t\xi})$. Let $\beta > 0$ and assume that $Z_\beta := \Tr(e^{-\beta H}) < \infty$. Define the Gibbs state $\phi(x) := {1\over Z_\beta}\Tr(e^{-\beta H}x)$ for $x \in \mc{N}$. As in \Fref{ex: type_I_case}, we have $\sigma_{-t}^{\phi}(x) = e^{it\beta H}xe^{-it\beta H} = \rho(e^{t \beta \xi})x\rho(e^{-t \beta \xi})$ for any $x \in \mc{N}$. Consequently, $\phi \in \KMS(\rho, \beta \xi)$ and so $\rho$ is a KMS representation at $\beta \xi \in \g$. If in addition $\widehat{\phi} : G \to \C$ is smooth, then $\rho$ is smoothly-KMS at $\beta \xi \in \g$. By \Fref{lem: covariance_kms_states}, $\rho$ is also KMS at any element in the adjoint orbit of $\beta \xi$. In view of \Fref{ex: type_I_case}, the representation $\rho_\phi$ of $G \rtimes \R$ on $\mcH_\phi := \overline{\B(\mcH_\rho)}^{\langle \--, \--\rangle_\phi}$ is given by $\rho_\phi(g, t)x \Omega_\phi = \rho(g)\rho(e^{t\beta \xi})x\rho(e^{-t \beta \xi})\Omega_\phi = \rho(g)\sigma_{-t}^{\phi}(x)\Omega_\phi$, where $\Omega_\phi := I \in \B(\mcH_\rho) \subseteq \mcH_\phi$ denotes the cyclic vector.
	\end{example}

	\noindent
	In fact, \Fref{prop: gibbs_states_type_I_factor} below entails that any KMS representation $\rho$ for which $\mc{N}$ is a factor of type $\mathrm{I}$ is of the form described in \Fref{ex: Gibbs_KMS_reprs}. Moreover a complete characterization of such representations was very recently obtained in the context where $N$ is a finite-dimensional Lie group \cite{Tobias_typeI_factor_reps}.
	
	\begin{proposition}\label{prop: gibbs_states_type_I_factor}
		Let $\xi \in \g$ and $\beta > 0$. Suppose that $\restr{\rho}{N}$ is irreducible and that $\phi \in \KMS(\rho, \beta \xi, N)$. Let $H := -i \restr{d\over dt}{t=0}\rho(e^{t\xi})$. Then $Z_\beta := \Tr(e^{-\beta H}) < \infty$ and $\phi(x) = {1\over Z_\beta}\Tr(e^{- \beta H}x)$.
	\end{proposition}
	\begin{proof}
		As $\restr{\rho}{N}$ is irreducible, it follows that $\mc{N} = \B(\mcH_\rho)$. Thus $\phi(x) = \Tr(\delta x)$ for some $\delta \in L^1(\mcH_\rho)_+$ satisfying $\Tr(\delta) = 1$, where $L^1(\mcH_\rho)$ denotes Banach space of trace-class operators on $\mcH_\rho$. Moreover, in view of \Fref{rem: kms_states_and_mod_groups_uniqueness}(3), we know that $\phi$ is faithful on $\mc{N}$. By assumption, $\phi$ satisfies the modular condition for the automorphism group $t \mapsto \Ad(\rho(e^{-t \beta \xi})) =: \sigma_{-\beta t}$. On the other hand, as $\phi$ is faithful, there exists by \Fref{rem: kms_states_and_mod_groups_uniqueness}(1) a \textit{unique} automorphism group $\sigma_t^\phi$ of $\mc{N}$ for which $\phi$ satisfies the modular condition. It follows that $\sigma_{-\beta t} = \sigma_t^\phi$. When $\mc{N} = \B(\mcH_\rho)$ and $\phi(x) = \Tr(\delta x)$, the modular automorphism group $\sigma_t^\phi$ corresponding to $\phi$ is $\sigma_t^\phi(x) = \delta^{it}x\delta^{-it}$. In view of $\sigma_t^\phi = \sigma_{-t\beta}$,
		it follows that $\delta^{it}x \delta^{-it} = \rho(e^{-t\beta \xi})x \rho(e^{t\beta \xi})$ for every $x \in \mc{N}$. As $\mcZ(\mc{N}) = \C I$ and both  $t \mapsto \delta^{it}$ and $t \mapsto \rho(e^{t\beta \xi})$ are strongly continuous unitary $1$-parameter groups, it follows that there is some continuous homomorphism $c : \R \to \T$ such that $\delta^{it} = c(t)\rho(e^{-t\beta \xi}) = c(t)e^{-it \beta H}$ for all $t \in \R$. Thus there exists $\mu \in \R$ such that $\delta^{it} = e^{-it(\beta H + \mu I)}$ for all $t \in \R$. So $\log \delta = - (\beta H + \mu I)$. Since $\Tr(\delta) = 1$, we have $Z_\beta = \Tr(e^{-\beta H}) = \Tr(e^{-(\beta H + \mu)}e^{\mu}1) = \Tr(\delta e^{\mu}1) = e^{\mu}\phi(1) = e^\mu < \infty$. It follows that 
		\[{1\over Z_\beta}\Tr(e^{-\beta H}x) = e^{-\mu}\Tr(e^{-\beta H}x) = \Tr(e^{-(\beta H + \mu)}x) = \Tr(\delta x) = \phi(x), \qquad \forall x\in \B(\mcH).\qedhere\]
	\end{proof}

	\noindent
	For more interesting examples, one has to consider a Lie subgroup $N$ of $G$ which is not of type $\mathrm{I}$, so that the von Neumann algebra $\mc{N}$ need not be type $\mathrm{I}$.

	\begin{example}[Powers' factors]
		Define $G_n := \prod_{k=1}^n \SU(2)$ and let $\eta_n : G_n \hookrightarrow G_{n+1}$ be defined by 
		\[ \eta_n : G_n \xrightarrow{\id \times 1} G_n \times \SU(2) = G_{n+1}.\]
		Write $\g_n := \Lie(G_n)$ and $L(\eta_n) := \Lie(\eta_n)$. The direct limit $G := \varinjlim_{n} (G_n, \eta_n)$ consists of sequences $(u_k)$ in $\SU(2)$ with $u_k = 1$ for all but finitely many values of $k$. It can be equipped with the structure of a regular Lie group that is modeled on the locally convex inductive limit $\g := \varinjlim_{n} (\g_n, L(\eta_n))$ \cite[Thm.\ 4.3]{Glockner_direct_lim} and has the exponential map $\exp_G = \varinjlim_{n}\exp_{G_n}$ \cite[Prop.\ 4.6]{Glockner_direct_lim}. Let 
		$H := \begin{pmatrix}
			1 & 0 \\
			0 & -1
		\end{pmatrix}$ 
		and $\xi := iH \in \su(2)$. Consider the following $\R$-action $\alpha$ on $G$ defined by $(\alpha_t(u))_k := e^{t \xi}u_k e^{-t \xi}$ for $u \in G$. The corresponding action $\R \times G \to G$ is smooth. Indeed, the restriction of $\alpha$ to $\R\times G_n$ yields a smooth action $\alpha^{(n)} : \R \times G_n \to G_n$ for every $n \in \N$. It follows from \cite[Thm.\ 3.1]{Glockner_direct_lim} that $\varinjlim_n \alpha^{(n)} : \varinjlim_n (\R \times G_n) \to G$ is smooth. By \cite[Prop.\ 3.7]{Glockner_direct_lim} we further have $\varinjlim_n (\R \times G_n) = \R \times G$ as smooth manifolds. This shows that $\alpha : \R \times G \to G$ is smooth. Consider the Lie group $G^\sharp := G \rtimes_\alpha \R$ with Lie algebra $\g^\sharp := \g \rtimes_D \R \bm{d}$, where $\bm{d} := (0,1)$. Using the so-called Powers' factors, we define unitary representations $\rho$ of $G^\sharp$ which are smoothly-KMS at $\bm{d}$ relative to $G \triangleleft G^\sharp$ and for which $\rho(G)^{\prime \prime}$ is a factor of type $\mathrm{III}_\lambda$ for arbitrary $\lambda \in (0,1)$. Define the finite-dimensional $C^\ast$-algebra $\M_n := \bigotimes_{k=1}^n \B(\C^2)$ for every $n \in \N$. Let $\beta > 0$. Define the state $\phi(x) := {1\over Z}\Tr(e^{-\beta H}x)$ on $\B(\C^2)$, where $Z := \Tr(e^{- \beta H}) = 2\cosh(\beta)$. Let $\phi_n$ be the state on $\M_n$ defined by $\phi_n(x_1 \otimes \cdots \otimes x_n) = \prod_{k=1}^n \phi(x_k)$. The GNS-representation of $\B(\C^2)$ defined by $\phi$ is $\mcH_\phi := \B(\C^2)$ equipped with left $\B(\C^2)$-action and the inner product $\langle a,b\rangle := {1\over Z}\Tr(e^{-\beta H}a^\ast b)$. Similarly the GNS-representation of $\M_n$ corresponding to $\phi_n$ is $\mcH_{\phi_n} := \bigotimes_{k=1}^n \mcH_{\phi}$. The isometric inclusions $\mcH_{\phi_n} \hookrightarrow \mcH_{\phi_{n+1}}, x\mapsto x \otimes 1$ define a directed system of Hilbert spaces, and the algebraic direct limit $\varinjlim_{n} \mcH_{\phi_n}$ becomes naturally a pre-Hilbert space. Let $\mcH$ denote its Hilbert space completion. Let $\iota_{n} : \mcH_{\phi_n} \hookrightarrow \mcH$ denote the canonical inclusion. For every $n\in \N$, there is a $\ast$-representation $\pi_n$ of $\M_n$ on $\mcH$ defined for $x = x_1 \otimes \cdots \otimes x_n \in \M_n$ by 
		\[\pi_n(x) \iota_{m}(\psi_1\otimes \cdots \otimes \psi_m) := \iota_{m}(x_1\psi_1\otimes \cdots x_n \psi_n \otimes \psi_{n+1}\otimes \cdots \otimes \psi_m), \qquad m \geq n.\]
		Let $\M_\infty := \bigg(\bigcup_{n\in \N}\pi_n(\M_n)\bigg)^{\prime \prime}$. The vector $\Omega := 1 \otimes 1\otimes \cdots \in \mcH$ is cyclic and separating for $\M_\infty$ \cite[XIV, Prop.\ 1.11]{Takesaki_III}, so $\mcH$ may be identified with the GNS-representation of $\M_\infty$ w.r.t. the state $\phi_\infty := \langle \Omega, \bcdot\; \Omega \rangle$ on $\M_\infty$. Observe that $\phi_\infty$ satisfies $\phi_\infty(\iota_n(x)) = \phi_n(x)$ for all $n \in \N$ and $x \in \M_n$. The von Neumann algebra $(\M_\infty, \phi_\infty) =: \bigotimes_{k=1}^\infty (\B(\C^2), \phi)$ is the so-called Powers' factor with parameter $a := {e^{-\beta} \over 2\cosh(\beta)} \in (0, {1\over 2})$, which is a factor of type $\mathrm{III}_\lambda$ with $\lambda = e^{-2\beta} = {a \over 1-a} \in (0,1)$ \cite[XVIII, Theorem 1.1]{Takesaki_III}. The modular automorphism group $\sigma_t^{\phi_\infty}$ on $\M_\infty$ defined by $\phi_\infty$ is given by $\sigma_t^{\phi_\infty} = \bigotimes_{k=1}^\infty \Ad(e^{-\beta t \xi})$ \cite[XIV, Prop.\ 1.11]{Takesaki_III}, where $\bigotimes_{k=1}^\infty \Ad(e^{\beta t \xi}) \in \Aut(\M_\infty)$ satisfies $\bigotimes_{k=1}^\infty \Ad(e^{\beta t \xi}) \circ \iota_n = \iota_n \circ \bigotimes_{k=1}^n \Ad(e^{\beta t \xi})$ for all $t \in \R$ and $n \in \N$ and is defined from this condition by continuity, where we used \cite[XIV, Thm.\ 1.13]{Takesaki_III} and that $\phi \circ \Ad(e^{\beta t \xi}) = \phi$ for all $t \in \R$. Consider the unitary representation $\rho : G \rtimes_\alpha \R \to \U(\mcH)$ defined by
		\[\rho(u, \beta t) := \bigg(\bigotimes_{k=1}^\infty u_k\bigg) \circ \Delta_{\phi_\infty}^{-it}, \qquad u \in G, t \in \R\]
		which is well-defined because $u = (u_k) \in G$ is a sequence in $\SU(2)$ with $u_k = 1$ for all $k$ sufficiently large. Since $\rho(\beta t) = \Delta_{\phi_\infty}^{-it}$ and $\rho(G)^{\prime \prime} = \M_\infty$, it follows that $\rho$ is KMS at $\beta \bm{d} \in \g^\sharp$ relative to $G \triangleleft G^\sharp$. To see that $\widehat{\phi_\infty} : G \to \C$ is smooth, it suffices to show that its restriction to $G_n$ is smooth for every $n \in \N$, using the universal property of the smooth manifold structure on $G = \varinjlim_n G_n$ \cite[Thm.\ 3.1]{Glockner_direct_lim}. This is the case, as $\langle \Omega, \rho(u)\Omega\rangle =\prod_{k=1}^n \phi(u_k)$ for any $u \in G_n$, which is smooth $G_n \to \C$. Thus $\Omega \in \mcH_\rho^\infty$ and so $\rho$ is smoothly-KMS at $\beta \bm{d} \in \g^\sharp$ relative to $G \triangleleft G^\sharp$. 
	\end{example}

	\begin{example}[Standard real subspaces and Heisenberg representations]\label{ex: std_subspaces_heis_kms}
		Let $\mcH$ be a complex Hilbert space. Consider the real Heisenberg group $G := \mrm{H}(\mcH, \omega)$, where $\omega(v, w) = \mrm{Im} \langle v, w\rangle$. An $\R$-linear closed subspace $\mc{K} \subseteq \mcH$ is called \textit{cyclic} if $\mc{K} + i \mc{K}$ is dense in $\mcH$. It is called \textit{separating} if $\mc{K} \cap i \mc{K} = \{0\}$. A \textit{standard subspace} is a closed $\R$-linear subspace $\mc{K} \subseteq \mcH$ which is both cyclic and separating. We show that any standard real subspace gives rise to a smooth KMS representation. Let $\mc{K} \subseteq \mcH$ be a standard real subspace. Write $\delta_{\mc{K}}$ for the corresponding modular operator on $\mcH$, which is generally unbounded, positive and self-adjoint, see e.g.\ \cite[Sec. 3]{Neeb_Olafsson_mod_theory}. Then $t \mapsto \delta_{\mc{K}}^{it}$ is a strongly-continuous unitary $1$-parameter group on $\mcH$ satisfying in particular $\delta_{\mc{K}}^{it} \mc{K} \subseteq \mc{K}$. We first pass to the $\R$-smooth vectors $\mc{K}^\infty$ to obtain a regular Lie group $\mrm{H}(\mc{K}^\infty, \omega) \rtimes \R$. We then construct a KMS representation thereof using second-quantization. The details are given below.\\
		
		\noindent
		Let $\mc{K}^\infty$ denote the set $\R$-smooth vectors in $\mc{K}$. Then $\mc{K}^\infty$ is dense in $\mc{K}$ and $\R$-invariant. It moreover carries a Fr\'echet topology which is finer than the one inherited as a subspace of $\mc{K}$ and for which the action $\R \times \mc{K}^\infty \to \mc{K}^\infty$ is smooth \cite[Thm.\ 4.4, Lem.\ 5.2]{Neeb_diffvectors}. As $\omega : \mc{K}^\infty \times \mc{K}^\infty \to \R$ is bilinear and continuous w.r.t.\ this topology, it is smooth. Thus the generalized Heisenberg group $N := \mrm{H}(\mc{K}^\infty, \omega)$ is a Lie group. (Notice that $\restr{\omega}{\mc{K}^\infty}$ may be degenerate.) It is as a subgroup of $G$ generated by $\mc{K}^\infty$. As $\mc{K}^\infty$ is a Fr\`echet space, it is Mackey complete by \cite[Thm.\ I.4.11]{michor_convenient}, which implies using \cite[Thm.\ V.1.8]{neeb_towards_lie} that $N$ is regular. Write $\n := \Lie(N)$. By construction $\R$ acts smoothly on $N$ by $\delta_{\mc{K}}^{it}$, so that $N^\sharp := N \rtimes \R$ is a regular Lie group. Let $\n^\sharp := \n \rtimes \R \bm{d}$ denote its Lie algebra. We construct a representation of $N^\sharp$ which is smoothly KMS at $\bm{d} \in \n^\sharp$ relative to $N \triangleleft N^\sharp$. Let us recall the standard representation of $\mrm{H}(\mcH, \omega)$ on the Bosonic Fock space $\F(\mcH)$. Equip the symmetric algebra $S^\bullet(\mcH)$ with the inner product
		\begin{equation}\label{eq: ip_fock_space}
			\langle v_1\cdots v_n, w_1\cdots w_n\rangle = \sum_{\sigma \in S_n} \prod_{j=1}^n \langle v_j, w_{\sigma_j}\rangle.	
		\end{equation}
		Let $\F(\mcH)$ denote the Hilbert space completion of $S^\bullet(\mcH)$ and let $\Omega := 1 \in \mcH$ denote the vacuum vector. Then $\mcH$ contains (and is generated by) the vectors $e^v := \sum_{n=0}^\infty {1\over n!}v^n \in \mcH$ for $v \in \mcH$. There is a continuous irreducible unitary representation $W$ of $\mrm{H}(\mcH, \omega)$ on $\F(\mcH)$ satisfying $W(z,v)e^w = ze^{-{1\over 2}\norm{v}^2 - \langle v, w\rangle}e^{v + w}$ for $v,w \in \mcH$ and $z \in \T$ \cite[Sec. 9.5]{Segal_Loop_Groups}. Moreover, any unitary $u \in \U(\mcH)$ extends canonically to a unitary $\F(u) \in \U(\F(\mcH))$. We further have:
		\begin{equation}\label{eq: covariance_second_quantization}
			W(uv) = \F(u)W(v)\F(u)^{-1}, \qquad \forall u \in \U(\mcH), v \in \mcH
		\end{equation}
		In view of \eqref{eq: covariance_second_quantization}, $W$ and $\F$ together define a representation $\rho$ of the Lie group $N^\sharp$ by $\rho(n,t) := W(n)\F(\delta_{\mc{K}}^{it})$. Let $\mc{N} := W(N)^{\prime \prime}$. As $\mc{K}$ is a standard real subspace and $\mc{K}^\infty$ is dense in $\mc{K}$, it follows that $\Omega$ is cyclic and separating for $\mc{N}$ \cite[Lem.\ 6.2]{Neeb_Olafsson_mod_theory}. Let $\phi$ denote the faithful vector state on $\mc{N}$ defined by $\phi(x) = \langle \Omega, x\Omega\rangle$. Using \cite[Prop.\ 6.10]{Neeb_Olafsson_mod_theory} we have $\Delta_\phi^{it} = \F({\delta_{\mc{K}}^{it}})$ for all $t \in \R$. Consequently $\rho$ is KMS at $-\bm{d} \in \n^\sharp$ relative to $N \triangleleft N^\sharp$ (notice the minus sign in \Fref{def: mod_condition_kms}). To see it is smoothly KMS, observe that $\widehat{\phi} : N \to \C$ is smooth because it is given by
		\begin{equation}\label{eq: std_subspace_state_smooth}
			\widehat{\phi}(z,v) = \langle \Omega, W(z, v) \Omega\rangle = ze^{-{1\over 2}\norm{v}^2}.
		\end{equation}
	\end{example}

	\noindent
	The following provides an example where $\rho$ is smoothly-KMS at various $\xi_I \in \g$, relative to distinct subgroups $N_I \subseteq G$, where $I \in \mc{I}$ for some indexing set $\mc{I}$:

	\begin{example}[Bisognano-Wichmann and $\SU(1,1)$-covariant nets]\label{ex: mobius_cov_nets}~\\
		Recall that $\SU(1,1)$ acts on $S^1$. Explicitly, for $g = \begin{pmatrix}
			\alpha & \beta \\
			\overline{\beta} & \overline{\alpha}
		\end{pmatrix}\in \SU(1,1)$ with $\alpha, \beta \in \C$ satisfying $|\alpha|^2 - |\beta|^2 = 1$, define $g(z) := {\alpha z + \beta \over \overline{\beta}z + \overline{\alpha}}$ for $z \in \C$ with $|z| = 1$. With $g$ as above, define the unitary action of $\SU(1,1)$ on the complex Hilbert space $L^2(S^1; \C)$ by $(u(g)f)(z) := (\alpha - \overline{\beta}z)^{-1}f(g^{-1}(z))$ for $f \in L^2(S^1; \C)$. Let $H_+^2(S^1; \C)$ be the closed subspace of $L^2(S^1; \C)$ spanned by the non-negative Fourier modes. Let $H^2_-(S^1; \C)$ be its orthogonal complement in $L^2(S^1; \C)$. Notice that $\SU(1,1)$ leaves these subspaces invariant. Consider the complex Hilbert space $V := H_+^2(S^1; \C) \oplus \overline{H_-^2(S^1; \C)}$, where $\overline{H_-^2(S^1; \C)}$ denotes the Hilbert space complex-conjugate to $H_-^2(S^1; \C)$. Let $V_\R = L^2(S^1; \C)$ denote the real vector space underlying $V$. Define the real Fr\'echet space $V_\R^\infty := C^\infty(S^1; \C)$ and consider the symplectic vector space $(V_\R^\infty, \omega)$, where $\omega(v,w) := \mrm{Im}\langle v, w\rangle_V$ for $v,w \in V_\R^\infty$. Let $\mrm{H}(V_\R^\infty, \omega)$ be the corresponding real Heisenberg group. Consider the regular Fr\'echet-Lie group $G := \mrm{H}(V_\R^\infty, \omega)\rtimes \SU(1,1)$. Let $\bm{r} :=
		{i\over 2}
		\begin{pmatrix}
			1 & 0 \\
			0 & -1
		\end{pmatrix}
		$ and
		$\bm{d} := {1\over 2}
		\begin{pmatrix}
			0 & 1\\
			1 & 0
		\end{pmatrix}$
		denote the generators in $\su(1,1)$ of the rotation and the dilation subgroups in $\SU(1,1)$, respectively. By an interval of $S^1$, we mean a connected, open, non-empty and non-dense subset of $S^1$. Write $\mc{I}$ for the set of intervals of $S^1$, on which $\SU(1,1)$ acts naturally, and let $I_0$ denote the upper-semicircle. For $I \in \mc{I}$, define $\xi_I \in \su(1,1)$ by $\xi_I := \Ad_g(\bm{d})$, where $g \in \SU(1,1)$ is any element satisfying $g.I_0 = I$. Notice that $\xi_I$ is well-defined. Define further the closed real subspace $V_I := L^2(I; \C)$ of $V_\R$ and set $V_I^\infty := V_I\cap V_\R^\infty$. Let $N_I := \mrm{H}(V_I^\infty, \omega)\subseteq G$ be the corresponding closed subgroup of $G$. We construct a unitary representation $\rho$ of $G$ which is of p.e.\ at $\bm{r} \in \su(1,1)$ and which is KMS at $\xi_I \in \su(1,1)$ relative to $N_I$, for every $I \in \mc{I}$. The details are given below. \\
		
		\noindent
		As the $\SU(1,1)$-action $u$ on $L^2(S^1; \C)$ leaves both $H_+^2(S^1; \C)$ and $H_-^2(S^1; \C)$ invariant, we obtain a unitary representation $\widetilde{u}$ of $\SU(1,1)$ on $V = H_+^2(S^1; \C) \oplus \overline{H_-^2(S^1; \C)}$ which is by construction of p.e.\ at $\bm{r} \in \su(1,1)$. As in \Fref{ex: std_subspaces_heis_kms}, let $W$ denote the standard representation of the real Heisenberg group $\mrm{H}(V, \mrm{Im}\langle \--, \--\rangle)$ on the Fock space $\F(V)$. Letting $\SU(1,1)$ act on $\F(V)$ by second quantization, we obtain a smooth unitary representation $\rho$ of $G$ on $\F(V)$ which is of p.e.\ at $\bm{r} \in \su(1,1)$. Explicitly, $\rho$ is given by $\rho(v, g) = W(v)\mc{F}(\widetilde{u}(g))$ for $v \in \mrm{H}(V_\R^\infty, \omega)$ and $g \in \SU(1,1)$. It follows from \cite[Sec.\ II.14]{Wassermann_fusion_PE_reps} that $V_I \subseteq V$ is a standard real subspace for any interval $I \in \mc{I}$. Let $\delta_I^{it}$ denote the corresponding modular $1$-parameter group, as in \Fref{ex: std_subspaces_heis_kms}. The assignment $I \mapsto V_I$, called a \textit{net of standard subspaces}, satisfies $I_1 \subseteq I_2 \implies V_{I_1} \subseteq V_{I_2}$ (\textit{isotony}), $V_{g.I} = \widetilde{u}(g)V_I$ for $g \in \SU(1,1)$ (\textit{$\SU(1,1)$-covariance}) and $I_1 \cap I_2 = \emptyset \implies V_{I_2} \subseteq V_{I_1}^{\perp_{\omega}}$ (\textit{locality}). It moreover follows from \cite[Sec.\ II.14]{Wassermann_fusion_PE_reps} that $\delta_I^{it} = \widetilde{u}(e^{-2\pi t\xi_I})$ for all $t \in \R$ and $I \in \mc{I}$ (cf.\ \cite[Thm.\ 3.3.1]{longo_notes_I} and \cite[Thm.\ II.9]{Borchers_CPT}). Passing to the second quantization, let $\mc{N}_I := \rho(N_I)^{\prime \prime} = W(V_I^\infty)^{\prime \prime}$ denote the von Neumann algebra generated by $W(V_I^\infty)$ for $I \in \mc{I}$. By \Fref{ex: std_subspaces_heis_kms}, we obtain that
		\[\rho(e^{-2\pi t\xi_I}) = \F(\widetilde{u}(e^{-2\pi t\xi_I})) = \F(\delta_I^{it}) = \Delta_I^{it},\]
		where $\Delta_I$ denotes the modular operator on $\F(V)$ defined from $\mc{N}_I$ using the cyclic and separating vector $\Omega :=1 \in \F(V)$. Let $\phi = \langle \Omega, \bcdot\; \Omega\rangle$ be the corresponding state on $\mc{N}_I$. Then \eqref{eq: std_subspace_state_smooth} shows that $\widehat{\phi} : N_I \to \C$ is smooth. Thus $\rho$ is smoothly-KMS at $2\pi \xi_I \in \su(1,1)$ relative to $N_I\subseteq G$, for any $I \in \mc{I}$. For more details on the Bisognano-Wichmann property and nets of standard subspaces, see e.g.\ \cite{Morinelli_bw_sufficient} or \cite{Mund_BW}.
	\end{example}

	\part{Generalized Positive Energy Representations of Jet Lie Groups and Algebras}\label{part: jets}

	\noindent	
	We now depart from the general context of \Fref{part: reps}. Using the observations made in \Fref{part: reps}, we study projective unitary representations of jet Lie groups and algebras that are of generalized positive energy. Let us first fix our notation, which is kept throughout \Fref{part: jets}.
	
	\section{Notation}

	\noindent
	Let $V$ be a finite-dimensional real vector space and $K$ a $1$-connected compact simple Lie group with Lie algebra $\mfr{k}$. For any $n \in \N_{\geq 0}$, we denote by $P^n(V) \subseteq R$ the space of homogeneous polynomials on $V$ of degree $n$. Let $R := \R\llbracket V^\ast \rrbracket := \prod_{n=0}^\infty P^n(V)$ denote the ring of formal power series on $V$ with coefficients in $\R$, equipped with the product topology. Let $I = (V^\ast)$ be the maximal ideal of $R$, containing those elements with vanishing constant term. We write $\ev_0 : R \to \R \cong R/I$ for the corresponding quotient map. Let $\g$ be the $R$-module $\g := R \otimes \mfr{k}$ of formal power series on $V$ with coefficients in $\mfr{k}$. Then $\g$ is a topological Lie algebra with the Lie bracket defined by
	\[ [f \otimes X, g \otimes Y] := fg\otimes [X,Y], \qquad f,g \in R, \quad X,Y \in \mfr{k}.\]
	We also write $fX$ instead of $f\otimes X$ for $f \in R$ and $X \in \mfr{k}$. Define $R_k := R/I^{k+1}$, $I_k := I/I^{k+1}$ and $\g_k := \g/(I^{k+1}.\g)$ for $k \in \N_{\geq 0}$. Then $R = \varprojlim R_k$ and $\g = \varprojlim \g_k$ as topological vector spaces and Lie algebras, respectively. For $k \in \N_{\geq 0}$, let $G_k = J^k_0(V; K)$ be the unique $1$-connected Lie group integrating the finite-dimensional Lie algebra $\g_k$. Let $G := J^\infty_0(V; K) := \varprojlim G_k$ be the corresponding projective limit, which is a pro-Lie group with topological Lie algebra $\g = \varprojlim_k \g_k$. (See e.g.\ \cite{Hofmann_pro_lie_groups} for a detailed consideration of pro-Lie groups). Write $\X_I$ for the Lie algebra of formal vector fields on $V$ vanishing at the origin. Identify $\X_I \cong \der(I)$ using the Lie derivative $\bm{v} \mapsto \mc{L}_{\bm{v}}$. Notice further that $\der(I) \cong I \otimes V$. Define similarly $\X_{I_k} := \X_I/(I^{k+1}\X_I) \cong \der(I_k)$.\\
	
	\noindent
	Let $\p$ be a finite-dimensional Lie algebra acting on $\g$ by the homomorphism $D : \p \to \der(\g)$. Using the fact that all derivations of $\mfr{k}$ are inner, by Whitehead's first Lemma \cite[III.7.\ Lem.\ 3]{Jacobson}, it follows from \cite[Ex. 7.4]{Kac_book} that $D(p)$ splits into a horizontal and vertical part according to $D(p) = -\mc{L}_{\bm{v}(p)} + \ad_{\sigma(p)}$, where $\bm{v} : \p \to \X_I^{\op}$ is a homomorphism of Lie algebras and where $\sigma : \p \to \g$ is a linear map that necessarily satisfies the following Maurer-Cartan equation:
	\begin{equation}\label{eq: lift_satisfies_mc}
		- \mc{L}_{\bm{v}(p_1)}\sigma(p_2) + \mc{L}_{\bm{v}(p_2)}\sigma(p_1) - \sigma([p_1, p_2]) + [\sigma(p_1), \sigma(p_2)] = 0, \qquad \forall p_1, p_2 \in \p.
	\end{equation}
	
	\begin{remark}\label{rem: mc_equation}
		As we shall see in \Fref{sec: normal_form_phi} below, \Fref{eq: lift_satisfies_mc} can be written as $\delta \sigma + {1\over 2}[\sigma, \sigma] = 0$ in the differential graded Lie algebra $(\bigwedge^\bullet\p^\ast) \otimes \g$, whose differential is that of the Chevalley-Eilenberg complex, where $\g$ is considered as $\p$-module according to $p\mapsto -\mc{L}_{\bm{v}(p)}$. \\
	\end{remark}
	
	\noindent
	We will refer to $D$ as a \textit{lift} of the $\p$-action on $R$ to $\g$ and we call $\sigma$ the \textit{vertical twist} of the lift $D$.  We remark also that $D(p)$ satisfies the following Leibniz rule:
	\begin{equation}\label{eq: D_leibniz}
		D(p)(f\xi) = -\mc{L}_{\bm{v}(p)}(f)\xi + f D(p)\xi, \qquad f \in R, \xi \in \g, \qquad \forall p \in \p.	
	\end{equation}

	\noindent
	We will denote by $j^k$ various $k$-jet projections $R \to R_k$, $\g \to \g_k$ and $\X_I  \to \X_{I_k}$. It should be clear from the context which map is being used. Also, we will freely identify the quotient $\g_0 \cong \mfr{k}$ with the Lie subalgebra $\mfr{k} \subseteq \g$ of formal power series having only a non-trivial constant term. Similarly, we identify $j^1 \X_I = \X_{I_1} \cong \gl(V)$ with the subalgebra $\gl(V) \subseteq \X_I$ of linear vector fields on $V$.\\
	
	\noindent
	A first observation is the fact that $G = J^\infty_0(V; K)$ is not just a pro-Lie group, but actually a regular Lie group modeled on the Fr\'echet space $\g = J^\infty_0(V; \mfr{k})$.
	
	\begin{proposition}\label{prop: jet_lie_gp}
		Both $G$ and $G \rtimes_\alpha P$ are regular Fr\'echet-Lie groups.
	\end{proposition}
	\begin{proof}
		It is clear that the Lie algebras $\g$ and $\g \rtimes_D \p$ are Fr\'echet. For every $n \in \N$, the Lie group $G_n = J^n_0(V; K)$ is $1$-connected, because $K$ is so. Then also $G$ is $1$-connected, since $\pi_k(G) = \pi_k\big(\varprojlim_n G_n\big) = \varprojlim_n \pi_k(G_n)$ for every $k \in \N$, see e.g.\ \cite[Prop.\ 4.67]{Hatcher_alg_top}. Thus $G$ is the unique $1$-connected pro-Lie group with $\Lie(G) = \g$, which is locally contractible by \cite[Theorem 1.2]{hofmann_neeb_prolie}. Then \cite[Theorem 1.3, Prop.\ 5.7]{hofmann_neeb_prolie} entails that $G$ is a regular Lie group. As the action $\alpha : P \times G \to G$ is smooth, also $G \rtimes_\alpha P$ is a Lie group and it is regular by \cite[Thm.\ V.I.8]{neeb_towards_lie}, because both $G$ and $P$ are so.
	\end{proof}

	\noindent
	Moreover, we have the following useful fact:
	\begin{lemma}\label{lem: exp_restricts_to_diffeo}
		The exponential map $\exp_G : \g \to G$ restricts to a diffeomorphism from $I \otimes \mfr{k} = \ker\big(\ev_0 : \g \to \mfr{k}\big)$ onto $\ker\big(\ev_0 : G \to K\big)$. 
	\end{lemma}
	\begin{proof}
		For $k \in \N_{\geq 0}$, let $H_k := \ker\big(\ev_0 : G_k \to K\big)\triangleleft G_k$ be the maximal nilpotent normal subgroup of $G_k$. Then $\h_k := \Lie(H_k) = \ker\big(\ev_0 : \g_k \to \mfr{k}\big) = (I_k \otimes \mfr{k}) \triangleleft \g_k$. Write $H := \varprojlim_kH_k$ for the corresponding normal subgroup of $G$ and $\h = \varprojlim \h_k$ for its Lie algebra. Let $k \in \N$. Notice that $H_k$ is nilpotent and $1$-connected. Consequently, its exponential map is a diffeomorphism $\exp_{H_k} : \h_k \to H_k$ \cite[Thm.\ 1.2.1]{CorwinGreenleaf_book}. Write $\log_{H_k} : H_k \to \h_k$ for its inverse. If $m \geq k$, then $\exp_{H_k} \circ j^k = j^k \circ \exp_{H_m} : \h_m \to H_k$ and consequently $\log_{H_k} \circ j^k = j^k \circ \log_{H_m} : H_m \to \h_k$. Passing to the projective limit, we obtain the inverse $\log_H := \varprojlim_k \log_{H_k}$ of $\exp_H$. It is smooth because $H = \varprojlim_k H_k$ carries the projective limit topology and $\log_{H_k}$ is smooth for every $k \in \N$. Thus $\exp_H : \h \to H$ is a global diffeomorphism.
	\end{proof}

	\section{Normal Form Results}\label{sec: normal_form}

	\noindent
	By choosing suitable local coordinates, one may attempt to simplify the vector fields $\bm{v}(p)$ and the vertical twist $\sigma(p)$ of the lift $D(p) = -\mc{L}_{\bm{v}(p)} + \ad_{\sigma(p)}$ simultaneously. One might for example try to show that there are local coordinates in which the formal vector fields $\bm{v}(p)$ are linear for every $p \in \p$ simultaneously, thereby linearizing the formal $\p$-action. Similarly, one might aim to show that in suitable coordinates, $\sigma(p) \in \mfr{k} \subseteq R \otimes \mfr{k}$ is constant for all $p \in \p$, so that $\sigma$ is a Lie algebra homomorphism $\p \to \mfr{k}$. In the following, this \squotes{normal form problem} is considered. The results of \Fref{sec: restr_proj_reps} will depend on the availability of suitable normal forms, whose existence we study in the present section.\\
	
	\noindent
	In \Fref{sec: transformation_behavior}, we briefly recall the transformation behavior of $\bm{v}$ and $\sigma$ under suitable automorphisms of $\g$. We proceed in \Fref{sec: nform_vfields} to recollect some known results regarding normal forms for Lie algebras of vector fields with a common fixed point. Finally, we consider in \Fref{sec: normal_form_phi} the vertical twist $\sigma$.
	
	\subsection{Transformation Behavior}\label{sec: transformation_behavior}
	
		\begin{definition}\label{def: gauge_transf}~
		\begin{itemize}
			\item A \textit{formal diffeomorphism of $V$} is an automorphism $h$ of $R$. An automorphism of $\g$ is said to be \textit{horizontal} if it is of the form $h \otimes \id_{\mfr{k}}$ for some $h \in \Aut(R)$.  We write $h.\xi$ or $h(\xi)$ instead of $(h \otimes \id_{\mfr{k}})(\xi)$ for $\xi \in \g$.
			\item A \textit{gauge transformation} is an automorphism of $\g$ of the form $e^{\ad_\xi}$ for some $\xi \in \g$. 
		\end{itemize}	
	\end{definition}
	
	\begin{remark}\label{rem: formal_diffeo}
			Any formal diffeomorphism $h \in \Aut(R)$ preserves the maximal proper ideal $I$ and is determined by its restriction $\restr{h}{V^\ast}$, which can be regarded as an element $\widetilde{h}$ of $I \otimes V$ for which $j^1 \widetilde{h} \in V^\ast \otimes V \cong \gl(V)$ is invertible. It is then a consequence of Borel's Lemma \cite[Thm.\ 1.2.6]{Hormander_I} and the Inverse Function Theorem that for any automorphism $h$ of $R$, there exist $0$-neighborhoods $U, U^\prime \subseteq V$ and a diffeomorphism $h_0 : U \to U^\prime$ satisfying $h_0(0) = 0$ such that $h(j^\infty_0(f)) = j^\infty_0(f \circ h_0^{-1})$. Similarly, for $\xi \in \g$ there exists $\eta \in C^\infty_{\mrm{c}}(V; \mfr{k})$ s.t.\ $j_0^\infty(\eta) = \xi$, where we have identified $\g \cong J^\infty_0(V; \mfr{k})$. We then have $e^{\ad_\xi} \circ j^\infty_0 = j^\infty_0 \circ e^{\ad_\eta}$.\\
	\end{remark}

	\noindent
	To determine the transformation behavior of $D : \p \to \der(\g)$, we have to consider the adjoint action of $\Aut(\g)$ on $\der(\g)$. Instead of considering arbitrary automorphisms of $\g$, we will specialize to horizontal ones and to gauge transformations. For $h \in \Aut(R)$ and $v \in \X_I^{\op}$, we write $h.v$ for the action of $\Aut(R)$ on $\X_I^{\op}$ obtained from the adjoint action of $\Aut(R)$ on $\der(R) \cong \X_I \cong \X_I^{\op}$. The following two proofs are due to K.-H.\ Neeb and B.\ Janssens. They appear in the presently unpublished article \cite{BasNeeb_PE_reps_II}.
	
	\begin{lemma}[\cite{BasNeeb_PE_reps_II}]\label{lem: transf_behavior_of_derivations}
		Let $D \in \der(\g)$ and $\xi \in \g$. Then 
		\[ e^{\ad_\xi} \circ D \circ e^{-\ad_\xi} = D + \ad\big(F(\ad_\xi)D\xi\big), \]
		where $F(w) = -\int_{0}^{1}e^{tw}dt = - \sum_{n=0}^\infty{1\over (n+1)!} w^n$.
	\end{lemma}
	\begin{proof}
		Let $k \in \N$ be arbitrary. Consider the continuous path $\gamma : I \to \der(\g)$ defined by $\gamma(t) = e^{t\ad_{\xi}} D e^{-t\ad_{\xi}}$. Notice that $j^k \circ \gamma : I \to \der(\g_k)$ is smooth for all $k$ and consequently so is $\gamma$. Moreover
		\[ \gamma^\prime(t) = e^{t\ad_{\xi}} [\ad_{\xi}, D] e^{-t\ad_{\xi}} = -e^{t\ad_{\xi}} \ad_{D\xi} e^{-t\ad_{\xi}} = -\ad\big( e^{t \ad_{\xi}} D\xi\big),\]
		where the last step uses that $\alpha \circ \ad_{\eta} = \ad_{\alpha(\eta)} \circ \alpha$ for any $\alpha \in \Aut(\g)$. Thus 
		\[ e^{\ad_{\xi}} \circ D \circ e^{-\ad_{\xi}} - D = \int_0^1 \gamma^\prime(t) dt = -\int_0^1 \ad\big( e^{t \ad_{\xi}}D\xi\big) dt = -\ad\bigg( \int_0^1  e^{t \ad_{\xi}}dt\bigg)(D\xi) = \ad\big(F(\ad_{\xi})D\xi\big). \qedhere\]
	\end{proof}

	\begin{proposition}[\cite{BasNeeb_PE_reps_II}]\label{prop: transf_beh}
		Let $h \in \Aut(R)\subseteq \Aut(\g)$, $\sigma,\xi \in \g$ and $v \in \X_I$. \\
		Consider the derivation $D := -\mc{L}_v + \ad_\sigma \in \der(\g)$. Then
		\begin{equation}\label{eq: transf_beh}
			\begin{aligned}
				h \circ D \circ h^{-1} &= -\mc{L}_{h.v} + \ad(h.\sigma),\\
				e^{\ad_\xi} \circ D \circ e^{-\ad_\xi} &= -\mc{L}_v + \ad\bigg(e^{\ad_\xi}\sigma + F(\ad_\xi)(-\mc{L}_v\xi)\bigg).
			\end{aligned}
		\end{equation}
	\end{proposition}
	\begin{proof}
		It is trivial that $h \circ\mc{L}_v \circ h^{-1} = \mc{L}_{h.v}$. Moreover, $h \circ \ad_\sigma \circ h^{-1} = \ad_{h.\sigma}$ is valid because $\alpha \circ \ad_\sigma = \ad_{\alpha(\sigma)} \circ \alpha$ for any $\alpha \in \Aut(\g)$. Notice next that $F(\ad_\xi)([\sigma, \xi]) = \sum_{n=1}^\infty{1\over n!}{\ad_\xi}^n\sigma = e^{\ad_\xi}\sigma - \sigma$. It follows from \Fref{lem: transf_behavior_of_derivations}:
		\begin{align*}
			e^{\ad_\xi} \circ D \circ e^{-\ad_\xi} 
			& -\mc{L}_v + \ad_{\sigma} + \ad\big(F(\ad_\xi)(-\mc{L}_v\xi)\big) + \ad\big(F(\ad_\xi)[\sigma, \xi]\big) \\
			&= -\mc{L}_v + \ad\bigg(e^{\ad_\xi}\sigma + F(\ad_\xi)(-\mc{L}_v\xi)\bigg).\qedhere
		\end{align*}
	\end{proof}
	
	\begin{definition}\label{def: gauge_equivalence}~
		\begin{itemize}
			\item Two homomorphisms $\bm{v}, \bm{w} : \p \to \X_I^{\op}$ are said to be \textit{formally-equivalent} if there is a formal diffeomorphism $h\in \Aut(R)$ such that $h.\bm{v}(p) = \bm{w}(p)$ for all $p \in \p$. 
			\item Two linear maps $\sigma, \eta : \p \to R \otimes \mfr{k}$ satisfying the Maurer-Cartan equation \eqref{eq: lift_satisfies_mc} are called \textit{gauge-equivalent} if there is some $\xi \in \g$ such that
			\begin{equation}\label{eq: gauge_action}
				\eta(p) = e^{\ad_\xi}\sigma(p) + F(\ad_\xi)(-\mc{L}_{\bm{v}(p)}\xi), \qquad \forall p \in \p.	
			\end{equation}
			In this case, we write $\sigma \sim \eta$ and say that $\sigma$ and $\eta$ are related by the gauge transformation $e^{\ad_\xi}$. 
		\end{itemize}		
	\end{definition}
	
	\subsection{Lie Algebras of Formal Vector Fields with a Common Fixed Point}\label{sec: nform_vfields}
	
	\noindent
	The normal form problem for vector fields near a fixed point has been subject to extensive study. Let us first gather some relevant known results.

	\subsubsection*{The case of a single vector field}
	\noindent
	Naturally, the special case which has been considered most is the case where $\p$ is simply $\R$, in which case one is looking for normal forms of dynamical systems near a fixed point, in the formal context. This case is already quite interesting. Let us recollect some relevant results. For more information, we refer to \cite{arnold_normal_forms}. \\

	\noindent
	Let $\bm{v}$ be a vector field on $V$. Write $\bm{v} = \bm{v}_{\mrm{l}} + \bm{v}_{\mrm{ho}}$, where $\bm{v}_{\mrm{l}} = j_0^1(\bm{v}) \in \gl(V)\subseteq \X_I$ is the linearization of $\bm{v}$ at $0 \in V$ and $\bm{v}_{\mrm{ho}} \in \X_{I^2}$ is a vector field vanishing up to first order at $0 \in V$. Let $\bm{v}_{\mrm{l}}  = \bm{v}_{\mrm{l,s}} + \bm{v}_{\mrm{l,n}}$ be the Jordan decomposition of $\bm{v}_{\mrm{l}}$ over $\C$, where $\bm{v}_{\mrm{l,s}}$ is semisimple and $\bm{v}_{\mrm{l,n}}$ is nilpotent. Write $V_\C := V \otimes_\R\C$. Let $(e_j)_{j=1}^d$ be a basis of eigenvectors of $\bm{v}_{\mrm{l,s}}$ in $V_\C$ with dual basis $(x_j)_{j=1}^d$ of $V^\ast$. Let $(\mu_j)_{j=1}^d$ denote the corresponding eigenvalues.
	
	\begin{definition}\label{def: resonances}
		 Let $\bm{n} \in \N_{\geq 0}^d$ be a multi-index. A monomial vector field $x^{\bm{n}} \partial_{x_j}$ with $|\bm{n}| \geq 2$ is called \textit{resonant} if $\langle \bm{n}, \bm{\mu} \rangle = \mu_j$, where $\langle \bm{n}, \bm{\mu} \rangle := \sum_{i=1}^d n_i \mu_i$. Identifying $\bm{v}_{\mrm{l,s}}$ with the linear vector field $\sum_{j=1}^d \mu_j x_j \partial_{x_j}$ on $\C^d$, this is equivalent to $[\bm{v}_{\mrm{l,s}},\, x^{\bm{n}} \partial_{x_j}] = 0$.
	\end{definition}
	
	\begin{theorem}[Poincar{\'e}-Dulac Theorem \cite{dulac_thm} {\cite[ch.5]{arnold_normal_forms}}] \label{thm: poincare_dulac}~\\
		There exists $\bm{w}_{\mrm{ho}} \in \X_{I^2}$ which is a $\C$-linear combination of resonant monomials s.t.\ $\bm{v}$ is formally equivalent to $\bm{w} = \bm{v}_{\mrm{l}} + \bm{w}_{\mrm{ho}} \in \X_I$. In particular $[\bm{v}_{\mrm{l,s}},\, \bm{w}_{\mrm{ho}}] = 0$ in $\X_I$.
	\end{theorem}
	
	\begin{corollary}[Poincar{\'e} \cite{poincare1879proprietes}]\label{cor: no_resonances_then_linearizable}
		If there are no resonances, that is to say, if $\langle \bm{n},  \bm{\mu}\rangle \neq \mu_j$ for all $\bm{n} \in \N_{\geq 0}^d$ with $|\bm{n}|\geq 2$ and $j\in \{1,\cdots, d\}$, then the vector field $\bm{v}$ can be formally linearized, so that $\bm{v}$ is formally equivalent to the linear vector field $\bm{v}_{\mrm{l}}$.
	\end{corollary}

	\subsubsection*{The case of actions by a compact Lie group}
	For actions of compact Lie groups there is the following well-known result, see also \cite[Ch. 2.2]{Duistermaat_book}.

	\begin{theorem}[Bochner's Linearization Theorem \cite{Bochner_linearization}]\label{thm: bochner_linearization}~\\
		Let $G \times M \to M$ be a smooth action of compact Lie group on a smooth manifold which has a fixed point $a \in M$. Then, in suitably chosen smooth local coordinates around the fixed point, the action is linear.
	\end{theorem}

	\subsubsection*{The case of actions by semisimple Lie algebras}
	
	\noindent
	Next, we move to Lie algebra representations by formal vector fields of semisimple Lie algebras. As nicely explained in \cite{Fernandez_linearization_poiss} and first observed by Hermann in \cite{Hermann_linearization}, in the formal setting the obstructions to being able to linearize a Lie algebra of vector fields simultaneously lie in various first Lie algebra cohomology groups $H^1(\p, W)$ for suitable finite-dimensional $\p$-modules $W$. In view of Whitehead's First Lemma \cite[III.7.\ Lem.\ 3]{Jacobson}, this results in:
	
	\begin{theorem}[\cite{Hermann_linearization}]\label{thm: herman_linearization}
		Let $\p$ be a semisimple Lie algebra and $\bm{v} : \p \to \X_I^{\op}$ be a Lie algebra homomorphism. Then $\bm{v}$ is formally equivalent to its linearization $j^1\bm{v} : \p \to \gl(V)\subseteq \X_I^{\op}$ around the origin.
	\end{theorem}

	\begin{remark}
		Corresponding statements of \Fref{thm: herman_linearization} in the setting of germs of smooth/analytic vector fields and diffeomorphisms have been proven in \cite{Sternberg_linearization} and \cite{Fernandez_linearization_poiss} under additional assumptions. They are false in general without suitable extra conditions, as was shown in \cite{Sternberg_linearization}.
	\end{remark}
	
	\subsection{Normal Form Results for the Vertical Twist}\label{sec: normal_form_phi}

	\noindent
	Let us next consider the vertical twist $\sigma : \p \to \g$ of the lift $D(p) = -\mc{L}_{\bm{v}(p)} + \ad_{\sigma(p)}$ to $\g$ of the $\p$-action $-\mc{L}_{\bm{v}(p)}$ on $R$, which has to satisfy the Maurer Cartan equation \eqref{eq: lift_satisfies_mc}. We fix the horizontal part $\bm{v} : \p \to \X_I^{\op}$ and act by gauge transformations. The main results of this section are the following two theorems, whose proof comprises the remainder of the section. The reader who is eager to consider the projective unitary g.p.e.\ representations of $\g$ can proceed to \Fref{sec: restr_proj_reps} after reading \Fref{thm: normalformverticaltwistss} and \Fref{thm: normalformverticaltwistoned} below.\\
	
	\noindent
	Let us also remark that the methods used in this section to prove \Fref{thm: normalformverticaltwistss} and \Fref{thm: normalformverticaltwistoned} were communicated to the author by B.\ Janssens and K.H.\ Neeb and appear in similar form in their presently unpublished work \cite{BasNeeb_PE_reps_II}, albeit in a more specific context. The author has placed their approach in a more general context and extracted the two theorems below.

	\begin{restatable}{theorem}{normalformverticaltwistss}\label{thm: normalformverticaltwistss}
		Assume that $\p$ is semisimple. Let the linear map $\sigma : \p \to \g$ satisfy the Maurer-Cartan equation \eqref{eq: lift_satisfies_mc}. Then $\sigma$ is gauge-equivalent to $\sigma_0 := \ev_0 \circ \sigma : \p \to \mfr{k}$. If $\p$ has no non-trivial compact ideals, then $\sigma$ is gauge-equivalent to $0$.
	\end{restatable}

	\noindent
	The next result concerns the case $\p = \R$, in which case we identify $\bm{v}$ with $\bm{v}(1) \in \X_I$ and $\sigma$ with $\sigma(1) \in \g$. In this case, the Maurer-Cartan equation \eqref{eq: lift_satisfies_mc} is trivially satisfied for any $\sigma \in \g$ and $\bm{v} \in \X_I$. 

	\begin{restatable}{theorem}{normalformverticaltwistoned}\label{thm: normalformverticaltwistoned}
		Assume that $\p = \R$. Let $\sigma \in \g$ and $\bm{v} \in \X_I$. Let $\bm{v}_{\mrm{l}} := j^1\bm{v} \in \gl(V)$ be the linearization of $\bm{v}$ at $0\in V$. Assume w.l.o.g. that $\sigma_0 := \ev_0(\sigma) \in \mfr{t}$ for some maximal torus $\mfr{t} \subseteq \mfr{k}$. The following assertions hold:
		\begin{enumerate}
			\item Assume that $\langle \bm{n}, \bm{\mu}\rangle \neq \alpha(\sigma_0)$ for any root $\alpha \in i\mfr{t}^\ast$ of $\mfr{k}$ and $\bm{n} \in \N_{\geq 0}^{d}$ with $|\bm{n}| \geq 1$. \\
			Then $\sigma$ is gauge-equivalent to some $\sigma^\prime \in R \otimes \mfr{t}$.
			\item If $\bm{v}_{\mrm{l}}$ is semisimple, then $\sigma$ is gauge-equivalent to some $\nu \in R \otimes \mfr{k}$ satisfying $-\mc{L}_{\bm{v}_{\mrm{l}}}\nu + [\sigma_0, \nu] = 0$.
			\item Suppose that $\bm{v} = \bm{v}_{\mrm{l}}$ is linear. Assume that $D = -\mc{L}_{\bm{v}} + \ad_\sigma$ integrates to a continuous $\T = \R/2\pi\Z$-action on $\g$. Then $\sigma$ is gauge-equivalent to $\sigma_0 \in \mfr{t}$. Moreover $\Spec(\bm{v}_{\mrm{l}})\cup \Spec(\ad_\sigma) \subseteq 2\pi i\Z$.
		\end{enumerate} 
	\end{restatable}

	\begin{remark}\label{rem: compact_action}
		Suppose that $\mc{K} \to M$ is a principal fiber bundle with compact simple structure group $K$. Let $\alpha : \T \to \Aut(\mc{K})$ be a smooth action on $\mc{K}$ by bundle automorphisms. Suppose that $a \in M$ is a fixed point of the induced $\T$-action on $M$ and set $V := T_a(M)$. By \Fref{thm: bochner_linearization}, the $\T$-action on $M$ is linear in suitable local coordinates around $a \in M$. Passing to $J^\infty_a(M) \cong R$ and $J^\infty_a(\Ad(\mfr{K})) \cong \g$, one obtains a $\T$-action on both $R$ and $\g$. The corresponding derivations are given by $-\mc{L}_{\bm{v}}$ and $D = -\mc{L}_{\bm{v}} + \ad_\sigma$ respectively, for some linear semisimple vector field $\bm{v}$ on $V$ and some $\sigma \in \g$. This is the setting of the third item in \Fref{thm: normalformverticaltwistoned}, according to which we may further assume that $\sigma \in \mfr{t}$, where $\mfr{t} \subseteq \mfr{k}$ is a maximal torus, by acting with gauge transformations.
	\end{remark}

	\noindent
	The remainder of this section is devoted to the proof of \Fref{thm: normalformverticaltwistss} and \Fref{thm: normalformverticaltwistoned}.
	
	\subsubsection*{Reformulation using differential graded Lie algebras}
	
	\noindent
	In order to classify the equivalence classes of vertical twists $\sigma : \p \to \g$, we interpret \fref{eq: lift_satisfies_mc} as the Maurer Cartan equation in the differential graded Lie algebra (DGLA) $L := L_R := (\bigwedge^\bullet \p^\ast) \otimes \g$. As a cochain complex, $L$ is the Chevalley-Eilenberg complex associated to the $\p$-module $\g$, where $\p$ acts on $\g$ by $p.\psi = -\mc{L}_{\bm{v}(p)}\psi$. Explicitly, the differential $\delta$ is given by
	\[ \delta(\alpha)(p_1, \cdots, p_{k+1}) = \sum_{i}(-1)^{i+1}p_i.\alpha(p_1, \cdots, \widehat{p_i}, \cdots, p_{k+1}) + \sum_{i < j} (-1)^{i+j} \alpha([p_i, p_j], p_1, \cdots, \widehat{p_i}, \cdots, \widehat{p_j}, \cdots, p_{k+1}),\]
	where as usual, the arguments with a caret are to be omitted. The graded Lie bracket on $L$ is the unique bilinear map $[\--, \--] : L \times L \to L$ satisfying  $[\alpha \otimes \sigma, \beta \otimes \psi] := (\alpha\wedge\beta) \otimes [\sigma, \psi]$ for $\alpha, \beta \in \bigwedge^\bullet \p^\ast$ and $\sigma, \psi \in \g$. Write $L^k := \big(\bigwedge^k\p^\ast\big) \otimes \g$ for the degree $k$-elements in $L$. Interpreting $\sigma$ as a degree-1 element of $L$, \fref{eq: lift_satisfies_mc} can now equivalently be written as the usual MC-equation $\delta \sigma + {1\over 2}[\sigma, \sigma] = 0$ in $L$.\\
	
	\noindent
	Let us next reformulate the gauge-action \eqref{eq: gauge_action} of $\g$ on the set of vertical twists, using the DGLA $L$. Consider the extended DGLA $L \rtimes \R D$, where $D$ is a degree-1 element satisfying $[D, \sigma] = \delta(\sigma)$ for any $\sigma \in L$. Notice for $\xi \in \g = L^0$ that $\delta(\xi)(p) = -\mc{L}_{\bm{v}(p)}\xi$. We define the gauge-action of $L^0 = \g$ on $L$ by
	\begin{equation}\label{eq: gauge_action_2}
		\xi.\sigma = e^{\ad(\xi)}(D + \sigma) - D = e^{\ad_\xi}(\sigma) + F(\ad_\xi)(\delta(\xi)),	\qquad \xi \in \g,
	\end{equation}
	considered as an expression in $L \rtimes \R D$, where $F(w) = - \sum_{n=0}^\infty {1\over (n+1)!}w^n = - \int_0^1 e^{tw}dt$. Let us check that the above is indeed well-defined, even though $L$ is not a nilpotent DGLA. Since $G = \varprojlim_k G_k$ is a regular Lie group, it has an exponential map and so the automorphism $e^{\ad_\xi} := \Ad(e^\xi)$ on $\g$ is defined. Consequently, so is
	\[ F(\ad_\xi)(-\mc{L}_{\bm{v}(p)}\xi) = \int_{0}^1 e^{t \ad_\xi}(\mc{L}_{\bm{v}(p)}\xi)dt\]
	for any $p \in \p$. Thus the expression in \fref{eq: gauge_action_2} makes sense. Notice further that for $\sigma \in L^1$, the above reduces precisely to the transformation behavior \eqref{eq: gauge_action} of the vertical twist. In accordance with \Fref{def: gauge_equivalence}, we say that the MC-elements $\sigma, \sigma^\prime \in L^1$ are gauge-equivalent if they satisfy $\sigma^\prime = \xi.\sigma$ for some $\xi \in L^0$, in which case we write $\sigma \sim \sigma^\prime$. Our goal is to study the MC-elements in $L^1$ up to gauge-equivalence. \\
	
	\noindent
	Let $n \in \N_{\geq 0}$. Define analogously the following DGLAs, where we consider $P^n(V)$ as $\p$-module by identifying $P^{n}(V)$ with $I^n/I^{n+1}$ for $n \in \N_{\geq 0}$, so that $p.f = -\mc{L}_{\vl{p}}f$ for $p \in \p$ and $f \in P^n(V)$:
	\begin{alignat*}{2}
		L_I &:= {\textstyle\bigwedge^\bullet}\, \p^\ast \otimes (I \otimes \mfr{k}), \qquad \qquad &
		L_{R_n} &:= {\textstyle\bigwedge^\bullet}\, \p^\ast \otimes (R_n \otimes \mfr{k}),\\
		L_{I_n} &:= {\textstyle\bigwedge^\bullet}\, \p^\ast \otimes (I_n \otimes \mfr{k}),&
		L_{P^n} &:= {\textstyle\bigwedge^\bullet}\, \p^\ast \otimes (P^n(V) \otimes \mfr{k}).
	\end{alignat*}

	\subsubsection*{Shifted DGLAs}
	\noindent
	It will be beneficial to split off the constants terms of the $\mfr{k}$-valued formal power series, because contrary to $L_R$, $L_I$ is a projective limit of \textit{nilpotent} DGLAs. We discuss next how this can be done.\\
	
	\noindent
	For any MC-element $\chi \in L_{R_0}^1 = \p^\ast\otimes \mfr{k} \subseteq L_R$ of degree $1$, define the "shifted" DGLA $L^\chi_R$, which agrees with $L_R$ as a graded Lie algebra but has a shifted differential given by $\delta_\chi(\sigma) := \delta(\sigma) + [\chi, \sigma]$. The differential $\delta_\chi$ agrees with the Chevalley-Eilenberg differential of $\big(\bigwedge^\bullet \p^\ast\big) \otimes \g$ if $\g$ is considered as $\p$-module with the twisted action $p.\sigma := -\mc{L}_{\bm{v}(p)} + [\chi, \--]$. In particular, $\delta_\chi^2 = 0$. Let us write $R\otimes_\chi \mfr{k}$ for this module structure to distinguish it form the usual one on $\g = R\otimes \mfr{k}$, which was given by $p.\xi = -\mc{L}_{\bm{v}(p)}\xi$. Define also the extended DGLA $L_R^\chi \rtimes \R D_\chi$, where $[D_\chi, \sigma] = \delta_\chi(\sigma)$. Define in analogous fashion $L_{I_n}^\chi$, $L_I^{\chi}$ and $L_{P^n}^\chi$, where we have used that the $\p$-action on $R\otimes_\chi \mfr{k}$ leaves $I_n\otimes \mfr{k}$ invariant for every $n$, so that $P^n(V) \otimes \mfr{k} \cong (I_n \otimes \mfr{k}) / (I_{n+1}\otimes \mfr{k}$) is naturally a $\p$-module. The following is a standard result: 
	
	\begin{lemma}\label{lem: splitting_off_constants}~
		\begin{enumerate}
			\item Let $\sigma \in L_R^1$. Then $\chi + \sigma$ is a MC-element in $L_R$ if and only if $\sigma$ is a MC-element in $L_R^\chi$.
			\item Let $\sigma, \sigma^\prime \in L_{R}^\chi$ be degree-1 MC-elements. Then $\chi + \sigma \sim \chi + \sigma^\prime$ in $L_R$ if and only if $\sigma \sim \sigma^\prime$ in $L_R^\chi$.
			\item Let $\psi \in L_R^1$ be a MC-element. Then $\psi = \chi + \sigma$ for some degree-1 MC-elements $\sigma \in L_I^\chi$ and $\chi \in L_{R_0}$.
		\end{enumerate}
	\end{lemma}
	\begin{proof}~
		\begin{enumerate}
			\item As $\chi$ is a MC-element and $[\sigma, \chi]  = [\chi, \sigma]$ we have
			\begin{align*}
				\delta(\chi + \sigma) + {1\over 2}[\chi + \sigma, \chi+ \sigma] = \delta(\sigma) + {1\over 2}[\sigma, \sigma] + [\chi, \sigma] = \delta_\chi(\sigma) + {1\over 2}[\sigma, \sigma].
			\end{align*}
			\item Observe that $-F(\ad_\xi)([\xi, \chi]) = e^{\ad_\xi}(\chi) - \chi$. Consequently 
			\begin{equation}\label{eq: different_gauge_actions}
				\begin{aligned}
					F(\ad_\xi)(\delta_\chi(\xi)) = F(\ad_\xi)(\delta(\xi)) + F(\ad_\xi)([\chi, \xi]) = F(\ad_\xi)(\delta(\xi)) + e^{\ad_\xi}(\chi) - \chi.	
				\end{aligned}
			\end{equation}
			Thus, for any $\xi \in \g$ we have
			\[ e^{\ad_\xi}(\chi + \sigma) + F(\ad_\xi)(\delta(\xi)) =  \chi + \bigg(e^{\ad_\xi}(\sigma) + F(\ad_\xi)(\delta_\chi(\xi)) \bigg).\]
			\item Since $R = R_0 \oplus I$ as a vector space, we can write $\psi = \chi + \sigma$, where $\chi = j^0(\psi) \in L_{R_0}$ and $\sigma \in L^\chi_{I}$. As $j^0$ is a morphism of DGLAs, it is clear that $\chi = j^0(\psi)$ is a MC-element in $L_{R_0} \subseteq L_R$. By the first point it follows that $\sigma$ is a MC-element in $L_I^\chi \subseteq L_R^\chi$.\qedhere
		\end{enumerate}
	\end{proof}

	\subsubsection*{Study of MC-elements}

	\noindent
	In view of \Fref{lem: splitting_off_constants}, let us first study the classification problem of gauge-orbits of MC-elements in $L_{R_0}^1$ and then, for each MC-element $\chi \in L_{R_0}^1$ consider the orbits in $L_{I}^\chi$ under the gauge-action.
	
	\begin{lemma}\label{lem: mc_element_constant_term}~
		\begin{enumerate}
			\item Let $\chi : \p \to \mfr{k}$ be linear. Then $\chi$ is a MC-element in $L_{R_0}^1$ if and only if it is a Lie algebra homomorphism. Thus if there are no homomorphisms $\p \to \mfr{k}$, then any MC-element $\chi \in L_{R_0}^1 = \p^\ast \otimes \mfr{k}$ is trivial.
			\item The gauge-action of $X \in \mfr{k} = L_{R_0}^0$ on $L_{R_0}$ is given by $X.\chi = e^{X}\chi$.
		\end{enumerate}
	\end{lemma}
	\begin{proof}
		Notice that $\delta(X) = 0$ for any $X \in \mfr{k} \subseteq \g$, because $-\mc{L}_{\bm{v}(p)}X = 0$ for any $p \in \p$. So $\p$ acts trivially on $\mfr{k} = \g_0 = j^0 \g$. Thus the Maurer-Cartan condition reads simply $\chi([p_1, p_2]) - [\chi(p_1), \chi(p_2)] = 0$ for all $p_1, p_2 \in \p$, proving the first statement. The second statement follows at once from the definition \eqref{eq: gauge_action_2}, using once more that the $\p$-action on $\mfr{k}$ is trivial.
	\end{proof}
	
	\noindent
	Next, we fix a homomorphism $\chi : \p \to \mfr{k}$ and turn to the MC-elements of the twisted DGLAs $L_I^\chi$. Consider the following diagram of DGLAs:
	\[
	\begin{tikzcd}
			{}		& L_I^\chi  \arrow{d} \arrow{rd} \arrow{rrrd} 	& {}             	&  {}            	& {}   \\
		\cdots \arrow{r} & L_{I_{k+1}}^\chi \arrow{r} & L_{I_{k}}^\chi \arrow{r} & \cdots \arrow{r} & L_{I_{0}}^\chi = \{0\}
	\end{tikzcd}
	\]
	\noindent
	Any MC-element in $L_I^\chi$ projects to one in $L_{I_k}^\chi$ for any $k \in \N$, and all maps in the above diagram are equivariant w.r.t. the gauge-actions. Notice further that each $L_{I_k}^\chi$ is nilpotent. To study the MC-elements in $L_I^\chi$, we consider lifts of MC-elements from $L_{I_{k}}^\chi$ to $L_{I_{k+1}}^\chi$, so as to solve the problem step-by-step. This can be done using the following central extension of nilpotent DGLAs, where $L^\chi_{I_{0}} = \{0\}$ is trivial:
	\begin{equation}\label{eq: exact_seq_dglas}
		0 \to L^\chi_{P^{k}} \to L^\chi_{I_{k}} \to L^\chi_{I_{k-1}} \to 0, \qquad k \in \N
	\end{equation}
	in combination with the following known result from deformation theory (cf.\ \cite[Sec.\ V.6]{Manetti_lectures}):
	
	\newpage
	\begin{lemma}\label{lem: extension_of_mc_elements}
		Let $0 \to K \to L \to M \to 0$ be a central extension of nilpotent DGLAs. Let $\sigma_M \in M^1$ be a MC-element.
		\begin{enumerate}
			\item Suppose that $\sigma_L \in L$ projects to $\sigma_M$. Then $h:= \delta \sigma_L + {1\over 2}[\sigma_L, \sigma_L] \in K^2$	is closed and $[h] \in H^2(K)$ is independent of the lift $\sigma_L$ of $\sigma_M$. Moreover, there is some $\eta \in K^1$ such that $\sigma_L + \eta$ is a MC-element in $L^1$ if and only if $[h] = 0$ in $H^2(K)$. 
			\item If $\sigma_L$ and $\sigma_L^\prime$ are two lifts of $\sigma_M$ that are both MC-elements in $L^1$, then $\Delta := \sigma_L^\prime - \sigma_L \in K^1$ is closed. Conversely, if $\Delta \in K^1$ is closed and $\sigma_L$ is a lift of $\sigma_M$ which is a MC-element, then $\sigma_L^\prime := \sigma_L + \Delta$ is also a lift of $\sigma_M$ which is a MC-element. Moreover, the class $[\Delta] \in H^1(K)$ vanishes if and only if $\sigma_{L}$ and $\sigma_{L^\prime}$ are related by a gauge transformation of some element $\xi \in K^0$.
			\item If $\sigma_L$ is any lift of $\sigma_M$ which is a MC-element, then the map $\Delta \mapsto \sigma_L + \Delta$ induces a bijection between $H^1(K)$ and $K^0$-orbits of MC-elements in $L^1$ lifting $\sigma_M$.
		\end{enumerate}
	\end{lemma}
	\begin{proof}~
		\begin{enumerate}
			\item It is clear that $h \in K^2$ as it projects to zero in $M^2$. Since $\delta h = [\delta \sigma_L, \sigma_L]$ (by the graded Leibniz rule), we find using $\delta \sigma_L = h - {1\over 2}[\sigma_L, \sigma_L] $ that 
			\[ \delta h = [h, \sigma_L] - {1\over 2}[[\sigma_L, \sigma_L], \sigma_L] = 0, \]
			where the second term vanishes by the graded Jacobi identity and the first term vanishes because $h \in K$ is central. Thus $h$ is closed. Suppose that $\sigma_L^\prime$ is some other lift of $\sigma_M$ and define $h^\prime := \delta \sigma_L^\prime + {1\over 2}[\sigma_L^\prime, \sigma_L^\prime] \in K^2$. Then $\Delta := \sigma_L^\prime - \sigma_L \in K$ lies in the center, so that $h^\prime = h + \delta \Delta$. It follows that $[h] \in H^2(K)$ does not depend on the lift. If there is some $\eta \in K^1$ such that $\sigma_L + \eta$ is a MC-element in $L^1$, then 
			\[0 = \delta(\sigma_L + \eta) + {1\over 2}[\sigma_L + \eta, \sigma_L + \eta] = \delta\eta + h.\]
			Hence $[h] = 0$. Conversely, if $[h] = 0 \in H^2(K)$, then there exists $\eta \in K^1$ such that $h + \delta\eta = 0$. Then $\sigma_L + \eta$ is a MC-element, by the same computation.
			\item Let $\sigma_L^\prime$ and $\sigma_L$ be MC-elements in $L^1$ lifting $\sigma_M$. We have already noticed that $h^\prime = h + \delta\Delta$, where $\Delta := \sigma_L^\prime - \sigma_L \in K^1$. Since $h = h^\prime = 0$ by assumption, it follows that $\delta \Delta= 0$. Conversely, suppose $\Delta \in K^1$ is closed and that $\sigma_L$ is a MC-element projecting to $\sigma_M$. Then $\sigma_L^\prime := \sigma_L+ \Delta$ projects to $\sigma_M$ as well. Also, $\sigma_L^\prime$ is a MC-element, because $\delta \sigma_L^\prime + {1\over 2}[\sigma_L^\prime, \sigma_L^\prime] = \delta \sigma_L + {1\over 2}[\sigma_L, \sigma_L] + \delta \Delta = 0$. To see that $[\Delta] = 0$ in $H^1(K)$ if and only if $\sigma_L$ and $\sigma_L^\prime =\sigma_L + \Delta$ are related by a gauge transformation by some element $\xi \in K^0$, observe that if $\xi \in K^0$, then as $\xi$ is central we have
			\[ \xi.\sigma_L = e^{\ad_\xi}(\sigma_L + D) - D = \sigma_L - [D, \xi] = \sigma_L - \delta \xi. \]
			\item This is immediate from the previous point.
		\end{enumerate}
	\end{proof}

	\noindent
	Next, we apply \Fref{lem: extension_of_mc_elements} to the exact sequences \eqref{eq: exact_seq_dglas}.
	
	\begin{lemma}\label{lem: limit_gauge_transf}
		For every sequence $(\xi_k)_{k \in \N}$ of degree-$0$ elements in $L_I^\chi$ with $\xi_k \in P^k(V) \otimes_\chi \mfr{k}$ for every $k \in \N$, there exists $\eta \in I \otimes_\chi \mfr{k}$ such that $j^n(\eta.\sigma) = \xi_n.\xi_{n-1}.\cdots.\xi_1.\sigma$ for every $\sigma \in L_I^\chi$ and $n \in \N$.
	\end{lemma}
	\begin{proof}
		Consider the Lie subgroup $H := \ker\big(\ev_0 : G \to K\big)\triangleleft G$ with Lie algebra $\h := \ker\big(\ev_0 : \g \to \mfr{k}\big) = I \otimes \mfr{k}$. Similarly, for $n \in \N$ let $H_n := \ker\big(\ev_0 : G_n \to K\big)$ and $\h_n := \Lie(H_n)$. Recall that the exponential map $\exp : \h \to H$ is a global diffeomorphism, by \Fref{lem: exp_restricts_to_diffeo}. Write $\log : H \to \h$ for its inverse. From $j^n \circ \exp = \exp \circ j^n : \h \to H_n$ we obtain that $\log \circ j^n = j^n \circ \log : H \to \h_n$ for any $n \in \N$. As $[\xi_k, I \otimes \mfr{k}] \subseteq I^{k+1}\otimes \mfr{k}$ and the $\xi_k$ are of increasing order, we claim that the limit $ \eta := \lim_{N \to \infty} \log\big(\prod_{k=1}^N e^{\xi_k}\big)$ exists in $I \otimes \mfr{k}$ w.r.t. the projective limit-topology, where $k$ increases from \textit{right to left} in the expression. Indeed, to see this it suffices to show that for each $n \in \N$ the sequence $(j^n \eta_N)_{N=1}^\infty$ stabilizes for large enough values of $N$, where $\eta_N :=\log\big(\prod_{k=1}^N e^{\xi_k}\big)$. This is the case because for $N \geq n$ we have
		\[ j^n \eta_N = j^n \log\bigg(\prod_{k=1}^N e^{\xi_k}\bigg) = \log\bigg(\prod_{k=1}^N e^{j^n(\xi_k)}\bigg) = \log\bigg(\prod_{k=1}^n e^{j^n(\xi_k)}\bigg) = j^n \eta_n\]
		where it was used that $j^n(\xi_k) = 0$ for all $k > n$, because $\xi_k \in I^{k}$. Thus $\eta = \lim_N \eta_N$ is well-defined and satisfies $j^n \eta = j^n\eta_n$ for all $n \in \N$. Let $\sigma \in L_I^\chi$. Using the fact that 
		\begin{align*}
			\xi_{n+1}.\eta_n.\sigma 
			&= e^{\ad_{\xi_{n+1}}}(\eta_n.\sigma + D_\chi) - D_\chi\\
			&= e^{\ad_{\xi_{n+1}}}e^{\ad_{\eta_n}}(D_\chi + \sigma) - D_\chi = e^{\ad_{\eta_{n+1}}}(D_\chi+\sigma) - D_\chi = \eta_{n+1}.\sigma,	
		\end{align*}
		it follows by induction that for any $n\in \N$, the equality $\eta_n.\sigma = \xi_{n}.\xi_{n-1}\cdots \xi_1.\sigma$ is valid. We thus get:
		\[j^n(\eta.\sigma) = j^n(\eta_{n}.\sigma) = j^n(\xi_n. \xi_{n-1}\cdots \xi_1.\sigma), \qquad \forall n \in \N.\qedhere\]
	\end{proof}

	\begin{proposition}\label{prop: normal_form_phi_if_obstruction_vanishes}
		Assume that $H^1(\p, P^k(V) \otimes_\chi \mfr{k}) =0$ for every $k \in \N$. Then every degree-1 MC-element in $L_I^\chi$ is gauge-equivalent to $0$ in $L_I^\chi$.
	\end{proposition}
	\begin{proof}
		Fix a MC-element $\sigma \in L_I^\chi$. Recall that $j^n(\zeta.\sigma)$ is again a MC-element in $L_{I_n}^\chi$ for any $n \in \N$ and $\zeta \in L_I^{\chi}$ of degree $0$. Notice also that $j^0\sigma = 0$. As $H^1(\p, P^k(V) \otimes_\chi \mfr{k}) =0 $ for every $k \in \N$, it follows using \Fref{lem: extension_of_mc_elements} and the exact sequences \eqref{eq: exact_seq_dglas}, by induction on $n \in \N$, that we can find a sequence of degree-$0$ elements $(\xi_k)_{k \in \N}$ in $L_I^\chi$ with $\xi_k \in P^k(V) \otimes \mfr{k}$ such that $j^{n}(\xi_n. \xi_{n-1}\cdots \xi_1.\sigma) = 0$ for every $n \in \N$. It follows from \Fref{lem: limit_gauge_transf} that there is some $\eta \in I \otimes_\chi \mfr{k}$ such that $j^n(\eta.\sigma) = \xi_n.\xi_{n-1}.\cdots.\xi_1.\sigma = 0$ in $L_I^\chi$ for all $n \in \N$. Thus $\sigma \sim 0$.
	\end{proof}
	
	\begin{lemma}\label{lem: no_homs_noncompact_to_compact}
		Let $\p$ be a semisimple Lie algebra with no nontrivial compact ideals. If $\mfr{k}$ is a compact semisimple Lie algebra, then there are no non-trivial homomorphisms $\p \to \mfr{k}$.
	\end{lemma}
	\begin{proof}
		Let $\chi : \p \to \mfr{k}$ be a homomorphism. Then $\p/\ker\chi$ is isomorphic to a subalgebra of $\mfr{k}$ and is therefore compact. As $\p$ is semisimple and has no nontrivial compact ideals, it also has no non-trivial compact quotients. Thus $\p/\ker \chi = \{0\}$ or equivalently $\p = \ker \chi$, so $\chi$ is trivial.
	\end{proof}

	\begin{proposition}\label{prop: lift_easy_when_p_semisimple}
		Assume that $\p$ is semisimple. Let $\chi : \p \to \mfr{k}$ be a homomorphism and let $\sigma^\prime \in L_{I}^\chi$. Suppose that $\sigma := \chi + \sigma^\prime$ is a degree-$1$ MC-element in $L_R$. Then $\sigma$ is gauge-equivalent to $\chi$ in $L_R$. 
	\end{proposition}
	\begin{proof}
		Since $H^1(\p, P^k(V) \otimes_\chi \mfr{k}) = 0$ for all $k \in \N_{\geq 0}$ by Whitehead's Lemma \cite[III.7.\ Lem.\ 3]{Jacobson}, \Fref{prop: normal_form_phi_if_obstruction_vanishes} implies that $\sigma$ is equivalent to $0$ in $L_{I}^\chi$. Equivalently $\chi + \sigma$ is equivalent to $\chi$ in $L_R$.
	\end{proof}

	\normalformverticaltwistss*

	\begin{proof}
		Since $\sigma \in \p^\ast \otimes \g$ is a MC-element in $L_R$, there is some degree-1 MC-element $\sigma^\prime \in L_I^{\sigma_0}$ such that $\sigma = \sigma_0 + \sigma^\prime$, by \Fref{lem: splitting_off_constants}. Then \Fref{prop: lift_easy_when_p_semisimple} implies that $\sigma$ is gauge-equivalent to $\sigma_0$. By \Fref{lem: mc_element_constant_term} we further know that $\sigma_0 : \p \to \mfr{k}$ is a homomorphism of Lie algebras. Thus, if $\p$ has no non-compact ideals then $\sigma_0$ is trivial by \Fref{lem: no_homs_noncompact_to_compact}. 
	\end{proof}

	\begin{remark}
		Alternatively, \Fref{thm: normalformverticaltwistss} also follows from the structure theory of pro-Lie algebras, developed in \cite{Hofmann_pro_lie_groups}. To see this, assume that $\p$ is semisimple. Consider the pro-Lie algebra $\h \rtimes_{D_0} \p$, were $\h := I \otimes \mfr{k} \subseteq \g$ and where $D_0 : \p \to \der(\h)$ is given by $D_0(p) = -\mc{L}_{\bm{v}(p)} + \ad_{\sigma_0(p)}$ for $p \in \p$. Since $\p$ is semisimple, the radical and Levi-factor of $\h \rtimes_{D_0} \p$ are $\h$ and $\p$, respectively. A Levi subalgebra of $(\g \rtimes_{D_0} \p)$ is equivalently given by a splitting of the exact sequence
		\begin{equation}
			0 \to \h \to \h \rtimes_{D_0} \p \to \p \to 0,
		\end{equation}
		which in turn is equivalently given by a linear map $\sigma^\prime : \p \to \h$ satisfying the Maurer-Cartan equation
		\[ \sigma^\prime([p_1, p_2]) = [\sigma^\prime(p_1), \sigma^\prime(p_2)] + D_0(p_1)\sigma^\prime(p_2) - D_0(p_2)\sigma^\prime(p_1), \qquad \forall p_1, p_2 \in \p.\]
		That is, by a degree-1 MC-element $\sigma^\prime$ in the DGLA $L_I^{\sigma_0}$. The splitting $s_{\sigma^\prime}$ and Levi subalgebra $\mfr{l}_{\sigma^\prime}$ corresponding to $\sigma^\prime$ are given by $s_{\sigma^\prime} : \p \to \h \rtimes_{D_0} \p, \; s_{\sigma^\prime}(p) := (\sigma^\prime(p), p)$, and $\mfr{l}_{\sigma^\prime} := \set{ (\sigma^\prime(p), p) \st p \in \p} \subseteq \h \rtimes_{D_0} \p$, respectively. Any two Levi subalgebras in $\h \rtimes_{D_0} \p$ are conjugate by an automorphism of the form $e^{\ad_\xi}$ for some $\xi \in \h$, by \cite[Thm.\ 7.77$(i)$]{Hofmann_pro_lie_groups}. So if $\sigma^\prime \in L_I^{\sigma_0}$ is a degree-1 MC-element, there exists $\xi \in \h$ such that $e^{\ad_\xi}(\sigma^\prime(p), p) = (0, p)$ for all $p \in \p$. Notice for $p \in \p$ that $e^{\ad_\xi}(\sigma^\prime(p), p) = ((\xi.\sigma^\prime)(p), p)$, where 
		\[(\xi . \sigma^\prime)(p) = e^{\ad_\xi}\sigma^\prime(p) + F(\ad_\xi)(D_0(p)\xi) = e^{\ad_\xi}\sigma^\prime(p) + F(\ad_\xi)(\delta_{\sigma_0}(\xi)(p))\]
		is precisely the gauge action of the degree-zero elements $(L_I^{\sigma_0})^0 = \h$ on $L_I^{\sigma_0}$. We thus find that $\xi . \sigma^\prime = 0$, so $\sigma^\prime \sim 0$ in $L_I^{\sigma_0}$. By \Fref{lem: mc_element_constant_term}, this is equivalent with $\sigma \sim \sigma_0$ in $L_R$.\\
	\end{remark}

	\noindent
	We now prove \Fref{thm: normalformverticaltwistoned}:
	
	\normalformverticaltwistoned*
	\begin{proof}~
		\begin{enumerate}
			\item Using \Fref{lem: splitting_off_constants}, write $\sigma = {\sigma_0} + \sigma^\prime$, where $\sigma^\prime \in L_I^{\sigma_0}$ is a degree-1 MC-element in the shifted DGLA $L_I^{\sigma_0}$. Passing to the complexification, observe for $n \in \N$ that 
			\[P^{n}(V_\C) \otimes_{\sigma_0} (\mfr{k}/\mfr{t})_\C \cong \bigoplus_{\alpha} P^{n}(V_\C) \otimes_{\sigma_0} (\mfr{k}_\C)_\alpha\]
			as $\p$-modules. The eigenvalues of $-\mc{L}_{\bm{v}_{\mrm{l}}} + \ad_{\sigma_0}$ acting on $P^n(V) \otimes_{\sigma_0} (\mfr{k}_\C)_\alpha$ are given by $\alpha({\sigma_0}) - \langle \bm{n}, \bm{\mu}\rangle$, as $\bm{n}$ ranges over the multi-indices $\bm{n} \in \N_{\geq 0}^d$ with $|\bm{n}| = n$, and $\alpha$ over the roots of $\mfr{k}$. Thus $-\mc{L}_{\bm{v}_{\mrm{l}}} + \ad_{\sigma_0}$ is invertible on $P^n(V) \otimes_{\sigma_0} (\mfr{k}_\C)_\alpha$. Consequently 
			\[H^0(\p, P^n(V_\C) \otimes_{\sigma_0} (\mfr{k}_\C)_\alpha) = H^1(\p, P^n(V_\C) \otimes_{\sigma_0} (\mfr{k}_\C)_\alpha) = 0\]
			for any $n\in \N$ and root $\alpha$, which in turn implies that $H^0(\p, P^n(V) \otimes_{\sigma_0} \mfr{k}/\mfr{t}) = H^1(\p, P^n(V) \otimes_{\sigma_0} \mfr{k}/\mfr{t}) = 0$. By the long exact sequence of cohomology groups associated to the short exact sequence 
			\[0 \to \bigg({\textstyle\bigwedge^\bullet}\, \p^\ast\bigg)\otimes P^n(V) \otimes \mfr{t} \to \bigg({\textstyle\bigwedge^\bullet}\, \p^\ast\bigg) \otimes P^n(V) \otimes_{\sigma_0} \mfr{k} \to \bigg({\textstyle\bigwedge^\bullet}\, \p^\ast\bigg) \otimes P^n(V) \otimes_{\sigma_0} \mfr{k}/\mfr{t} \to 0,\]
			it follows that for every $n \in \N$, the inclusion $\p^\ast \otimes P^n(V) \otimes \mfr{t} \hookrightarrow \p^\ast \otimes P^n(V) \otimes_{\sigma_0} \mfr{k}$ induces an isomorphism $H^1(\p, P^n(V)\otimes \mfr{t}) \cong H^1(\p, P^n(V) \otimes_{\sigma_0} \mfr{k})$. It then follows using \Fref{lem: extension_of_mc_elements} by induction on $n \in \N$ that we can find elements $\xi_k \in P^k(V) \otimes_{\sigma_0} \mfr{k}$ s.t.\ $j^{n}(\xi_n. \xi_{n-1}\cdots \xi_1.\sigma^\prime) \in I_n \otimes \mfr{t}$ for every $n \in \N$, the gauge action taking place in $L_{I}^{\sigma_0}$. By \Fref{lem: limit_gauge_transf} there exists $\eta \in I \otimes_{\sigma_0} \mfr{k}$ s.t.\ $j^n(\eta.\sigma^\prime) = j^n(\xi_n.\xi_{n-1}.\cdots.\xi_1.\sigma^\prime) \in I_n \otimes \mfr{t}$ for every $n \in \N$. Hence $\zeta := \eta.\sigma^\prime \in I \otimes \mfr{t}$ and $\sigma^\prime$ is gauge-equivalent to $\zeta$ in $L_I^{\sigma_0}$. By \Fref{lem: splitting_off_constants} it follows that $\sigma = {\sigma_0} + \sigma^\prime$ is gauge-equivalent to ${\sigma_0} + \zeta \in R \otimes \mfr{t}$.
			
			\item As before, decompose $\sigma = {\sigma_0} + \sigma^\prime$ using \Fref{lem: splitting_off_constants}, where $\sigma^\prime \in L_I^{\sigma_0}$ is a degree-1 MC-element in the shifted DGLA $L_I^{\sigma_0}$. Let $n \in \N$. Identify $\p^\ast \otimes P^n(V) \otimes_{\sigma_0} \mfr{k}$ with $P^n(V) \otimes_{\sigma_0} \mfr{k}$ by evaluating elements of $\p^\ast$ at $1 \in \p = \R$. This induces an isomorphism
			\[H^1(\p, P^n(V)\otimes_{\sigma_0} \mfr{k}) \cong (P^n(V)\otimes_{\sigma_0} \mfr{k}) / \mrm{Im}(-\mc{L}_{\bm{v}_{\mrm{l}}} + \ad_{\sigma_0})\]
			Since ${\bm{v}_{\mrm{l}}}$ is semisimple, so is $-\mc{L}_{\bm{v}_{\mrm{l}}} + \ad_{\sigma_0}$ as operator on $P^n(V) \otimes_{\sigma_0} \mfr{k}$. Consequently, the inclusion $(P^n(V) \otimes_{\sigma_0} \mfr{k})^\p \hookrightarrow P^n(V) \otimes_{\sigma_0} \mfr{k}$ induces an isomorphism 
			\[(P^n(V) \otimes_{\sigma_0} \mfr{k})^\p \cong (P^n(V)\otimes_{\sigma_0} \mfr{k}) / \mrm{Im}(-\mc{L}_{\bm{v}_{\mrm{l}}} + \ad_{\sigma_0}).\]
			So every element of $H^1(\p, P^n(V)\otimes_{\sigma_0} \mfr{k})$ admits a representative in $\p^\ast \otimes (P^n(V) \otimes_{\sigma_0} \mfr{k})^\p$, for any $n \in \N$. By a similar argument as in the previous item, it follows that $\sigma^\prime$ is gauge-equivalent to some $\zeta \in I \otimes_{\sigma_0} \mfr{k}$ in $L_I^{\sigma_0}$ that satisfies $-\mc{L}_{\bm{v}_{\mrm{l}}}(\zeta) + [{\sigma_0}, \zeta] = 0$. By \Fref{lem: splitting_off_constants} it follows that ${\sigma_0} + \sigma^\prime \sim {\sigma_0} + \zeta$ in $L_R$. Notice that $\nu := {\sigma_0} + \zeta$ satisfies $-\mc{L}_{\bm{v}_{\mrm{l}}}\nu + [{\sigma_0}, \nu] = 0$.
			
			\item Observe first that any derivation $D^\prime \in \der(\g)$ satisfying $D^\prime (I \otimes \mfr{k}) \subseteq (I \otimes \mfr{k})$ integrates to a unique $1$-parameter group $t \mapsto e^{t D^\prime}$ of automorphisms on $\g$ that leave the ideal $I \otimes \mfr{k} \subseteq \g$ invariant. Indeed, this follows from fact that the corresponding statement is true for the finite-dimensional Lie algebra $\der(\g_n)$ for every $n \in \N$, where we use that $\g$ is the projective limit $\g = \varprojlim_n \g_n$. In particular this applies to the $\T$-action $e^{tD}$ on $\g$, which therefore leaves $I \otimes \mfr{k}$ invariant. It thus induces a continuous $\T$-action on $V^\ast \otimes \mfr{k} \cong (I \otimes \mfr{k})/(I^2 \otimes \mfr{k})$, which integrates the linear operator $-\mc{L}_{{\bm{v}_{\mrm{l}}}} \otimes 1 + 1\otimes \ad_{\sigma_0}$ on $V^\ast \otimes \mfr{k}$. This implies that ${\bm{v}_{\mrm{l}}} \in \gl(V)$ and $\ad_{\sigma_0} \in \der(\mfr{k})$ integrate to continuous $\T = \R/2\pi\Z$-actions on $V$ and $\mfr{k}$, respectively. As $\T$ is compact it follows in particular that ${\bm{v}_{\mrm{l}}}$ is semisimple and that $\Spec({\bm{v}_{\mrm{l}}}) \cup \Spec(\ad_{\sigma_0}) \subseteq 2\pi i\Z$. By \Fref{lem: splitting_off_constants} we know that there is some degree-1 MC-element $\sigma^\prime \in L_I^{\sigma_0}$ such that $\sigma = {\sigma_0} + \sigma^\prime$. By the previous item, it follows that we may assume that $\sigma^\prime$ satisfies $-\mc{L}_{{\bm{v}_{\mrm{l}}}}\sigma^\prime + [{\sigma_0}, \sigma^\prime] = 0$, by acting with gauge transformations in $L_I^{\sigma_0}$ if necessary. Let $n \in \N$. The $\T$-action on $\g_n = R_n \otimes \mfr{k}$ must be unitarizable because $\T$ is compact, so that its generator $D_n := -\mc{L}_{{\bm{v}_{\mrm{l}}}} + [j^n\sigma, \--] \in \der(\g_n)$ must be semisimple. Notice further that $-\mc{L}_{\bm{v}_{\mrm{l}}} + [{\sigma_0}, \--]$ is semisimple on $\g_n$ whereas $[j^n\sigma^\prime, \--]$ is nilpotent. Since $-\mc{L}_{\bm{v}_{\mrm{l}}}\sigma^\prime + [{\sigma_0}, \sigma^\prime] = 0$, the operators $-\mc{L}_{\bm{v}_{\mrm{l}}} + [{\sigma_0}, \--]$ and $[j^n\sigma^\prime, \--]$ on $\g_n$ commute. Thus $D_n = \big(-\mc{L}_{\bm{v}_{\mrm{l}}} + [{\sigma_0}, \--]\big) + [j^n\sigma^\prime, \--]$ is the Jordan decomposition of $D_n$. As $D_n$ is semisimple, this implies that $[j^n\sigma^\prime, \--]=0$. Thus $j^n\sigma^\prime \in \mfr{Z}(\g_n)$, where $\mfr{Z}(\g_n)$ denotes the center of $\g_n$. As $\mfr{k}$ is simple, we know that $\mfr{Z}(\g_n) = P^n(V)\otimes \mfr{k} \subseteq \g_n$. Thus $\sigma^\prime \in I^{n-1}\otimes \mfr{k}$ for every $n \in \N$, where $I^0 := R$. As $\bigcap_{n\in \N}(I^{n-1}\otimes \mfr{k}) = \{0\}$, it follows that $\sigma^\prime = 0$. Hence $\sigma = {\sigma_0} \in \mfr{t}$.\qedhere
		\end{enumerate}
	\end{proof}

	\section{Projective Unitary G.P.E.\ Representations of $(R \otimes \mfr{k}) \rtimes_D \p$}\label{sec: restr_proj_reps}

	\noindent
	Having obtained the normal form results \Fref{thm: normalformverticaltwistss} and \Fref{thm: normalformverticaltwistoned}, we now proceed with the study of continuous projective unitary representation of jet Lie groups and algebras that are of generalized positive energy.\\
		
	\noindent
	Let us begin by briefly recalling the setting and our notation. We have that $V$ is a finite-dimensional real vector space, $R = \R\llbracket V^\ast \rrbracket := \prod_{n=0}^\infty P^k(V)$ is the ring of formal power series on $V$ with coefficients in $\R$ and equipped with the product topology. Moreover, $\g$ denotes the topological Lie algebra $\g = R \otimes \mfr{k}$, where $\mfr{k}$ is a compact simple Lie algebra and $\p$ is a finite-dimensional real Lie algebra acting on $\g$ by the homomorphism $D : \p \to \der(\g)$, which splits into a horizontal and a vertical part according to $D(p) = -\mc{L}_{\bm{v}(p)} + \ad_{\sigma(p)}$, where $\bm{v} : \p \to \X_I^{\op}$ is a homomorphism and where the linear map $\sigma : \p \to \g$ satisfies the Maurer Cartan equation \eqref{eq: lift_satisfies_mc}. Let $P$ and $K$ be the $1$-connected Lie groups integrating $\p$ and $K$, respectively. For $n \in \N$ write $G_n := J^n_0(V; K)$, $G_n^\sharp := J^n_0(V; K) \rtimes P$ and $\g_n^\sharp := \g_n \rtimes_D \p$. Define further $G := J^\infty_0(V; K) := \varprojlim_n G_{n}$, $G^\sharp :=G \rtimes_\alpha P$ and $\g^\sharp := \g \rtimes_D \p \cong \varprojlim_n \g_n^\sharp$.  \\

	\noindent
	In the following, we are interested in understanding the extent to which the linearization $\bm{v}_{\mrm{l}}$ of $\bm{v}$ and the values of $\sigma(p)$ at the origin already determine properties of the class of representations which are of g.p.e.\ at a given cone $\mfr{C} \subseteq \p$. To describe the main results, we first have to introduce some more notation.\\
	
	\noindent
	Define $\sigma_0 := \ev_0 \circ \sigma : \p \to \mfr{k}$ and let ${\bm{v}_{\mrm{l}}} = j^1 \bm{v} : \p \to \gl(V)$ be the linearization of $\bm{v}$ at the origin. For $p \in \p$, the vector fields $\bm{v}(p)$ splits as $\bm{v}(p) = \vl{p} + \vho{p}$ for some formal vector field $\vho{p} \in \X_{I^2}$ vanishing up to first order at the origin. Let $\vl{p} = \vls{p} + \vln{p}$ be the Jordan decomposition of $\vl{p}$ over $\C$. Let $V_{\mrm{c}}^\C(p)$ denote the span in $V_\C$ of all generalized eigenspaces of $\vl{p}$ corresponding to eigenvalues with zero real part. Set $V_{\mrm{c}}(p) := V_{\mrm{c}}^\C(p) \cap V$. If $\mfr{C} \subseteq \p$ is a subset, define $V_{\mrm{c}}(\mfr{C}) := \bigcap_{p \in \mfr{C}}V_{\mrm{c}}(p)$. We call $V_{\mrm{c}}(\mfr{C})$ the \squotes{center subspace associated to $\mfr{C}$}, in analogy with the center manifold of a fixed point of a dynamical system. Let $V_{\mrm{c}}(\mfr{C})^\perp\subseteq V^\ast$ denote the annihilator of $V_{\mrm{c}}(\mfr{C})$ in $V^\ast$. For any $p \in \p$, let $\Sigma_p \subseteq \C$ denote the additive subsemigroup of $\C$ generated by $\Spec(\vl{p})$. Recall from \Fref{def: qpe} that for any continuous projective unitary representation $\olpi$ of $\g\rtimes_D \p$, the set $\mfr{C}(\olpi)$ consists of all points $p \in \p$ for which $\olpi$ is of generalized positive energy at $p$.\\
	
	\noindent
	Let us describe the main results of this section. In the context of positive energy representations, we have:
	
	\drhofactpecase

	\begin{remark}
		Notice that $\Spec(\ad_{\sigma_0(p)}) = \set{\alpha(\sigma_0(p)) \st \alpha \in \Delta} \cup \{0\}$ is a finite subset of $i\R$. In particular, the condition $\Spec(\vl{p}) \cap \Spec(\ad_{\sigma_0(p)}) = \emptyset$ is satisfied if ${\bm{v}_{\mrm{l}}}$ has no purely imaginary eigenvalues. 
	\end{remark}
	
	\noindent
	This is complemented by the following results, which in particular give sufficient conditions for $\restr{\olpi}{\g}$ to factor through $\mfr{k}$, because $\R\llbracket V_{\mrm{c}}(\mfr{C})^\ast \rrbracket \otimes \mfr{k} \cong \mfr{k}$ whenever $V_{\mrm{c}}(\mfr{C}) = \{0\}$.
	
	\thmfactresultspectralcondition
	
	\begin{restatable}{theorem}{thmfactresultnformcondition}\label{thm: fact_result_nform_condition}
		Let $\mfr{t} \subseteq \mfr{k}$ be a maximal Abelian subalgebra. Let $\olpi$ be a continuous projective unitary representation of $\g^\sharp$. Let $\mfr{C} \subseteq \mfr{C}(\olpi)$. Assume for every $p \in \mfr{C}$ that $\sigma(p) \in R \otimes \mfr{t}$ and $[\vls{p},\, \vho{p}] = 0$.  Then $RV_{\mrm{c}}(\mfr{C})^\perp\otimes \mfr{k} \subseteq \ker \olpi$ and hence $\restr{\olpi}{\g}$ factors through $\R\llbracket V_{\mrm{c}}(\mfr{C})^\ast \rrbracket \otimes \mfr{k}$. 
	\end{restatable}

	\noindent
	To prove these results, we consider in \Fref{sec: second_la_cohom} the second continuous Lie algebra cohomology $H_{\ct}^2(\g \rtimes_D \p; \R)$ so as to obtain particular representatives $\omega$ of cohomology classes therein. We proceed in \Fref{sec: fact_fin_jets} to show that any irreducible smooth projective unitary representation $G \rtimes_\alpha P$ factors through the finite-dimensional $G_n \rtimes P$ for some $n \in \N$. This gives us access to techniques that are available for finite-dimensional Lie groups, and in particular to \Fref{cor: pe_vanishing_ideal}, which leads to \Fref{thm: fact_pe_case}. In \Fref{sec: case_p_R} and \Fref{sec: restr_mom_gen_case} we study the kernel of the quadratic form $\xi \mapsto \omega(D(p)\xi, \xi)$. Recalling from \Fref{cor: qpe_kernel_special_case} that
	\[ [D(p)\eta, \eta] = 0 \implies \bigg(\omega(D(p)\eta, \eta) = 0 \iff \olpi(D(p)\eta) = 0 \bigg), \qquad \forall \eta \in \g,	\]
	this leads to an ideal in $\g$ contained in $\ker \olpi$, and to the proof of \Fref{thm: fact_result_spectral_condition} and \Fref{thm: fact_result_nform_condition}. These results are supplemented in \Fref{sec: restr_mom_set_semisimple} by a consideration of the special case where $\p$ is a simple non-compact Lie algebra, in which case \Fref{thm: normalformverticaltwistss} is available. This leads to the following:
	
	\begin{restatable}{theorem}{thmfactresultss}\label{thm: factorization_result_semisimple}
		Assume that $\p$ is non-compact and simple. Let $\olpi$ be a continuous projective unitary representation of $\g \rtimes_D\p$. Write $\mfr{C} := \mfr{C}(\olpi) \subseteq \p$. Then $\restr{\olpi}{\g}$ factors through $\R\llbracket V_{\mrm{c}}(\mfr{C})^\ast \rrbracket \otimes \mfr{k}$. 
	\end{restatable}
	
	\thmsemisimplecone

	\subsection{The Second Lie Algebra Cohomology $H_{\ct}^2(\g \rtimes_D \p, \R)$.}\label{sec: second_la_cohom}
	\noindent
	We next determine suitable representatives of classes in the second Lie algebra cohomology $H^2_{\ct}(\g \rtimes_D \p, \R)$, which classifies the continuous $\R$-central extensions of $\g \rtimes \p =: \g^\sharp$ up to equivalence. As an intermediate step we first consider $H^2_{\ct}(\g, \R)$, which is completely understood. \\
	
	\noindent
	Define $\Omega_R^k := R \otimes \bigwedge^k V^\ast$, equipped with projective limit topology obtained from $\Omega_R^k = \varprojlim_n \Omega_{R_n}^k$, where $\Omega_{R_n}^k := R_n \otimes \bigwedge^k V^\ast$. This makes $\Omega_R^k$ into a Fr\'echet space. In particular $\Omega_R^0 = R$. Since $J_0^\infty(\Omega^k(V)) \cong \Omega_R^k$, we can define a continuous differential $d : \Omega_R^k \to \Omega_R^{k+1}$ by $dj^\infty_0\alpha := j^\infty_0d\alpha \in J^\infty_0(\Omega^{n+1}(V)) \cong \Omega_R^{k+1}$, which is indeed well-defined. Choosing a basis $(e_\mu)_{\mu=1}^d$ of $V$ with dual basis $(d x_\mu)_{\mu=1}^d$ of $V^\ast$, the above differential $d$ is on $R$ given by $df = \sum_\mu (\partial_\mu f) \otimes dx_\mu$ for $f \in R$.
	
	\begin{lemma}\label{lem: one_forms_universal}
		Let $E$ be a topological $R$-module and let $D : R \to E$ be a continuous derivation. Then there exists a unique continuous $R$-linear map $\overline{D} : \Omega^1_R \to E$ such that $D = \overline{D} \circ d$. 
	\end{lemma}
	\begin{proof}
		Let $\R[V^\ast]$ denote the ring of polynomial functions on $V$. As $\R[V^\ast] \subseteq R$, $E$ is also a $\R[V^\ast]$-module and $\restr{D}{\R[V^\ast]} : \R[V^\ast] \to E$ is a derivation. Using the universal property of the K\"ahler differential forms $\Omega^1_{\R[V^\ast]} \cong \R[V^\ast] \otimes V^\ast$, there is a unique $\R[V^\ast]$-linear map $\overline{D} : \Omega^1_{\R[V^\ast]} \to E$ such that $\restr{\overline{D} \circ d}{\R[V^\ast]} = \restr{D}{\R[V^\ast]}$. As $\Omega^1_{\R[V^\ast]}$ is dense in $\Omega_R^1$, it remains to extend $\overline{D}$ continuously to the Fr\'echet space $\Omega_R^1$. Let $\alpha \in \Omega^1_R$ and let $(\alpha_n)_{n\in \N}$ be a sequence in $\Omega^1_{\R[V^\ast]}$ s.t.\ $\alpha_n \to \alpha$ in $\Omega^1_R$. We show that $D\alpha := \lim_{n \to \infty}\overline{D} \alpha_n$ exists and is independent of the approximating sequence $(\alpha_n)$. Choose a basis $(dx_\mu)_{\mu=1}^d$ of $V^\ast$. Write $\alpha_n = \sum_{\mu=1}^d f_\mu^{(n)} dx_\mu$ and $\alpha = \sum_{\mu=1}^d f_\mu dx_\mu$ for some unique $f_\mu^{(n)} \in \R[V^\ast]$ and $f_\mu \in R$. Then $f_\mu^{(n)} \to f_\mu$ in $R$ for every $\mu$ and hence $\lim_{n \to \infty} \overline{D}\alpha_n  = \sum_{\mu=1}^d \lim_{n \to \infty} f_\mu^{(n)} Dx_\mu = \sum_{\mu=1}^d f_\mu Dx_\mu$, which is independent of the approximating sequence $(\alpha_n)$. It follows that $\overline{D}$ extends to a continuous $R$-linear map $D : \Omega^1_R \to E$, which satisfies $\overline{D} \circ d = D$ by construction. It is unique with these properties because its restriction to the dense subspace $\R[V^\ast]\subseteq R$ is so.
	\end{proof}

	\begin{remark}
		\Fref{lem: one_forms_universal} entails that $d : R \to \Omega_R^1$ is the universal differential module $R$ in the category of complete locally convex $\R$-modules, in the sense of \cite[Thm.\ 6]{Maier_CE_top_current_algs}
	\end{remark}
		
	\begin{proposition}\label{prop: characterization_of_cts_second_cohom}
		Any class $[\omega] \in H_{\ct}^2(\g, \R)$ has a unique representative of the form $\omega(\xi, \eta) = \lambda(\kappa(\xi, d\eta))$, where $\lambda \in {\Omega^1_R}^\ast$ is closed and continuous functional on $\Omega^1_R$ and $\kappa$ is the Killing form on $\mfr{k}$. (Closed meaning that $\lambda(dR) = 0$.) Conversely, any such $\lambda$ defines a $2$-cocycle representing some non-zero class in $H^2_{\ct}(\g, \R)$. Consequently, the center of the universal central extension of $\g$ is $(\Omega^1_R / dR)$.
	\end{proposition}
	\begin{proof}
		This is a special case of \cite[Theorem 16]{Maier_CE_top_current_algs}, seeing as $\g = R \otimes \mfr{k}$, where $R$ is a unital, associative and commutative Fr\'echet algebra and $\mfr{k}$ is a simple Lie algebra.
	\end{proof}

	\begin{lemma}\label{lem: lambda_p_invariant}
		Let $\omega$ be an extension of the 2-cocycle $\lambda(\kappa(\xi, d\eta))$ on $\g$ to a $2$-cocycle on $\g^\sharp = \g \rtimes \p$. Then $\lambda$ is $\p$-invariant in the sense that $\lambda(\mc{L}_{\bm{v}(p)} \Omega^1_R) = 0$ for every $p \in \p$.
	\end{lemma}
	\begin{proof}
		Take $\xi = f \otimes X$ and $\eta = g \otimes X$ for $f,g \in R$ and $0\neq X \in \mfr{k}$. Notice that $[\xi, \eta] = 0$ in $R \otimes\mfr{k}$. Using the cocycle identity, this implies
		\begin{align*}
			\omega(D(p)\xi, \eta) + \omega(\xi, D(p)\eta) &= \omega(D(p), [\xi, \eta]) = 0\\
			\omega([\sigma(p), \xi], \eta) + \omega(\xi, [\sigma(p), \eta]) &= \omega(\sigma(p), [\xi, \eta]) = 0.
		\end{align*}
		Using $D(p)\xi = -\mc{L}_{\bm{v}(p)}\xi + [\sigma(p), \xi]$ it follows that $0 =\omega(\mc{L}_{\bm{v}(p)} \xi, \eta) + \omega(\xi, \mc{L}_{\bm{v}(p)}\eta) = \lambda(\mc{L}_{\bm{v}(p)}fdg)\kappa(X, X)$. As $\kappa(X, X) \neq 0$ and $R dR = \Omega^1_R$, this shows the claim.
	\end{proof}

		\begin{lemma}\label{lem: lambda_factors_through_finite_jets}
		Let $\omega : \g^\sharp \times \g^\sharp \to \R$ be a continuous $2$-cocycle on $\g^\sharp = \g \rtimes_D \p$. Then there exists $n \in \N$ and a $2$-cocycle $\omega_n$ on $\g_n^\sharp$ such that $\omega(\xi, \eta) = \omega_n(j^n\xi, j^n\eta)$ for all $\xi, \eta \in \g^\sharp$.
	\end{lemma}
	\begin{proof}
		Let $\omega : \g^\sharp \times \g^\sharp \to \R$ be a continuous $2$-cocycle on $\g^\sharp$. Choose norms $\norm{\--}_{n}$ on the finite-dimensional Lie algebras $\g_n^\sharp$ s.t.\ the quotient maps $j^n : \g_m^\sharp \to \g_n^\sharp$ are contractive for any $n,m \in \N$ with $n \leq m$. The topology on $\g^\sharp = \varprojlim \g_n^\sharp$ is specified by the seminorms $\xi \mapsto \norm{j^n\xi}_n$ for $n \in \N$ and $\xi \in \g^\sharp$. As $\omega$ is continuous and the maps $\g_m^\sharp \to \g_n^\sharp$ are contractive for $n \leq m$, there exist $n \in \N$ such that $|\omega(\xi, \eta)| \leq \norm{j^n\xi}_n \norm{j^n\eta}_n$ for all $\xi, \eta \in \g^\sharp$ (e.g.\ using \cite[Prop. 43.1 and Prop.\ 43.4]{Treves_tvs}). As $j^n : \g^\sharp \to \g_n^\sharp$ is surjective, it follows that $\omega(\xi, \eta) = \omega_n(j^n \xi, j^n \eta)$ for a unique $2$-cocycle $\omega_n$ on $\g_n^\sharp$. 
	\end{proof}

	\subsection{Factorization Through Finite Jets}\label{sec: fact_fin_jets}
	
	\noindent
	In the context of smooth projective unitary representations $\overline{\rho}$ of the Lie group $G^\sharp$, it is no loss of generality to consider the case where $\overline{\rho}$ factors through the finite-dimensional Lie group $G_n^\sharp$ for some $n \in \N$:
	
	\begin{theorem}\label{thm: always_factor_through_fin_jet}
		Let $\overline{\rho}$ be a smooth projective unitary representation of $G^\sharp$ with lift $\rho : \circled{G} \to \U(\mcH_\rho)$ for some central $\T$-extension $\circled{G}$ of $G^\sharp$. Then $\rho$ decomposes as a (possibly uncountable) direct sum $\rho = \bigoplus_{i \in \mc{I}}\rho_i$ s.t.\ for every $i \in \mc{I}$ there exists $n \in \N$ s.t.\ the projective unitary representations $\overline{\rho}_i$ associated to $\rho_i$ factors through $G_n^\sharp$.  In particular, if $\overline{\rho}$ is irreducible then it factors through $G_n^\sharp$ for some $n \in \N$.
	\end{theorem}
	\begin{proof}
		Write $N_{m}^\sharp := \ker\bigg(j^{m} : G^\sharp \to G_{m}^\sharp\bigg)$ and $\n_{m}^\sharp := \ker\bigg(j^{m} : \g^\sharp \to \g_{m}^\sharp\bigg)$ for any $m \in \N_{\geq 0}$, so that $G_m^\sharp \cong G^\sharp/N_m^\sharp$ for any $m \in \N_{\geq 0}$. Notice that $N_m^\sharp \subseteq N_{n}^\sharp$ whenever $n\leq m$. Since $\overline{\rho}$ is a smooth projective representation, it follows from \cite[Thm.\ 4.3]{BasNeeb_ProjReps} that $\circled{G}$ is a Lie group. It is moreover regular by \cite[Thm.\ V.I.8]{neeb_towards_lie}, because both $G^\sharp$ and $\T$ are so. Let $\circled{\g} := \Lie(\circled{G})$. Then $\circled{\g}$ is a central $\R$-extension of $\g^\sharp$ in the category of locally convex Lie algebras. Let the continuous $2$-cocycle $\omega : \g^\sharp \times \g^\sharp \to \R$ represent the corresponding class in $H^2_{\ct}(\g^\sharp, \R)$. By \Fref{lem: lambda_factors_through_finite_jets}, there is some $n \in \N$ such that for all $\xi \in \g^\sharp$ we have $j^{n}\xi = 0 \implies \omega(\xi, \eta) = 0$ for all $\eta \in \g^\sharp$. Let $\circled{N}_n$ be the closed normal subgroup of $\circled{G}$ covering $N_n^\sharp$ and let $\circled{\n}_n$ be its Lie algebra. Then $\circled{N}_n$ is a central $\T$-extension of $N_n^\sharp$ integrating $\circled{\n}_n$. Since $\restr{\omega}{\n_n^\sharp \times \n_n^\sharp} = 0$, the central $\R$-extension $\circled{\n}_n$ is trivial. Hence $\circled{\n}_n \cong \R \oplus \n_n^\sharp$ as central $\R$-extensions of $\n_n^\sharp$. As $N_n^\sharp$ is regular and $1$-connected, it follows from \cite[Thm.\ III.1.5]{neeb_towards_lie} that there is a commutative diagram
		\[
		\begin{tikzcd}
			{\R} \arrow{r} \arrow{d}{e^{2\pi i \, \bcdot}} & {\R \times N_n^\sharp} \arrow{r} \arrow{d}{\widetilde{\phi}} & {N_n^\sharp} \arrow{d}{\id} \\
			{\T} \arrow{r}           & {\circled{N}_n} \arrow{r}           & {N_n^\sharp}          
		\end{tikzcd}
		\]
		of locally convex regular Lie groups. Observe that $\widetilde{\phi}$ is surjective and that $\ker \widetilde{\phi} = \Z$. Thus $\circled{N}_n \cong \T \times N_n^\sharp$ as central $\T$-extension of $N_n^\sharp$. Let $\phi : \T \times N_n^\sharp \to \circled{N}$ realize the isomorphism. For any integer $m \geq n$, let $\mc{N}_m := \phi(\{1\} \times N_m^\sharp) \subseteq \circled{N}_n \subseteq \circled{G}$, which is a closed normal subgroup of $\circled{G}$ isomorphic to and covering $N_m^\sharp$. Then $\mc{N} := \{\mc{N}_m\}_{m \geq n}$ is a filter basis of (decreasing) closed normal subgroups of $\circled{G}$ satisfying $\varinjlim \mc{N} = \{1\}$, in the sense that for any $1$-neighborhood $U$ of $\circled{G}$ there exists $m \geq n$ such that $\mc{N}_m \subseteq U$. Indeed, since $G^\sharp = \varprojlim_m G_m^\sharp$ carries the projective limit topology and $\circled{G}$ is a locally trivial principal $\T$-bundle over $G^\sharp$ \cite[Thm.\ 4.3]{BasNeeb_ProjReps}, it follows that any $1$-neighborhood $U \subseteq \circled{G}$ contains $\phi(I \times {N}_m^\sharp)$ for large enough $m$ and some open $1$-neighborhood $I \subseteq \T$. It now follows from \cite[Thm.12.2]{Neeb_diffvectors} that $\rho$ decomposes as a possibly uncountable direct sum $\rho\cong \bigoplus_{i \in \mc{I}} \rho_i$ such that for every $i \in \mc{I}$ there exists some $m \geq n$ with $\rho_i(\mc{N}_m) = \{1\}$, which implies that $\overline{\rho}_i(N_m^\sharp) = \{1\}$.
	\end{proof}

	\noindent
	\Fref{thm: always_factor_through_fin_jet} gives us access to techniques that are available for finite-dimensional Lie groups, and in particular to \Fref{cor: pe_vanishing_ideal}. This can be used to prove \Fref{thm: fact_pe_case}.

	\drhofactpecase*
	
	\noindent
	To prove \Fref{thm: fact_pe_case}, it suffices by \Fref{thm: always_factor_through_fin_jet} to consider the case where $\overline{\rho}$ factors through the finite-dimensional Lie group $G_k^\sharp$ for some $k \in \N$, which we thus assume. Write $\mfr{a}$ for the ideal in $\g_k^\sharp$ generated by $p \in \p$. Let $\mfr{a}_n \subseteq \mfr{a}$ denote the maximal nilpotent ideal in $\mfr{a}$. According to \Fref{cor: pe_vanishing_ideal} we have $[\mfr{a}, [\mfr{a}_n, \mfr{a}_n]] \subseteq \ker d\overline{\rho}$. Recall that $I_k := I/I^{k+1}$.

	\begin{lemma}\label{lem: computation_for_an}
		Suppose that $V^\ast \otimes \mfr{k} \subseteq j^1 \mfr{a}_n$. Then $I_k \otimes \mfr{k} \subseteq \mfr{a}_n$. 
	\end{lemma}
	\begin{proof}
		By assumption $V^\ast \otimes \mfr{k} \subseteq \mfr{a}_n + I_k^2 \otimes \mfr{k}$. As $\mfr{k}$ is perfect it follows that
		\[I_k^{l+1}\otimes \mfr{k} = [V^\ast \otimes \mfr{k}, I_k^l \otimes \mfr{k}] \subseteq \mfr{a}_n + [I_k^2 \otimes \mfr{k}, I_k^l \otimes \mfr{k}] = \mfr{a}_n + I_k^{l+2}\otimes \mfr{k}, \qquad \forall l \in \N.\]
		Thus it follows by induction that $V^\ast \otimes \mfr{k} \subseteq \mfr{a}_n + I_k^{l+1}\otimes \mfr{k}$ for all $l \in \N$. As $\bigcap_l (\mfr{a}_n + I_k^{l+1}\otimes \mfr{k}) = \mfr{a}_n$, it follows that $V^\ast \otimes \mfr{k} \subseteq \mfr{a}_n$ and hence $I_k \otimes \mfr{k} \subseteq \mfr{a}_n$.
	\end{proof}

	\begin{proof}[Proof of \Fref{thm: fact_pe_case}]
		We may assume that $\overline{\rho}$ factors through $G_k^\sharp$ for some $k \in \N$. It suffices to show that $d\overline{\rho}$ factors through $\g_2^\sharp$ and that the image of $-\mc{L}_{\vl{p}} + \ad_{\sigma_0(p)}$ in $P^2(V) \otimes \mfr{k} \subseteq \g_2$ is contained in $\ker d\overline{\rho}$. By \Fref{cor: pe_vanishing_ideal} we know that $\big[\mfr{a}, [\mfr{a}_n, \mfr{a}_n]\big] \subseteq \ker d\overline{\rho}$. Moreover $I_k^3 \otimes \mfr{k} = [I_k \otimes \mfr{k}, [I_k \otimes \mfr{k}, I_k \otimes \mfr{k}]]$, because $\mfr{k}$ is perfect. To see that $d\overline{\rho}$ factors through $\g_2^\sharp$ it thus suffices to show that $I_k \otimes \mfr{k} \subseteq \mfr{a}_n$. By \Fref{lem: computation_for_an} it is further sufficient to show that $V^\ast \otimes \mfr{k} \subseteq j^1(\mfr{a}_n)$. Write $D_1(p) := -\mc{L}_{\vl{p}} + [\sigma_0(p), \--]$. Notice that $j^1(D(p)\xi) = D_1(p) \xi$ for $\xi \in V^\ast \otimes \mfr{k}$. The assumption $\Spec(\vl{p}) \cap \Spec(\ad_{\sigma_0(p)}) = \emptyset$ implies that $D_1(p)$ is invertible on $V^\ast \otimes \mfr{k}$. Thus if $\eta \in V^\ast \otimes \mfr{k}$ is arbitrary, there exists $\xi \in V^\ast \otimes \mfr{k}$ such that $\eta = D_1(p) \xi$. Then $\eta = D_1(p) \xi = j^1(D(p) \xi) = j^1([p,\xi]) \in j^1(\mfr{a}_n)$. Thus $V^\ast \otimes \mfr{k} \subseteq j^1(\mfr{a}_n)$. We obtain that $I_k \otimes \mfr{k} \subseteq \mfr{a}_n$ and $I_k^3 \otimes \mfr{k} \subseteq \ker d\overline{\rho}$, so $d\overline{\rho}$ factors through $\g_2^\sharp$. We may thus assume that $k=2$. We then obtain 
		\[D_1(p)(P^2(V)\otimes \mfr{k}) = D(p)(I_k^2 \otimes \mfr{k}) = [p, [I_k \otimes \mfr{k}, I_k \otimes \mfr{k}]] \subseteq [\mfr{a}, [\mfr{a}_n, \mfr{a}_n] \subseteq \ker d\overline{\rho}.\qedhere\]
	\end{proof}

	\subsection{The Case Where $\p = \R$}\label{sec: case_p_R}
	
	\noindent
	We proceed with the study of projective unitary representations $\olpi$ of $\g^\sharp = \g \rtimes_D \p$ which are of generalized positive energy. We first specialize to the case where $\p = \R$, aiming to consider its consequences for the general case afterwards. \\
	
	\noindent
	As $\p = \R$, we may as well identify $\bm{v}$ with $\bm{v}(1) \in \X_I$, $D$ with $D(1)$ and $\sigma$ with $\sigma(1) \in \g$. Recall that the derivation $D$ is given by $D = -\mc{L}_{\bm{v}} + [\sigma, \--]$. Write $\bm{v}  = \bm{v}_{\mrm{l}} + \bm{v}_{\mrm{ho}}$, where $\bm{v}_{\mrm{l}} := j^1 \bm{v} \in \gl(V)$ is the linearization of $\bm{v}$ at $0 \in V$ and where $\bm{v}_{\mrm{ho}} \in \X_{I^2}$. Let $\bm{v}_{\mrm{l}} = \bm{v}_{\mrm{l,s}} + \bm{v}_{\mrm{l,n}}$ denote the Jordan decomposition of $\bm{v}_{\mrm{l}}$ over $\C$. Write $V_{\mrm{c}}^\C$ for the span of the eigenspaces of $\bm{v}_{\mrm{l}}$ whose corresponding eigenvalue has zero real part. Set $V_{\mrm{c}} := V_{\mrm{c}}^\C \cap V$. Write $\bm{d} := (0,1) \in \g^\sharp = \g \rtimes_D \R$. Let $V_{\mrm{c}}^\perp \subseteq V^\ast$ denote the annihilator of $V_{\mrm{c}}$ in $V^\ast$, so $V_{\mrm{c}}^\perp \cong (V/V_{\mrm{c}})^\ast$.
	
	\begin{restatable}{theorem}{thmcentersubspacefactorization}\label{thm: general_factorization_result}
		Let $\mfr{t} \subseteq \mfr{k}$ be a maximal Abelian subalgebra. Assume that $\sigma \in R \otimes \mfr{t} \subseteq \g$ and $[\bm{v}_{\mrm{l,s}}, \bm{v}_{\mrm{ho}}] = 0$ in $\X_I$. Let $\olpi$ be a continuous projective unitary representation of $\g^\sharp$ on the pre-Hilbert space $\mc{D}$. Assume that $\olpi$ is of g.p.e.\ at $\bm{d} \in \g \rtimes_D \R\bm{d}$. Then $RV_{\mrm{c}}^\perp \subseteq \ker \olpi$. Consequently, $\restr{\olpi}{\g}$ factors through $\R\llbracket V_{\mrm{c}}^\ast \rrbracket \otimes \mfr{k}$. In particular, if $V_{\mrm{c}} = \{0\}$ then $\restr{\olpi}{\g}$ factors through $\mfr{k}$.
	\end{restatable}

	\begin{remark}
		By acting with formal diffeomorphisms if necessary, one may by \Fref{thm: poincare_dulac} always bring $\bm{v}$ into a normal form, in the sense that $[\bm{v}_{\mrm{l,s}},\bm{v}_{\mrm{ho}}] = 0$ in $\X_I$. Moreover, \Fref{thm: normalformverticaltwistoned} provides sufficient conditions guaranteeing that $\sigma$ is gauge equivalent to some element in $R\otimes \mfr{t}$.
	\end{remark}

	\subsubsection*{Proof of \Fref{thm: general_factorization_result}}
	
	\noindent
	Let $\omega$ be a continuous $2$-cocycle on $\g^\sharp$ that represents the class in $H^{2}_{\ct}(\g^\sharp, \R)$ corresponding to the central $\R$-extension of $\g^\sharp$ obtained from $\olpi$ by pulling back $\mfr{u}(\mcD) \to \pu(\mcD)$ along $\olpi$. In view of \Fref{prop: characterization_of_cts_second_cohom} and \Fref{lem: lambda_p_invariant}, we may and do assume that $\omega$ satisfies $\omega(\xi, \eta) = \lambda(\kappa(\xi, d\eta))$ for any $\xi, \eta \in \g$, where $\lambda : \Omega_R \to \R$ is continuous, $\p$-invariant and closed. We write $fX$ instead of $f \otimes X$ for $f \in R_\C$ and $X \in \mfr{k}_\C$. Let $\Delta\subseteq i\mfr{t}^\ast$ denote the set of roots of $\mfr{k}$. Finally, write $\h := \mfr{t}_\C \subseteq \mfr{k}_\C$. Recall from \Fref{cor: qpe_kernel_special_case} that
	\begin{equation}\label{eq: qpe_kernel_special_case_repeat}
		 [D\eta, \eta] = 0 \implies \bigg(\omega(D\eta, \eta) = 0 \iff \olpi(D\eta) = 0 \bigg), \qquad \forall \eta \in \g.		
	\end{equation}
	Moreover, $\omega(D\eta, \eta) \geq 0$ whenever $[D\eta, \eta] = 0$. In the present setting, this yields:
	
	\begin{proposition}\label{prop: kernel_of_repr_and_quadr_form}
		Fix $f \in R$. Then $\olpi(R\mc{L}_{\bm{v}}f \otimes \mfr{k}) = \{0\} \iff \lambda(fd\mc{L}_{\bm{v}}f) = 0$.
	\end{proposition}
	\begin{proof}
		\noindent
		For any $H \in \mfr{t}$, observe that $DfH = -\mc{L}_{\bm{v}}fH$ because $\sigma \in R \otimes \mfr{t}$, so $[DfH, fH] = -[\mc{L}_{\bm{v}}fH, fH] = 0$. Using \eqref{eq: qpe_kernel_special_case_repeat} we obtain that
		\begin{equation}\label{eq: kernel_rerp_1temp}
			\kappa(H,H)\lambda(\mc{L}_{\bm{v}}(f)df) = 0 \iff \olpi(\mc{L}_{\bm{v}}fH) = 0, \qquad \forall H \in \mfr{t}. 
		\end{equation}
		Assume that $\olpi(R\mc{L}_{\bm{v}}f \otimes \mfr{k}) = \{0\}$. Then $\olpi(\mc{L}_{\bm{v}}fH) = 0$ for any $H \in \mfr{t}$, so $\lambda(\mc{L}_{\bm{v}}(f)df) = 0$ by \eqref{eq: kernel_rerp_1temp}. Conversely, suppose that $\lambda(\mc{L}_{\bm{v}}(f)df) = 0$. Then $\olpi(\mc{L}_{\bm{v}}fH) = 0$ for all $H \in \mfr{t}$, by \eqref{eq: kernel_rerp_1temp}. Taking the commutator with $\olpi(gX_\alpha)$, where $g \in R$, $\alpha \in \Delta$ is a root and $X_\alpha \in (\mfr{k}_\C)_\alpha$ is a corresponding root vector, it follows that
		\begin{equation}\label{eq: lem_kernel_intermediate_1}
			\olpi(g\mc{L}_{\bm{v}}f X_\alpha) = 0 \qquad \forall X_\alpha \in (\mfr{k}_{\C})_{\alpha}, \; g \in R.	
		\end{equation}
		Take $X_{\alpha} \in (\mfr{k}_\C)_{\alpha}$ and $Y_{-\alpha} \in (\mfr{k}_\C)_{-\alpha}$. Write $H_\alpha = [X_\alpha, Y_{-\alpha}]$. By taking commutators with $\olpi(1\otimes Y_{-\alpha})$ in \fref{eq: lem_kernel_intermediate_1} we find that $\olpi(g\mc{L}_{\bm{v}}f H_{\alpha}) = 0$. As $\h = \sum_{\alpha} [(\mfr{k}_{\C})_{\alpha}, (\mfr{k}_{\C})_{-\alpha}]$, this shows by linearity together with \fref{eq: lem_kernel_intermediate_1} and the root space decomposition that $\olpi(R\mc{L}_{\bm{v}}f~\otimes~\mfr{k})~=~\{0\}$.
	\end{proof}

	\noindent
	Define the quadratic form $q(f) := \lambda(\mc{L}_{\bm{v}}(f)df) = -\lambda(fd\mc{L}_{\bm{v}}f)$ on $R$. Let $\mc{N} := \ker q$ denote its kernel. By \Fref{prop: kernel_of_repr_and_quadr_form}, $\mcN$ generates an ideal $J \otimes \mfr{k}$ on which $\olpi$ vanishes, where $J := R\mc{L}_{\bm{v}}\mcN$.

	\begin{corollary}\label{cor: vanishing_ideal_from_full_quadratic_form} 
		Set $J := R \mc{L}_{\bm{v}}\mcN$. Then $J\otimes \mfr{k} \subseteq \ker(\olpi)$.
	\end{corollary}

	\noindent
	Together with the fact that $\lambda$ vanishes on exact forms and is $\mc{L}_{\bm{v}}$-invariant, this puts severe restrictions on the representation $\olpi$ and leads to \Fref{thm: general_factorization_result}. Let us also remark the following:
	
	\begin{lemma}
		The bilinear form $\beta_q(f,g) := \lambda(\mc{L}_{\bm{v}}(f)dg)$ on $R$ associated to $q$ is symmetric, the quadratic form $q$ is positive semi-definite and $\mc{N} = \set{f \in R \st \beta_q(f,g) = 0 \;\; \forall g \in R}$.
	\end{lemma}
	\begin{proof}
		As $\lambda$ is closed and $\mc{L}_{\bm{v}}$-invariant, it follows that $\beta$ is symmetric. To see that $q$ is positive semi-definite, let $f \in R$ and $0 \neq H \in \mfr{t}$. Write $\eta := fH$ and notice that $[D\eta, \eta] = 0$. By \Fref{cor: qpe_kernel_special_case} we know that $-\kappa(H,H)\lambda(\mc{L}_{\bm{v}}(f)df) = \omega(D\eta, \eta) \geq 0$. As $\kappa$ is negative definite on $\mfr{k}$ we obtain that $q$ is positive semi-definite. It follows that $|\beta_q(f,g)|^2 \leq q(f)q(g)$, which implies $\mc{N} = \set{f \in R \st \beta_q(f,g) = 0 \qquad \forall g \in R}$.
	\end{proof}

	\noindent
	The following observation is also noteworthy, although it will not be used:
	
	\begin{lemma}\label{lem: ker_qp_subalg}
		$\mc{N} \subseteq R$ is a subalgebra.
	\end{lemma}
	\begin{proof}
		Let $f,g \in \mc{N}$. Then using the Leibniz rule and \Fref{prop: kernel_of_repr_and_quadr_form} we obtain 
		\[\olpi(R\mc{L}_{\bm{v}}(fg)\otimes \mfr{k}) \subseteq \olpi(fR\mc{L}_{\bm{v}}g\otimes \mfr{k}) +  \olpi(gR\mc{L}_{\bm{v}}f\otimes \mfr{k}) \subseteq \{0\},\]
		Applying \Fref{prop: kernel_of_repr_and_quadr_form} once more, we conclude that $fg \in \mc{N}$.
	\end{proof}
	
	\begin{lemma}\label{lem: lambda_invariant_under_linear_ss_part}
		$\lambda \circ \mc{L}_{\bm{v}_{\mrm{l,s}}} = 0$.
	\end{lemma}
	\begin{proof}
		 As $\lambda : \Omega^1_R \to \R$ is continuous, it factors through the finite-dimensional space $\Omega^1_{R_k} = R_k \otimes V^\ast$ for some $k \in \N$. Notice that both $\mc{L}_{\bm{v}_{\mrm{l,n}}}$ and $\mc{L}_{\bm{v}_{\mrm{ho}}}$ are nilpotent on $\Omega^1_{R_k} \otimes_\R \C$, whereas $\mc{L}_{\bm{v}_{\mrm{l,s}}}$ is semisimple on it. Also $[\mc{L}_{\bm{v}_{\mrm{l,s}}}, \mc{L}_{\bm{v}_{\mrm{l,n}}} + \mc{L}_{\bm{v}_{\mrm{ho}}}] = 0$ because $[\bm{v}_{\mrm{l,s}}, \bm{v}_{\mrm{ho}}] = [\bm{v}_{\mrm{l,s}}, \bm{v}_{\mrm{l,n}}] = 0$. Thus $\mc{L}_{\bm{v}} = \mc{L}_{\bm{v}_{\mrm{l,s}}} + \big(\mc{L}_{\bm{v}_{\mrm{l,n}}} + \mc{L}_{\bm{v}_{\mrm{ho}}}\big)$ is the Jordan decomposition of $\mc{L}_{\bm{v}}$ acting on $\Omega^1_{R_k} \otimes_\R \C$. Thus $\mrm{Im}(\mc{L}_{\bm{v}_{\mrm{l,s}}}) \subseteq \mrm{Im}(\mc{L}_{\bm{v}})$ when $\mc{L}_{\bm{v}_{\mrm{l,s}}}$ and $\mc{L}_{\bm{v}}$ are considered as operators on $\Omega^1_{R_k} \otimes_\R \C$. As $\lambda$ is $\mc{L}_{\bm{v}}$-invariant, we know $\lambda \circ \mc{L}_{\bm{v}} = 0$. Thus $\lambda \circ \mc{L}_{\bm{v}_{\mrm{l,s}}} = 0$.
	\end{proof}
	
	\noindent
	In particular, $\lambda$ vanishes on the eigenspaces in $\Omega^1_{R_\C}$ of $\mc{L}_{\bm{v}_{\mrm{l,s}}}$ corresponding to non-zero eigenvalues. We introduce some more notation. Let $E_\C$ denote the span of all eigenspaces in $R_\C$ of $\mc{L}_{\bm{v}_{\mrm{l,s}}}$ corresponding to eigenvalues with non-zero real part. Define $E := E_\C \cap R$ and $E^n := E \cap I^n$. 
	
	\begin{lemma}\label{lem: center_subspace_in_kernel}
		$E \subseteq \mcN$. 
	\end{lemma}
	\begin{proof}
		Let $\mu \in \Spec(\mc{L}_{\bm{v}_{\mrm{l,s}}})$ with $\mrm{Re}(\mu) \neq 0$. Set $E_\mu := \ker(\mc{L}_{\bm{v}_{\mrm{l,s}}} - \mu I) \subseteq R_\C$. Suppose first that $\mu \in \R$. If $\psi \in E_\mu \cap R$ then because $\mc{L}_{\bm{v}}$ leaves the eigenspaces of $\mc{L}_{\bm{v}_{\mrm{l,s}}}$ invariant, the 1-form $\psi d \mc{L}_{\bm{v}}\psi$ is an eigenvector of $\mc{L}_{\bm{v}_{\mrm{l,s}}}$ with non-zero eigenvalue $2\mu$. By \Fref{lem: lambda_invariant_under_linear_ss_part} it follows that $q(\psi) = 0$ and hence $\psi \in \mc{N}$. Thus $E_\mu \subseteq \mc{N}$. Next, suppose that $\mu$ is not real. Then also $\overline{\mu}$ is an eigenvalue of $\mc{L}_{\bm{v}_{\mrm{l,s}}}$. Write $W_\C := E_\mu \oplus E_{\overline{\mu}}$ and $W := W_\C\cap R$. Take $\psi \in W$ arbitrary. Then $\psi = \eta + \overline{\eta}$ for some $\eta \in E_\mu$ (and hence $\overline{\eta} \in E_{\overline{\mu}}$). As $\mu + \overline{\mu} = 2\mrm{Re}(\mu) \neq 0$ and $\mc{L}_{\bm{v}}$ leaves the eigenspaces of $\mc{L}_{\bm{v}_{\mrm{l,s}}}$ invariant, each of the 1-forms $\eta d\mc{L}_{\bm{v}}\eta,\; \eta d\mc{L}_{\bm{v}}\overline{\eta}, \; \overline{\eta} d\mc{L}_{\bm{v}}\eta$ and $\overline{\eta} d\mc{L}_{\bm{v}}\overline{\eta}$ are eigenvectors of $\mc{L}_{\bm{v}_{\mrm{l,s}}}$ with non-zero eigenvalue. Using \Fref{lem: lambda_invariant_under_linear_ss_part} it follows that $q(\psi) = 0$ and hence $\psi \in \mc{N}$. Thus $W \subseteq \mc{N}$. As $\mc{N}$ is a linear subspace, we have shown $E \subseteq \mcN$. 
	\end{proof}
	
	\begin{lemma}\label{lem: char_of_ideal_when_lin_part_non_deg}
		 $RE \subseteq R\mc{L}_{\bm{v}}E$.
	\end{lemma}
	\begin{proof}
		Write $J := R\mc{L}_{\bm{v}}E$. As $J$ is an ideal in $R$ it suffices to show $E \subseteq J$. We claim that $E^n \subseteq J + E^{n+1}$ for every $n \in \N_{\geq 0}$. Indeed, take $\psi \in E^n$. As $\mc{L}_{\bm{v}_{\mrm{l}}}$ is invertible on $E^n$ (which is true because $\mc{L}_{\bm{v}_{\mrm{l}}}$ is invertible on every finite-dimensional and $\mc{L}_{\bm{v}_{\mrm{l}}}$-invariant subspace $E \cap P^k(V) \subseteq E$), there exists some $\eta \in E^n$ s.t.\ $\mc{L}_{\bm{v}_{\mrm{l}}} \eta = \psi$. Observe that $\mc{L}_{\bm{v}_{\mrm{ho}}} E \subseteq E$ because $[\mc{L}_{\bm{v}_{\mrm{l,s}}}, \mc{L}_{\bm{v}_{\mrm{ho}}}] = 0$. Also $\mc{L}_{\bm{v}_{\mrm{ho}}} I^n \subseteq I^{n+1}$, since ${\bm{v}_{\mrm{ho}}} \in \X_{I^2}$. Thus $\mc{L}_{\bm{v}_{\mrm{ho}}} E^n \subseteq E^{n+1}$. In particular $\mc{L}_{\bm{v}_{\mrm{ho}}} \eta \in E^{n+1}$. Then $\psi =  \mc{L}_{\bm{v}_{\mrm{l}}} \eta =  \mc{L}_{\bm{v}}\eta - \mc{L}_{\bm{v}_{\mrm{ho}}}\eta \in J + E^{n+1}$, as required. By induction it follows that $E = E^0 \subseteq J + E^{n}$ for every $n \in \N$. As $\bigcap_{n \in \N}(J + E^{n}) = J$, this implies $E \subseteq J$.
	\end{proof}
	
	\begin{proof}[Proof of \Fref{thm: general_factorization_result}:]~\\
		Using \Fref{lem: center_subspace_in_kernel} we obtain $E \subseteq \mcN$. By \Fref{cor: vanishing_ideal_from_full_quadratic_form}, this implies $J\otimes \mfr{k} \subseteq \ker \olpi$, where $J = R\mc{L}_{\bm{v}}E$. By \Fref{lem: char_of_ideal_when_lin_part_non_deg}, we know $RE \subseteq J$. Notice that $E \cap V^\ast = V_{\mrm{c}}^\perp$, so in particular $R V_{\mrm{c}}^\perp \subseteq J$. Thus $RV_{\mrm{c}}^\perp \otimes \mfr{k} \subseteq \ker \olpi$. Notice that $R/(R V_{\mrm{c}}^\perp) \cong \R \llbracket V_{\mrm{c}}^\ast \rrbracket$, because $V_{\mrm{c}}^\ast = V^\ast/V_{\mrm{c}}^\perp$. We conclude that $\olpi$ factors through the quotient $(R \otimes \mfr{k})/(R V_{\mrm{c}}^\perp \otimes \mfr{k}) \cong (\R\llbracket V_{\mrm{c}}^\ast \rrbracket \otimes \mfr{k})$.
	\end{proof}

	\subsection{The Case of General $\p$}\label{sec: restr_mom_gen_case}
	
	\noindent
	Let us return to the case where $P$ is a $1$-connected finite-dimensional Lie group with Lie algebra $\p$. Let us recall some of the notation introduced earlier in \Fref{sec: restr_proj_reps}.\\
	
	\noindent
	Define $\sigma_0 := \ev_0 \circ \sigma : \p \to \mfr{k}$ and let ${\bm{v}_{\mrm{l}}} = j^1 \bm{v} : \p \to \gl(V)$ be the linearization of $\bm{v}$ at $0 \in V$. For $p \in \p$, the vector fields $\bm{v}(p)$ splits as $\bm{v}(p) = \vl{p} + \vho{p}$ for some formal vector field $\vho{p} \in \X_{I^2}$ vanishing up to first order at the origin. Let $\vl{p} = \vls{p} + \vln{p}$ be the Jordan decomposition of $\vl{p}$ over $\C$. Let $V_{\mrm{c}}^\C(p)$ denote the span in $V_\C$ of all generalized eigenspaces of $\vl{p}$ corresponding to eigenvalues with zero real part. Set $V_{\mrm{c}}(p) := V_{\mrm{c}}^\C(p) \cap V$. If $\mfr{C} \subseteq \p$ is a subset, define $V_{\mrm{c}}(\mfr{C}) := \bigcap_{p \in \mfr{C}} V_{\mrm{c}}(p)$. Let $V_{\mrm{c}}(\mfr{C})^\perp \subseteq V^\ast$ denote the annihilator of $V_{\mrm{c}}(\mfr{C})$ in $V^\ast$. Let $\Sigma_p \subseteq \C$ denote the additive subsemigroup of $\C$ generated by $\Spec(\vl{p})$. For any continuous projective unitary representation $\olpi$ of $\g\rtimes_D \p$, the set $\mfr{C}(\olpi)$ consists of all points $p \in \p$ for which $\olpi$ is of generalized positive  energy at $p$.\\
	
	\noindent
	We use \Fref{thm: general_factorization_result} combined with suitable normal form results to prove \Fref{thm: fact_result_spectral_condition} and \Fref{thm: fact_result_nform_condition}.
	
	\begin{lemma}\label{lem: ann_of_intersection}
		Let $W$ be a finite-dimensional real vector space and let $W_i \subseteq W$ be a collection of linear subspaces, where $i \in \mc{I}$ for some indexing set $\mc{I}$. Then $\big(\bigcap_{i \in \mc{I}} W_i\big)^\perp = \Span_{i \in \mc{I}}W_i^\perp$. 
	\end{lemma}
	\begin{proof}
		Notice first that $\bigcap_{i \in \mc{I}} W_i^\perp  = \big[\Span_{i \in \mc{I}} W_i\big]^\perp$. Applying this observation to the subspaces $W_i^\perp\subseteq W^\ast$, we obtain that $\bigcap_{i \in \mc{I}} W_i = \bigcap_{i \in \mc{I}} (W_i^\perp)^\perp = \big[\Span_{i \in \mc{I}} W_i^\perp\big]^\perp$, where we also used that $(W_i^{\perp})^{\perp} \cong W_i$ for any $i\in \mc{I}$. Taking annihilators, this implies $\big(\bigcap_{i \in \mc{I}} W_i\big)^\perp = \Span_{i \in \mc{I}}W_i^\perp$.
	\end{proof}
	
	\thmfactresultspectralcondition*
	\begin{proof}
		Let $p \in \mfr{C}$. By \Fref{thm: poincare_dulac}, there is a formal vector field $\who{p} \in \X_{I^2}$ satisfying $[\vls{p}, \who{p}] = 0$ s.t.\ $\bm{v}(p)$ is formally equivalent to $\bm{w}(p) := \vl{p} + \who{p}$. If $h \in \Aut(R)\subseteq \Aut(\g)$ is a formal diffeomorphism s.t.\ $\bm{w}(p) = h.\bm{v}(p)$, then $h$ leaves the constant part $\sigma_0(p)$ of $\sigma(p)$ fixed, so $\ev_0(h.\sigma(p)) = \ev_0\sigma(p) = \sigma_0(p)$. Thus, we may assume that $[\vls{p}, \vho{p}] = 0$ and $\Spec(\ad_{\sigma_0(p)}) \cap \Sigma_p = \emptyset$. By acting with gauge transformations, we may by \Fref{thm: normalformverticaltwistoned} further assume that $\sigma \in R \otimes \mfr{t}$, where $\mfr{t}$ is a maximal torus containing $\sigma_0(p)$. By \Fref{thm: general_factorization_result}, it follows that $RV_{\mrm{c}}(p)^\perp\subseteq \ker \olpi$. The above holds for all $p \in \mfr{C}$, so $\Span_{p \in \mfr{C}}RV_{\mrm{c}}(p)^\perp \subseteq \ker \olpi$. By \Fref{lem: ann_of_intersection} we know $\Span_{p \in \mfr{C}}\big(V_{\mrm{c}}(p)^\perp\big) =  V_{\mrm{c}}(\mfr{C})^\perp$, so that $R/(\Span_{p \in \mfr{C}}RV_{\mrm{c}}(p)^\perp) \cong \R\llbracket V_{\mrm{c}}(\mfr{C})^\ast \rrbracket$.
	\end{proof}

	\thmfactresultnformcondition*
	\begin{proof}
		By \Fref{thm: general_factorization_result} it follows that $\Span_{p \in \mfr{C}}RV_{\mrm{c}}(p)^\perp = RV_{\mrm{c}}(\mfr{C})^\perp \subseteq \ker \olpi$. 
	\end{proof}

	\subsection{The Case Where $\p$ is Simple.}\label{sec: restr_mom_set_semisimple}

	\noindent
	Let us consider the special case where $\p$ is simple. Let $P$ be a $1$-connected Lie group with $\Lie(P) = \p$. In this case, suitable normal form theorems for $\bm{v}$ and $\sigma$ are available (see \Fref{thm: herman_linearization} and \Fref{thm: normalformverticaltwistss}). In particular, $\bm{v} : \p \to \X_I^{\op}$ is always formally equivalent to its linearization $\bm{v}_{\mrm{l}}$ at $0 \in V$. Similarly, by \Fref{thm: normalformverticaltwistss}, the vertical twist $\sigma: \p \to \g$ is gauge-equivalent to some Lie algebra homomorphism $\sigma_0 : \p \to \mfr{k}$. In particular, if $\p$ is not compact then we may and do assume that $\sigma =  0$ by acting with gauge transformations if necessary, for in that case there are no homomorphisms $\p \to \mfr{k}$ (because $\mfr{k}$ is compact, see \Fref{lem: no_homs_noncompact_to_compact}). Combined with \Fref{thm: fact_result_nform_condition} we immediately obtain \Fref{thm: factorization_result_semisimple} below, where $V_{\mrm{c}}(\mfr{C}) := \bigcap_{p \in \mfr{C}}V_{\mrm{c}}(p)$.
	
	\thmfactresultss*
	
	\noindent
	Let $\p = \mfr{k}_0 \oplus \p_0$ be a Cartan decomposition of $\p$, so that $\mfr{k}_0$ and $\p_0$ are the $+1$ and $-1$ eigenspaces of a Cartan-involution $\theta$, respectively \cite[Cor.\ 6.18]{Knapp_lie_beyond}. Let $\mfr{a}_0 \subseteq \p_0$ be a maximal Abelian subalgebra of $\p_0$. According to the Iwasawa decomposition \cite[Prop.\ 6.43]{Knapp_lie_beyond}, $\p$ decomposes as $\p \cong \mfr{k}_0 \oplus \mfr{a}_0 \oplus \n_0$, where $\n_0 \subseteq \p$ is nilpotent. For $p \in \p$ we write $p = p_e + p_h + p_n$ for the corresponding decomposition of $p$, where $p_e \in \mfr{k}_0, p_h \in \mfr{a}_0$ and $p_n \in \n_0$. Then $\Spec(\ad_{p_e}) \subseteq i\R$, $\Spec(\ad_{p_h}) \subseteq \R$ and $\ad_{p_n}$ is nilpotent \cite[Lem.\ 6.45]{Knapp_lie_beyond}. Moreover, $\mfr{a}_0$ is contained in a Cartan subalgebra of $\p_0$ \cite[Cor.\ 6.47]{Knapp_lie_beyond}.

	\begin{proposition}\label{prop: cone_must_not_contain_hyperbolics}
		Suppose that $\p$ is simple and that the $\p$-representation $\bm{v}_{\mrm{l}}$ on $V$ is non-trivial and irreducible. Let $\mfr{C}\subseteq \p$ be an $\Ad_P$-invariant convex cone and let $V_{\mrm{c}}(\mfr{C}) := \bigcap_{p \in \mfr{C}}V_{\mrm{c}}(p)$. Assume that $\mfr{C}$ contains some non-zero $p_h \in \mfr{a}_0$. Then $V_{\mrm{c}}(\mfr{C}) = \{0\}$.
	\end{proposition}
	\begin{proof}
		 Notice first that as $P$ is $1$-connected, the $\p$-action $\bm{v}_{\mrm{l}} : \p \to \gl(V)$ integrates to a continuous representation of $P$ on $V$. As $\mfr{C}$ is $\Ad_P$-invariant, the subspace $V_{\mrm{c}}(\mfr{C})$ is $P$-invariant. Thus either $V_{\mrm{c}}(\mfr{C}) = \{0\}$ or $V_{\mrm{c}}(\mfr{C}) = V$, so it suffices to show $V_{\mrm{c}}(\mfr{C}) \neq V$. By assumption $p_h \neq 0$. In view of Cartan's unitary trick, see e.g.\ \cite[V. Prop.\ 5.3]{Knapp_repr_by_examples}, the image of elements in $\mfr{a}_0$ in any finite-dimensional representation are semisimple and have real spectrum. Thus $\Spec(\vl{p_h}) \subseteq \R$. As $\p$ is simple and $\bm{v}_{\mrm{l}}$ is a non-trivial $\p$-representation by assumption, it follows that $\bm{v}_{\mrm{l}}$ is injective. As $\vl{p_h}\in \gl(V)$ is semisimple, there exists $0\neq v \in V$ s.t.\ $\vl{p_h}v = \mu v$ for some $0 \neq \mu \in \R$. Thus $0\neq v \notin V_{\mrm{c}}(\mfr{C})$. Hence $V_{\mrm{c}}(\mfr{C}) \neq V$ and so $V_{\mrm{c}}(\mfr{C}) = \{0\}$.
	\end{proof}

	\thmsemisimplecone*
	
	\begin{remark}
		Notice that if $\p$ is simple and $\mcC$ is a closed $\Ad_P$-invariant convex cone which is not pointed, then $\mcC \cap -\mcC = \p$ and hence $\mcC = \p$. 
	\end{remark}
	
	\begin{proof}[Proof of \Fref{thm: semisimple_cone_must_be_pointed}:]
		The edge $\mfr{e} := \overline{\mfr{C}} \cap -\overline{\mfr{C}}$ of the closure $\overline{\mfr{C}}$ of $\mfr{C}$ is an ideal in $\p$. Assume that $\mfr{C}$ is not pointed. Then neither is $\overline{\mfr{C}}$. As $\p$ is simple, it follows that $\mfr{e} = \p$ and hence $\overline{\mfr{C}} = \p$. Thus $\mfr{C}$ is a dense convex cone in the finite-dimensional real vector space $\p$, which implies that $\mfr{C} = \p$. As $\p$ is non-compact, it contains some hyperbolic element. Thus, so does $\mfr{C}$. By \Fref{prop: cone_must_not_contain_hyperbolics} it follows that $V_{\mrm{c}}(\mfr{C}) = \{0\}$ and hence \Fref{thm: factorization_result_semisimple} implies that $\olpi$ factors through $\mfr{k}$. 
	\end{proof}	

	\noindent
	Thus if $\mc{C}$ is an $\Ad_P$-invariant convex cone which is not pointed, then $\g$ admits no continuous projective unitary representations which are of g.p.e.\ at $\mcC \subseteq \p$ other than those which factor through $\mfr{k}$. On the other hand, we know by \cite[Cor.\ 2.3]{Paneitz_invariant_cones} that if $\p$ is simple, then a non-trivial pointed closed and $P$-invariant convex cone exists in $\p$ if and only if $\p$ is of \textit{hermitian type}, meaning that $\dim(\z(\mfr{k}_0)) = 1$, where $\p = \mfr{k}_0 \oplus \p_0$ is a Cartan decomposition of $\p$ and where $\mfr{k}_0$ is the Lie algebra of a compact Lie group. \\

	\noindent
	Let us shift our attention to positive energy representations, in which case a different argument is available.

	\begin{lemma}\label{lem: non_compact_norm_continuous_trivial}
		Suppose that $P$ is a non-compact simple connected Lie group. If $(\sigma, \mcH_\sigma)$ is a unitary $P$-representation that is norm-continuous, then $\sigma$ is trivial.
	\end{lemma}
	\begin{proof}
		As $\p$ is simple, $d\sigma$ is either injective or trivial. Assume that $d\sigma$ is not trivial. Let $\p = \mfr{k}_0 \oplus \mfr{a}_0 \oplus \n_0$ be the Iwasawa decomposition of $\p$. Take $x \in \mfr{a}_0$. Then $\ad_x$ is semisimple and $\Spec(\ad_x) \subseteq \R$. As $\sigma$ is unitary and $d\sigma$ is injective, $z \mapsto \norm{d\sigma(z)}_{\B(\mcH)}$ defines a $P$-invariant norm on $\p$. With respect to this norm, $e^{t \ad_x}$ is an isometry on $\p$ for every $t \in \R$. As $\ad_x$ is semisimple, it follows that $\Spec(\ad_x) \subseteq i\R$. So $\Spec(\ad_x) \subseteq \R \cap i\R = \{0\}$ and hence $\ad_x = 0$. Since $\p$ has trivial center it follows that $x = 0$. So $\mfr{a}_0 = \{0\}$ and hence $\p$ is compact. But $P$ is non-compact by assumption. So $d\sigma$ must be trivial. As $P$ is connected, it follows that $\sigma$ is trivial.
	\end{proof}

	\begin{proposition}\label{prop: non_compact_simple_full_cone_no_pe_reps}
		Suppose that $P$ is a non-compact $1$-connected simple Lie group. Assume that the $P$-action on $V$ is irreducible and non-trivial. Let $\overline{\rho}$ be a continuous projective unitary representation of $G$ which is of positive energy at $\mcC := \p$. Then $\restr{\overline{\rho}}{G}$ factors through $K$. 
	\end{proposition}
	\begin{proof}
		 By \Fref{thm: always_factor_through_fin_jet} it suffices to consider the case where $\overline{\rho}$ factors through $G_k$ for some $k\in \N$. From Whitehead's Second Lemma, \cite[III.9.\ Lem.\ 6]{Jacobson}, we know that $H^2(\p, \R) = \{0\}$. Using in addition that $P$ is simply connected, it follows that $\restr{\overline{\rho}}{P}$ lifts to a continuous unitary representation $\sigma : P \to \U(\mcH_\rho)$ of $P$, so that $\overline{\rho}(p) = [\sigma(p)]$ in $\PU(\mcH_\rho)$ for all $p \in P$. By \Fref{lem: full_cone_then_norm_continuous}, the fact that $\restr{\overline{\rho}}{P}$ is of p.e.\ at $\mcC = \p$ implies that $\sigma$ is norm-continuous. It follows from \Fref{lem: non_compact_norm_continuous_trivial} that $\sigma$ is trivial. Thus $\overline{\rho}(\alpha_p(g)) = \overline{\rho}(g)$ for all $g \in G$ and $p \in P$. It follows that $d\overline{\rho}$ vanishes on $D(\p)\g$. As $\p$ acts irreducibly and non-trivially on $V$, it follows that the ideal in $\g$ generated by $D(\p)\g$ is $I\otimes \mfr{k}$. Thus $I \otimes \mfr{k} \subseteq \ker d\overline{\rho}$. This implies that $\restr{\overline{\rho}}{G_k}$ factors through $K$.
	\end{proof}

	\noindent
	The following provides a simple example of a projective p.e.\ representation $\overline{\rho}$ of $G_1 \rtimes P$ s.t.\ $\restr{\overline{\rho}}{G_1}$ does not factor through $K$.

	\begin{example}
	Let $P = \Mp(2, \R)$ be the double-cover of $\SL(2, \R)$. Let $P$ act on $V:= \R^2$ via the defining action of $\SL(2,\R)$. We consider a trivial vertical twist, so that the $\p$-action on $\g = R \otimes \mfr{k}$ is given by $D(p) = -\mc{L}_{\bm{v}(p)}$. In this case the generator $p_0$ of rotations generates the unique (up to a sign) non-trivial closed, pointed, $P$-invariant convex cone $\mcC$ in $\p$. Explicitly, $\bm{v}(p_0) = y\partial_x - x\partial_y$. Let us construct a non-trivial continuous projective unitary representation of $G_1 \rtimes P \cong (V^\ast \otimes \mfr{k}) \rtimes (K \times P)$ that is of p.e.\ at the cone $\mcC \subseteq \p$. Write $W := V^\ast \otimes \mfr{k}$. \\
			
	\noindent
	We begin by specifying a suitable $2$-cocycle on $V^\ast \otimes \mfr{k} \subseteq \g_1$. Notice that $\big(\bigwedge^{2} V\big)^\p \cong \R$ is one-dimensional. Let $0\neq \lambda \in \big(\bigwedge^{2}V\big)^\p$ and consider it as a $\p$-invariant bilinear map $V^\ast \times V^\ast \to \R$. To be consistent with \Fref{prop: characterization_of_cts_second_cohom}, let us write $\lambda(fdg)$ instead of $\lambda(f,g)$ for $f,g \in V^\ast$. Let $x,y \in V^\ast$ be the usual basis of $V^\ast$. Then $\lambda$ is fully specified by the number $\lambda(ydx)$. If $\lambda(ydx) > 0$, then the quadratic form $q(v) := \lambda(\mc{L}_{\bm{v}(p_0)}v dv)$ is positive-definite, because $q(ax+by) = (a^2+b^2)\lambda(ydx)$ for $a,b \in \R$. Let $\omega$ be the unique symplectic form on $W$ satisfying $\omega(vX, wY) := \lambda(vdw)\kappa(X,Y)$ for $X,Y \in \mfr{k}$ and $v,w \in V^\ast$. Then $\omega(D(p_0)\xi, \xi) \geq 0$ for every $\xi \in W$ (recalling that $\kappa$ is \textit{negative} definite). Let $\mrm{H}(W, \omega)$ be the corresponding Heisenberg group. Let $L_\pm$ be the $\pm i$-eigenspaces in $W_\C$ of the complex structure $\mc{J} := D(p_0)$ on $W_\C$, so that $W_\C = L_- \oplus L_+$. The $\p$-invariance of $\lambda$ ensures that $\mc{J}^\ast \omega = \omega$. Indeed, extend $\omega$ $\C$-bilinearly to $W_\C$. As $\lambda$ is $\p$-invariant, it follows that $\omega(\mc{J}\xi, \eta) + \omega(\xi, \mc{J}\eta) = 0$ for all $\xi, \eta \in W_\C$, which implies that $L_\pm \subseteq W_\C$ are $\mc{J}$-invariant Lagrangian subspaces for $\omega$. Then $\mc{J}^\ast \omega = \omega$ follows from $\mc{J}^\ast \omega(w_+, w_-) = \omega(iw_+, -iw_-) = \omega(w_+, w_-)$ for $w_\pm \in L_\pm$. Notice further that $\omega(\mc{J}\xi, \xi) \geq 0$ holds for all $\xi \in W$, by construction. Equip $L_+$ with the positive definite hermitian form defined by $\langle v, w\rangle_{L_+} := -2i \omega(\overline{v}, w)$ for $v,w \in L_+$. For each $n \in \N$, equip the symmetric algebra $S^n(L_+)$ with the inner product satisfying
	\[ \langle v_1\cdots v_n, w_1\cdots w_n\rangle := \sum_{\sigma \in S_n}\prod_{k=1}^n \langle v_{\sigma_k}, w_k\rangle_{L_+}, \qquad v_k, w_k \in L_+.\]
	Let $\F := \overline{S^\bullet(L_+)}$ be the Hilbert space completion, where $S^\bullet(L_+) = \bigoplus_{n=0}^\infty S^n(L_+)$. The metaplectic representation $\rho$ of $\mrm{H}(W, \omega) \rtimes \Mp(W, \omega)$, with $\rho(z) = zI$ on the central $\T$ component, can be realized on the Fock space $\F$, where $\Mp(W, \omega)$ denotes the metaplectic group \cite[Thm\ X.3.3]{Neeb_book_hol_conv}. Notice that $\SL(2,\R) \hookrightarrow \Sp(W, \omega)$ because $\lambda$ is $\p$-invariant. By pulling back the metaplectic representation we obtain a continuous unitary representation of $\mrm{H}(W, \omega) \rtimes P$ which is of p.e.\ at $\mc{C}$ and does not factor through $K$.
	\end{example}

	\part*{Appendix}
	\appendix
	\section{From Germs to Jets}\label{sec: germs_to_jets}
	
	\noindent
	Let $\mc{K} \to M$ be locally trivial bundle of Lie groups with typical fiber a finite-dimensional Lie group $G$ with Lie algebra $\g$. Write $\mfr{K} \to M$ for the corresponding Lie algebra bundle. The following justifies the claim made in \Fref{sec: introduction} that any continuous projective unitary representation of $\Gamma_{\mrm{c}}(\mc{K})$ which factors through the germs at a point $a \in M$ actually factors through the $\infty$-jets $J^\infty_a(\mc{K})$ at $a \in M$. The group $\Gamma_{\mrm{c}}(\mc{K})$ is a locally exponential Lie group modeled on the LF-Lie algebra $\Gamma_{\mrm{c}}(\mfr{K})$ \cite[Prop.\ 2.3]{BasNeeb_cov_ce_gauge}.\\

	\noindent
	Let $U \subseteq \R^d$ be an open neighborhood of the origin. Let $C^\infty_{\mrm{flat}}(U)$ denote the kernel of the $\infty$-jet projection 
	\[j^\infty_0 : C_c^\infty(U) \to J_0^\infty(C^\infty_{\mrm{c}}(U)) \cong \R \llbracket x_1, \cdots, x_d \rrbracket\]
	at $0 \in U$. In the following we show the known fact that the closure $C^\infty_{\mrm{c}}(U\setminus \{0\})$ in $C^\infty(U)$ is $C^\infty_{\mrm{flat}}(U)$. As a consequence, we deduce that if a continuous projective unitary representation of the Lie algebra $\Gamma_{\mrm{c}}(\mfr{K})$ factors through the germs at a point $a \in M$, then it factors through the $\infty$-jets $J^\infty_a(\mfr{K})$ at $a \in M$. In turn, this implies a group level-analogue. \\

	\noindent
	If $K \subset U$ is a compact set, we write $C^\infty_K(U)$ for the subspace of $C^\infty(U)$ consisting of functions on $U$ with support in $K$. Then $C^\infty_K(U)$ is the projective limit $C^\infty_K(U) = \varprojlim_n C_K^n(U)$ of the Banach spaces $C_K^n(U)$, which we equip with the norm $\norm{f}_{C_K^n(U)} := \sup_{|k|\leq n} \norm{D^k f}_{C_K(U)}$, where the supremum runs over all multi-indices $k \in \N_{\geq 0}^d$ with $|k| \leq n$. Then $C^\infty_{\mrm{c}}(U) := \varinjlim C^\infty_K(U)$ is the corresponding locally convex inductive limit. See e.g.\ \cite[Thm.\ 6.5]{Rudin_FA} for a description of this topology. For $r > 0$, write $B_r := \set{x \in \R^d \st \norm{x} \leq r}$ for the closed ball centered at $0 \in \R^d$ with radius $r$.

	\noindent
	\begin{lemma}\label{lem: closure of g_c_is_g_flat}
		The closure of $C^\infty_{\mrm{c}}(U\setminus\{0\})$ in $C_c^\infty(U)$ is $C^\infty_{\mrm{flat}}(U)$. 
	\end{lemma}
	\begin{proof}
		As $C^\infty_{\mrm{flat}}(U) \subseteq C_c^\infty(U)$ is closed and $C^\infty_{\mrm{c}}(U\setminus\{0\}) \subseteq C^\infty_{\mrm{flat}}(U)$, it follows that $\overline{C^\infty_{\mrm{c}}(U\setminus\{0\})} \subseteq C^\infty_{\mrm{flat}}(U)$. It remains to show the reverse inclusion. Let $f \in C^\infty_{\mrm{flat}}(U)\subseteq C_c^\infty(U)$. We show $f \in \overline{C^\infty_{\mrm{c}}(U\setminus\{0\})}$. Let $K_0 \subseteq M$ be a relatively compact open subset such that $\supp f \subseteq K_0$. Set $K := \overline{K_0}$. We may assume that $0 \in K_0$, for otherwise $f \in C^\infty_{\mrm{c}}(U\setminus\{0\})$ and we are done. By \cite[Lem.\ I.4.2]{Malgrange_ideals_of_diff_functions}, we can find constants $C_k > 0$ for $k \in \N^d_{\geq 0}$, depending only on $k$, such that for any $0 < r < 1$ with $B_{2r} \subseteq K_0$, there exists a smooth function $\psi_r \in C^\infty(\R^d)$ s.t.\ $\psi_r \geq 0$, $\restr{\psi_r}{B_r} = 0$, $\restr{\psi_r}{(\R^d\setminus B_{2r})} = 1$ and $\sup_{x \in \R^d}|D^k\psi_r(x)| \leq C_k r^{-|k|}$ for every $k \in \N^{d}_{\geq 0}$. In particular $f \psi_r \in C^\infty_{\mrm{c}}(U\setminus\{0\})$ and $\supp f \psi_r \subseteq K$. Moreover, observe that $\supp (1-\psi_r) \subseteq B_{2r}$ and $\norm{(1-\psi_r)}_{C^n_{B_{2r}}(\R^d)} \lesssim r^{-n}$ for some constant depending on $n \in \N_{\geq 0}$, where we used that $0<r<1$. On the other hand, suppose that $\alpha \in \N_{\geq 0}^d$ is a multi-index. Since $j^\infty_0(D^\alpha f) = 0$, it follows from Taylor's Theorem that $\norm{D^\alpha f}_{C(B_{2r})} \lesssim r^{l}$ for arbitrary $l \in \N_{\geq 0}$, with a constant depending on $f$, $\alpha$ and $l$ but not on $r$. Thus $\norm{f}_{C^n(B_{2r})} \lesssim r^{l}$ for arbitrary $n,l \in \N_{\geq 0}$. In particular $\norm{f}_{C^n(B_{2r})} \lesssim r^{n+1}$. Combining the previous observations, we obtain that
		\[\norm{f - f \psi_r}_{C^n(K)} = \norm{f(1-\psi_r)}_{C^n(K)} = \norm{f(1-\psi_r)}_{C^n(B_{2r})} \lesssim \norm{f}_{C^n(B_{2r})} \norm{(1-\psi_r)}_{C^n(B_{2r})} \lesssim r,\]
		the constants depending only on $f$ and $n$ but not on $r$. This shows that $f \psi_r \to f$ in $C^\infty_K(U)$ as $r\to 0$. Thus $f \psi_r \to f$ in $C^\infty_{\mrm{c}}(U)$. Since $\psi_r f \in C^\infty_{\mrm{c}}(U\setminus\{0\})$ for every $r$, we conclude that $f \in \overline{C^\infty_{\mrm{c}}(U\setminus\{0\})}$.
	\end{proof}

	\noindent
	If $a \in M$, define the spaces of smooth section of $\mc{K}$ and $\mfr{K}$ which are flat at $a \in M$:
	\begin{align*}
		\Gamma_{\mrm{flat}(a)}(\mc{K}) := \ker\bigg(j^\infty_a : \Gamma_{\mrm{c}}(\mc{K}) \to J^\infty_a(\mc{K})\bigg),\qquad
		\Gamma_{\mrm{flat}(a)}(\mfr{K}) := \ker\bigg(j^\infty_a : \Gamma_{\mrm{c}}(\mfr{K}) \to J^\infty_a(\mfr{K})\bigg).
	\end{align*}

	\noindent
	\Fref{prop: two_tops_on_inf_jets_agree} below clarifies the apparent ambiguity in the topology on $J^\infty_a(\mfr{K})$, for which two candidates are available.
	\begin{proposition}\label{prop: two_tops_on_inf_jets_agree}
		Let $a \in M$. The projective limit topology on $J_a^\infty(\mfr{K}):= \varprojlim_kJ^k(\mfr{K})$ coincides with the quotient topology obtained from $J_a^\infty(\mfr{K}) \cong \Gamma_{\mrm{c}}(\mfr{K})/\Gamma_{\mrm{flat}(a)}(\mfr{K})$.
	\end{proposition}
	\begin{proof}
		The continuous $k$-jet projections $j_a^k : \Gamma_{\mrm{c}}(\mfr{K}) \to J_a^k(\mfr{K})$ at $a \in M$ all descend to continuous maps $\Gamma_{\mrm{c}}(\mfr{K})/\Gamma_{\mrm{flat}(a)}(\mfr{K}) \to J_a^k(\mfr{K})$. By the universal property of the projective limit, they induce a continuous map $\Phi: \Gamma_{\mrm{c}}(\mfr{K})/\Gamma_{\mrm{flat}(a)}(\mfr{K}) \to J_a^\infty(\mfr{K})$. Using Borel's Lemma \cite[Thm.\ 1.2.6]{Hormander_I}, it is not hard to check that this map is bijective. It remains to show it is an open map, which follows immediately from the Open Mapping Theorem \cite[Cor.\ 2.12]{Rudin_FA} because $\Gamma_{\mrm{c}}(\mfr{K})/\Gamma_{\mrm{flat}(a)}(\mfr{K})$ and $J^\infty_a(\mfr{K})$ are both Fr\'echet spaces and $\Phi$ is bijective and continuous.
	\end{proof}
	
	\begin{proposition}\label{prop: closure_flat}
		Let $a \in M$. 
		\begin{itemize}
			\item The closure of $\Gamma_{\mrm{c}}(M\setminus\{a\}; \mfr{K})$ in $\Gamma_{\mrm{c}}(M; \mfr{K})$ is $\Gamma_{\mrm{flat}(a)}(\mfr{K})$.
			\item The closure of $\Gamma_{\mrm{c}}(M\setminus\{a\}; \mc{K})$ in $\Gamma_{\mrm{c}}(M; \mc{K})$ is $\Gamma_{\mrm{flat}(a)}(\mc{K})$.
		\end{itemize}
	\end{proposition}
	\begin{proof}
		By a partition of unity argument, we may assume that the bundle $\mfr{K} \to M$ is trivial, that $M \subseteq \R^d$ is open neighborhood of $0 \in \R^d$ and that $a = 0$. Then $\Gamma_{\mrm{c}}(M; \mfr{K}) \cong C^\infty_{\mrm{c}}(M; \mfr{k})$. The claim now follows from \Fref{lem: closure of g_c_is_g_flat}. Notice for the second assertion that $\Gamma_{\mrm{flat}(a)}(M; \mc{K})$ is a locally exponential, being an embedded closed Lie subgroup of the locally exponential Lie group $\Gamma_{\mrm{c}}(M; \mc{K})$. The result is then immediate from the previous point.
	\end{proof}
	
	\begin{restatable}{proposition}{germstojets}\label{prop: germs_to_jets}
		Let $a \in M$.
		\begin{enumerate}
			\item Let $\olpi : \Gamma_{\mrm{c}}(M; \mfr{K}) \to \mc{L}^\dagger(\mcD)$ be a continuous projective unitary representation on the pre-Hilbert space $\mcD$.  Assume that $\olpi$ vanishes on $\Gamma_{\mrm{c}}(M \setminus\{a\}; \mfr{K})$. Then $\olpi$ factors continuously through $J^\infty_a(\mfr{K})$. 
			\item Let $\overline{\rho} : \Gamma_{\mrm{c}}(M; \mc{K}) \to \PU(\mcH)$ be a continuous projective unitary representation of $\Gamma_{\mrm{c}}(M; \mc{K})$. Assume that $\overline{\rho}$ vanishes on $\Gamma_{\mrm{c}}(M \setminus \{a\}; \mc{K})$. Then $\overline{\rho}$ factors through $\Gamma_{\mrm{flat}(a)}(M; \mc{K})$.
		\end{enumerate}
	\end{restatable}
	\begin{proof}
		For the first point, notice by continuity that $\olpi$ must also vanish on the closure of $\Gamma_{\mrm{c}}(M \setminus\{a\}; \mfr{K})$ in $\Gamma_{\mrm{c}}(M;\mfr{K})$, which by \Fref{prop: closure_flat} equals $\Gamma_{\mrm{flat}(a)}(M;\mfr{K})$. Thus $\Gamma_{\mrm{c}}(M;\mfr{K})$ factors continuously through the quotient space $J^\infty_a(\mfr{K}) \cong \Gamma_{\mrm{c}}(M;\mfr{K})/\Gamma_{\mrm{flat}(a)}(M;\mfr{K})$, where \Fref{prop: two_tops_on_inf_jets_agree} was used. The second point is proven similarly using \Fref{prop: closure_flat}.
	\end{proof}

	\newcommand{\etalchar}[1]{$^{#1}$}

%

\end{document}